\documentclass[11pt]{article}
\usepackage{amsmath}
\usepackage[thmmarks,amsmath,amsthm]{ntheorem}
\usepackage{comment}

\usepackage{amsfonts, psfrag, graphicx, amssymb,enumitem,indentfirst,float,geometry,color,enumerate,graphicx,subfigure,algorithm,algpseudocode,pifont,mathtools,gensymb}
\usepackage{pbox}
\usepackage{booktabs, cellspace, hhline}
\allowdisplaybreaks[4]
\numberwithin{equation}{section}
\numberwithin{figure}{section}

\usepackage{etoolbox}
\patchcmd{\thebibliography}{\chapter*}{\section*}{}{}
\usepackage[title]{appendix}

\usepackage{lipsum}

\newtheorem{definition}{Definition}[section]
\newtheorem{thm}{Theorem}[section]
\newtheorem{lemma}{Lemma}[section]
\newtheorem{rem}{Remark}[section]

\newcommand{\commentout}[1]{{}} 

\newcommand{\abs}[1]{\left|#1\right|}
\newcommand{\norm}[1]{\left\|#1\right\|}

\newcommand{\bfu}{{\bf u}}

\newcommand{\bfv}{{\bf v}}

\newcommand{\bfalpha}{{\boldsymbol \alpha}}

\newcommand{\vertiii}[1]{{\left\vert\kern-0.25ex\left\vert\kern-0.25ex\left\vert #1
    \right\vert\kern-0.25ex\right\vert\kern-0.25ex\right\vert}}

\begin{document}
\title{An Enriched Immersed Finite Element Method\\
for Interface Problems with Nonhomogeneous Jump Conditions}
\author{
Slimane Adjerid \thanks{Department of Mathematics, Virginia Tech, Blacksburg, VA 24061 (adjerids@math.vt.edu). This author was supported by J.T. Oden Faculty grant, October 2019 }\\
\and Ivo Babuska \thanks{ICES, University of Texas at Austin, TX, United States (babuska@ices.utexas.edu)} \\
\and Ruchi Guo \thanks{Department of Mathematics, The Ohio State University, Columbus, OH 43201 (guo.1778@osu.edu) } \\
\and Tao Lin \thanks{Department of Mathematics, Virginia Tech, Blacksburg, VA 24061 (tlin@vt.edu) }
  }
\date{}
\maketitle
\begin{abstract}
This article presents and analyzes a $p^{th}$-degree immersed finite element (IFE) method for elliptic interface problems with nonhomogeneous jump conditions. In this method, jump conditions are approximated optimally by basic IFE and enrichment IFE piecewise polynomial functions constructed by solving local Cauchy problems on interface elements. The proposed IFE method is based on a discontinuous Galerkin formulation on interface elements and a continuous Galerkin formulation on non-interface elements. This $p^{th}$-degree IFE method is proved to converge optimally under mesh refinement. In addition, this article addresses the stability of this IFE method and has established upper bounds for its condition numbers which are optimal with respect to the mesh size but suboptimal with respect to the contrast of the discontinuous coefficient.
\end{abstract}

\section{Introduction}
\label{sec:introduction}

Let $\Omega$ be an open bounded domain in $\mathbb{R}^2$ with a smooth boundary $\partial \Omega$ and let $\Gamma \subset \Omega$ be a $C^{p+1}$ simple closed curve that splits $\Omega$ into two sub-domains $\Omega^-$ and $\Omega^+$ such that $\overline{\Omega} = \overline{ \Omega^- \cup \Omega^+}$. Further we let $H^{k}({\Omega}^s)$ be the standard Sobolev spaces on the sub-domains $\Omega^s$ equipped with  the norm $\|\cdot\|_{H^{k}({\Omega}^s)}$ and semi-norm $|\cdot |_{H^{k}({\Omega^s})}$, $k\geqslant0$, $s=\pm$. We also need the following broken
Sobolev spaces $PH^{k}({\Omega}) = H^{k}({\Omega}^+) \bigoplus  H^{k}({\Omega}^-)$
which can be written as
\begin{equation}
PH^{k}({\Omega}) = \Big\{  v=(v^+,v^-)   ~:~  v^+\in H^{k}({\Omega}^{+})~ and ~ v^-\in H^{k}({\Omega}^{-})\Big\},  \label{Hil_sp_int_0}
\end{equation}
equipped with the norm and semi-norm
\begin{equation}
 \|  v \|_{ PH^{k}({\Omega})} = \|  v^+ \|_{ H^{k}({\Omega}^+)} + \|  v^- \|_{ H^{k}({\Omega}^-)}, ~~ \text{and} ~~ |  v |_{ PH^{k}({\Omega})} = |  v^- |_{ H^{k}({\Omega}^-)} + |  v^+ |_{ H^{k}({\Omega}^+)}.  \label{sp_int_norm}
\end{equation}
In particular, we consider the subspaces $PH^{k}_0({\Omega})$ consisting of all functions in $PH^{k}(\Omega)$ with zero trace
on $\partial{\Omega}$ in the sense of fractional Sobolev spaces $H^{\frac{1}{2}}$.

In this manuscript, we assume the data functions are such that $f \in PH^{p-1}(\Omega)$, $J_D \in H^{p+\frac{1}{2}}(\Gamma)$ and
$J_N \in PH^{p-\frac{1}{2}} (\Gamma)$ for some integer $p\geqslant1$ and a piecewise constant function
\begin{equation*}
\beta(X)=
\left\{\begin{array}{cc}
\beta^+ & \text{for} \; X\in \Omega^+, \\
\beta^- & \text{for} \; X\in \Omega^-,
\end{array}\right.
\end{equation*}
or $\beta=(\beta^+,\beta^-)$, with $\beta^-\geqslant\beta^+ >0$ and denote the contrast by $\rho =\beta^+/\beta^-$.
From now on we let $u \in PH^{p+1}_0(\Omega)$ be the solution of the following elliptic interface problem
\begin{subequations}
\label{model}
\begin{align}
 -\nabla\cdot(\beta\nabla u)=f, \;\;\;\; & \text{in} \; \Omega = \Omega^-  \cup \Omega^+,  \label{inter_PDE} \\
 u=0, \;\;\;\; &\text{on} \; \partial\Omega.   \label{b_condition}
\end{align}
In addition, we close the interface problem by enforcing the following jump conditions:
\begin{align}
&  [u]_{\Gamma} :=u^-|_{\Gamma} - u^+|_{\Gamma}=  J_D ~~~ \text{in} ~~~ H^{\frac{1}{2}}(\Gamma),  \label{nonhomo_jump_cond_1} \\
&  \big[\beta \nabla u\cdot \mathbf{n}\big]_{\Gamma} :=\beta^- \nabla u^-\cdot \mathbf{n}|_{\Gamma} - \beta^+ \nabla u^+\cdot \mathbf{n}|_{\Gamma} = J_N ~~~ \text{in} ~~~ H^{-\frac{1}{2}}(\Gamma), \label{nonhomo_jump_cond_2}
\end{align}
\end{subequations}
where $\mathbf{n}$ is a unit vector normal to the interface $\Gamma$.

Interface problems with homogeneous interface jump conditions such that $J_D=0$, $J_N=0$ appear in many problems in science and engineering, such as in electrical impedance tomography \cite{2013BelhachmiMeftahi,2005HolderDavid}, electroencephalography \cite{2010VallaghePapadopoulo}, plasma simulations \cite{1991BirdsallLangdon,1988HockneyEastwood} and Poisson-Boltzmann equations \cite{2007ChenHolstXu,2015YingXie}. Nonhomogeneous interface conditions are also used in many models.
For example, interface jump conditions such that $J_D=0$, $J_N\neq0$ are seen in \emph{(i)} potential problems where the surface charge density for electric potential is not zero on an interface separating two isotropic media \cite{1975Cook}, \emph{(ii)} flow in a domain consisting of two stratified porous media with a source at the interface \cite{2006Miyazaki}, and \emph{(iii)} Burton-Cabrera-Frank-type models for epitaxial growth of thin films \cite{2004BanschHauberLakkisLiVoigt,2003CaflischLi}.
By contrast, the interface conditions with $J_D\neq0$, $J_N=0$, respectively, are used  in Hele-Shaw flow \cite{1997HouLiOsherZhao} to model the Laplace-Young and the kinematic jump conditions. Moreover, the nonhomogeneous interface conditions with $J_D\neq0$, $J_N\neq0$
are used in: \emph{(i)} Navier-Stokes equations \cite{2003LeeLeVeque,1997LevequeLi,2001LiLai}
to model a discontinuous pressure due to surface tension and singular force at the interface, and \emph{(ii)} shape optimization methods \cite{2007AfraitesDambrine}.


Standard finite element methods on fitted meshes \cite{1970Babuska,1996BrambleKing,1982Xu} have been applied to solve interface problems, however, they may be inefficient for problems involving moving interfaces since the mesh has to be updated to resolve the evolving interface. In order to circumvent this difficulty scientists developed several numerical methods on unfitted meshes
that may be more efficient in solving interface problems
with moving interfaces, see the discussions of the advantages in \cite{2007AfraitesDambrine,2003CaflischLi,GuoLinLin2017,1997HouLiOsherZhao,2003LeeLeVeque,1997LevequeLi,2001LiLai}. The idea of numerical methods on unfitted meshes has attracted the interest of both the
finite difference \cite{1994LevequeLi,2006LiIto} and finite element communities including \cite{2015BurmanClaus,2002HansboHansbo,Lehrenfeld2017, 2012MASSJUNG,2016WangXiaoXu} for the cut finite element method (Cut-FEM), \cite{2010ChuGrahamHou,2009EfendievHou} for the multi-scale FEM method, and \cite{1996BabuskaMelenk,2004SukumarHuang} for the partition of unity method (PU-FEM), as well as the immersed finite element (IFE) method discussed in this article.

The basic idea of the IFE approach is, in spirit, similar to Hsieh-Clough-Tocher type macro polynomials \cite{2001Braess,1966CloughTocher} where we construct piecewise polynomial functions on interface elements to capture the interface jump behavior of the exact solution. Early  construction approaches of IFE functions were based on the piecewise linear approximation of the interface curve and jump conditions, see \cite{2007GongLiLi,2016GuoLin,2016GuoLinZhang, 2008HeLinLin,2011HeLinLin,2004LiLinLinRogers,2003LiLinWu,2001LinLinRogersRyan}. Recently, higher-degree IFE methods were proposed in \cite{2014AdjeridBenromdhaneLin,2018AdjeridRomdhaneLin,2016AdjeridGuoLin,2016GuzmanSanchezSarkis,2018ZhuangGuo} for homogeneous interface conditions combined with the following extended homogeneous jump conditions
\begin{equation}
\label{normal_jump_cond}
 \left[ \beta\frac{\partial^{j-2} \triangle u}{\partial\mathbf{ n}^{j-2}} \right]_{\Gamma} =0, ~~~ j=2,3,\ldots ,p, ~~ \text{in} ~~ H^{-j-\frac{1}{2}}(\Gamma),
\end{equation}
for $p\geqslant2$ which hold when both the interface and the source term $f$ are sufficiently smooth.
However, this is not the case in many problems such as the Poisson-Boltzmann equations \cite{2007ChenHolstXu,2015YingXie}.

Inspired by the generalized finite element method (GFEM) methodology of Babu\v{s}ka and collaborators \cite{1994BabuskaCalozOsborn,1983BabuskaOsborn}, we propose to construct
IFE functions on interface elements by solving local problems.
In the GFEM methodology the standard finite element spaces are enriched
to reflect the local solution behavior emanating from material
interfaces and other singularities. Recently,
Babu\v{s}ka and his collaborators discuss the stability of
GFEM method applied to interface problems with smooth and non-smooth interfaces in a series of papers
\cite{2012Babuska-Banarjee, 2017Babuska-Barnerjee-Kergrene,2016Kergrene-Babuska-Banerjee,2019Zhang-Banerjee-Babuska}

The local problems, as suggested by the regularity analysis for interface problems \cite{1998ChenZou,2010ChuGrahamHou}, result from the decomposition of the solution into a homogeneous component and nonhomogeneous components associated with nonhomogeneous interface jumps and/or a discontinuous source term. Specifically, these components can be interpreted as certain solutions to local Cauchy problems subjected to boundary conditions on the interface induced from the jumps at the interface. In particular,
on each interface element, cut by the interface into two subelements, an IFE function is defined as a discrete biharmonic extension of a $p^{th}$-degree polynomial from one subelement to the whole interface element by
solving the local Cauchy problems. This extension establishes a mapping on polynomial spaces, denoted as the \textit{Cauchy mapping} due to the origin from Cauchy problems. Two sets of IFE functions are introduced, the first set forms the IFE space and is used to approximate the homogeneous component of the solution to the interface problem while the second set of special IFE functions is used to approximate the nonhomogeneous components of the solution. These special IFE functions will be referred to as enrichment IFE functions.

The proposed IFE functions are then employed in a discontinuous Galerkin IFE scheme to solve the elliptic interface problems on unfitted meshes \cite{2007GongTHESIS,2010GongLi,2007GongLiLi,2009HeTHESIS,2011HeLinLin}.
In this scheme, those enrichment IFE functions will help homogenize the nonhomogeneous interface jump conditions such that they are transformed into problems with homogenous jump conditions in a weak sense.
A similar idea is also used with extended finite element methods (X-FEM) \cite{2000DolbowMoesBelytschko,2001DolbowMoesBelytschko,1999MoesDolbowBelytschko,2006VaughanSmithChopp}. Furthermore, a major advantage of the proposed IFE method is that
the enrichment IFE functions are determined \emph{a priori} by solving local Cauchy problems directly from
the known jump data, and they are moved to the right hand side
of the proposed IFE method. Consequently, using enrichment IFE functions does not result in additional degrees of freedom so that the degrees of freedom for the proposed IFE method
is the same as that for the interface problem with homogeneous jump conditions.
We note that the proposed enrichment IFE functions induced by nonhomogeneous jump conditions are piecewise polynomials on each interface element while those used with the XFEM \cite{2006VaughanSmithChopp} may be non polynomials. However, although this feature makes the proposed method closer to the standard finite element method which may both speed-up computations and make error analysis possible, there may be difficulties for the proposed IFE method to resolve solution singularities caused by, for instance, corners in the interface
\cite{2019Zhang-Banerjee-Babuska}.


A major difficulty in the analysis of IFE methods comes from the insufficient regularity in both the involved macro polynomials and the exact solution. Due to the lack of smoothness across the interface, neither the
scaling argument used to establish \emph{a priori} error estimates for the standard finite
element method nor other known techniques in the literature of unfitted mesh methods such as adding penalty terms are directly applicable.
The \textit{Cauchy mapping} on polynomial spaces newly developed for constructing IFE functions by solving local Cauchy problems turns out also to be a critical tool for our theoretical analysis of the proposed IFE spaces. All key intermediate results, such as the existence of
IFE functions, their approximation capabilities, and the trace/inverse inequalities, follow from properties of the \textit{Cauchy mapping}. These results help us establish
optimal convergence and stability results with respect to the mesh size and
polynomial degree. The upper bounds for the condition numbers
of both the local problem for computing the IFE functions and the global problem for computing
IFE solutions are independent of the interface
location relative to the mesh. Namely, the proposed method does not suffer from the presence
of small-cut interface elements. Nevertheless, numerical experiments indicate that the stability
estimates are suboptimal with respect to the contrast of the values of coefficient function $\beta$. To our best knowledge, this is the first arbitrary $p^{th}$-degree IFE method for elliptic
interface problems with nonhomogeneous interface jump conditions for which a solid mathematical foundation
has been established. Moreover, all methods based on pointwise enforcement \cite{1998Li,2004LiLinLinRogers,2016GuoLin},
weak jump conditions \cite{2014AdjeridBenromdhaneLin,2014AdjeridBenromdhaneLin2}, least-squares formulation \cite{2016AdjeridGuoLin,2018ZhuangGuo} and Cauchy problem \cite{2019GuoThesis,2019GuoLin} for interface problems with homogeneous jump conditions and linear interface yield the same IFE spaces. Thus,
the error analysis based on \textit{Cauchy mapping} extends to all those methods in the literature cited above
for constructing IFE spaces.

\commentout{
    For one-dimensional interface problems the IFE spaces existing in the litterature and constructed
    by directly enforcing the extended jump conditions are identical to the IFE spaces obtained by solving local Cauchy problems. Hence, the error and stability analyses given here hold for both IFE spaces. However, for two-dimensional
    problems the IFE spaces constructed by enforcing the extended interface conditions
    are, in general, different than those
    derived by solving Cauchy problems but both lead to optimal convergence rates.
    The relationship between the two IFE spaces needs further investigation.
}

This manuscript is organized as follows: in Section \ref{notations}, we recall a few notations and assumptions. In Section \ref{norm_equi}, we describe the construction procedure of the proposed IFE functions. In Section \ref{sec:enriched_IFE}, we present the proposed enriched IFE scheme and
derive \emph{a priori} error estimates in both energy and $L^2$ norms. In Section \ref{sec:conditioning},
we perform a stability analysis for both the local and global problems and derive upper bounds for
condition numbers. We perform several computational experiments and present numerical results in Section \ref{sec:num_example} to corroborate our theoretical results and to further explore stability numerically.

\section{Notations and Assumptions} \label{notations}

For a given bounded domain $\Omega \subset \mathbb{R}^2$, we consider an interface-independent, shape-regular, and quasi-uniform family of triangular meshes $\mathcal{T}_h,~h\geq 0$ with mesh size $h=\max_{T\in\mathcal{T}_h}\{h_T\}$ where $h_T$ is the diameter of an element $T\in\mathcal{T}_h$.
Let $\mathcal{N}_h$ and $\mathcal{E}_h$ be the sets of nodes and edges of the mesh $\mathcal{T}_h$, respectively.
For each mesh $\mathcal{T}_h$, the interface $\Gamma$ cuts some of its elements called interface elements while the remaining elements are called non-interface elements. Let $\mathcal{T}^i_h$ and $\mathcal{T}^n_h$ denote the sets of interface and non-interface elements, respectively. Finally, let $\mathcal{E}^i_h$ denote the collection of all the edges of elements in $\mathcal{T}^i_h$ and $\mathcal{E}^n_h=\mathcal{E}_h\backslash \mathcal{E}^i_h$.

For each interface element $T$, we define a fictitious element $T_{\lambda}$ as the homothetic image of $T$ where the homothetic center is the incenter $G$ of $T$ and the scaling factor $\lambda\geqslant1$ is independent of mesh size $h$, i.e.,
$$
T_{\lambda} = \{ X\in \mathbb{R}^2 : \exists Y\in T ~ \text{such that}~ \overrightarrow{GX}=\lambda \overrightarrow{GY} \}.
$$
See Figure \ref{fig:homothetic_1} for an illustration where $T=A_1A_2A_3$ is an interface element and
$T^\lambda = A^{\lambda}_1A^{\lambda}_2A^{\lambda}_3$ is an associated fictitious element. For simplicity's sake, we assume $T_{\lambda}\subset \Omega$ for every interface element $T$. In addition, for each fictitious element $T_{\lambda}$, we let $\Gamma^{\lambda}_T=T_{\lambda}\cap \Gamma$ which implies that $\Gamma_T:=\Gamma^{1}_T=T\cap \Gamma$. The fictitious element idea improves the conditioning of computing
higher-degree IFE functions on small-cut elements and helps establish \emph{a priori} error estimates \cite{2019GuoLin,2018ZhuangGuo}.
\begin{figure}[h]
  \centering
  \includegraphics[width=0.5\textwidth]{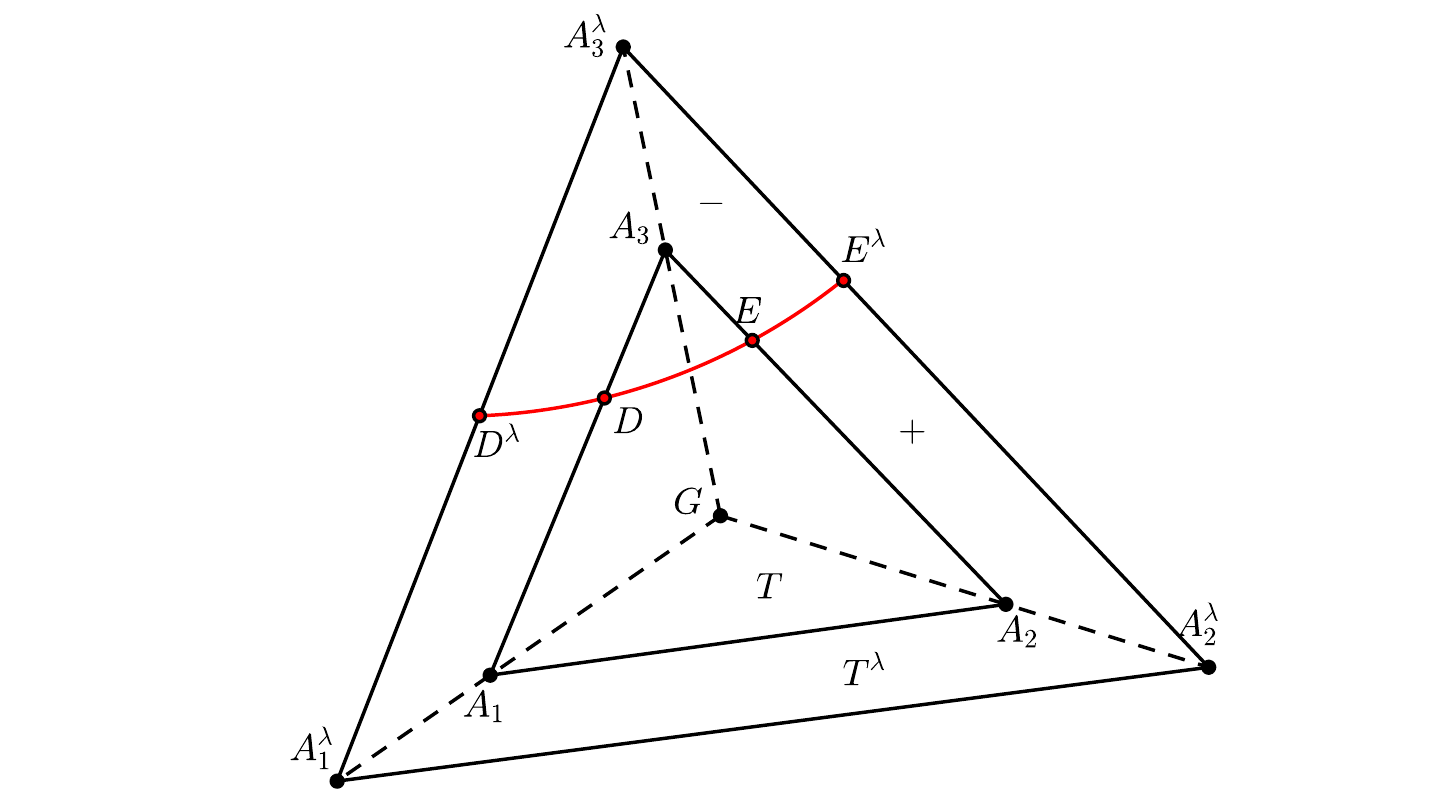}
  \caption{$T = \triangle A_1A_2A_3$ and its fictitious element
  $T_\lambda= \triangle A_1^\lambda A_2^\lambda A_3^\lambda$.}
  \label{fig:homothetic_1}
\end{figure}

To facilitate the error and stability analyses, we consider meshes $\mathcal{T}_h$
satisfying the following two assumptions:
\begin{itemize}
  \item[(\textbf{A1})] The interface intersects an element $T\in\mathcal{T}^i_h$ and its fictitious element $T_{\lambda}$ at two distinct points each located  on a different edge of $T$ and $T_{\lambda}$.
   \item[(\textbf{A2})] There exists a fixed integer $M$ such that for each element $K$ in $\mathcal{T}_h$, the cardinality of the set $\{ T\in\mathcal{T}^i_h:K\cap T_{\lambda}\neq\emptyset\}$ is less or equal to $M$.
\end{itemize}

Throughout this article, we use $\mathbb{P}_k(\omega)$ to denote the space of all
polynomials of degree not exceeding $k$, $k\geqslant 0$ on a sub-domain $\omega \subset \Omega$. Furthermore, we define $\Pi^{k}_{{\omega}}$ as the $L^2$ projection
operator: $H^{k}(\omega) \rightarrow \mathbb{P}_k(\omega)$. Next we consider two
Sobolev extension operators
$\mathfrak{E}^s$ ~:~ $PH^{k+2}(\Omega) \rightarrow H^{k+2}(\Omega)$ for $s=-,+$
that extends $u^s$ from $\Omega^s$ to $\Omega$. For instance, for $u=(u^+,u^-) \in PH^{k+2}(\Omega)$,
$\mathfrak{E}^+ (u)|_{\Omega^+} = u^+ $ and $\mathfrak{E}^+ (u)|_{\Omega^-} = v$ is
the extension of $u^+$ on $\Omega^-$. Similarly,  $\mathfrak{E}^- (u)|_{\Omega^-} = u^- $ and
$\mathfrak{E}^-(u)|_{\Omega^+} = w$ is the extension of $u^-$ on $\Omega^+$.
We further assume the extension operators $\mathfrak{E}^\pm$ \cite{2001GilbargTrudinger} satisfy
\begin{align}
       \sum_{m=1}^{k+2} | \mathfrak{E}^s (u) |_{H^m (\Omega)} \leqslant C_E \sum_{m=1}^{k+2} | u^{s} |_{H^{m}(\Omega^{s})}, ~~s=+,- \label{soblev_bound}
\end{align}
where the constant $C_E$ depends on $\Omega^{\pm}$, $\Omega$ and $k$. The estimate \eqref{soblev_bound} follows from the boundedness of the Sobolev extension (Theorem 7.25 in \cite{2001GilbargTrudinger}) and the Poincar\'e inequality.

Next, if the true solution $u$ of (\ref{model}) is in $PH^{k+2}(\Omega)$ and has Sobolev extensions
$\mathfrak{E}^s(u), ~s=+,-$ in $H^{k+2}(\Omega)$  we define the following operators
from $\mathfrak{F}_\Omega^s ~:~PH^k(\Omega) \rightarrow H^{k}(\Omega)$ such that
 $\mathfrak{F}_\Omega^s (f):=-\beta \triangle \mathfrak{E}^s (u), \ s=+,-$.
Let us note that
 $\mathfrak{F}_\Omega^+ (f)|_{\Omega^+}=-\beta^+ \triangle u^+ = f^+$
 while $\mathfrak{F}_{\Omega}^+ (f)|_{\Omega^-}=-\beta^- \triangle \mathfrak{E}^+ (u)| _{\Omega^-} \neq f^-$
in general, but it can be considered as an extension of $f^+$ on $\Omega^-$. Similarly,
  $\mathfrak{F}_\Omega^- (f)|_{\Omega^-}=-\beta^+ \triangle u^- = f^-$
 while $\mathfrak{F}_{\Omega}^- (f)|_{\Omega^+}=-\beta^+ \triangle \mathfrak{E}^+ (u)| _{\Omega^+} \neq f^+$ is considered as an extension of $f^-$ on $\Omega^+$. Hence,
 $\mathfrak{F}_{\Omega}^s$ is a special Sobolev extension operator
 for the right-hand side $f$ in (\ref{model}).

In the remainder of this manuscript we adopt the notation $A \ \lesssim \ B $ for the relation $A\ \leqslant\ C B$ with a generic constant $C>0$ independent of the mesh size, the coefficients $\beta^{\pm}$, and the interface location on all interface elements $T$ and $T_{\lambda}$. Similarly, we use the standard notation $\simeq$ for the equivalence of norms.

We end this section with a discussion of the solution regularity for interface problems.
Let us first note that the regularity of the solution of interface problem with homogeneous jump conditions is well established \cite{1970Babuska}. Improvements of these results appeared in \cite{2010ChuGrahamHou,2015GuzmanSanchezSarkisP1,2002HuangZou} which include details on
how the regularity constant depends on $\beta^{+}$ and $\beta^-$. The nonhomogeneous problem with
$J_D=0$ and $J_N\neq0$ was studied in \cite{1998ChenZou} for second-order problems and in \cite{1994NicaiseSandig1,1994NicaiseSandig2} for
more general interface problems. Here in order to show the effect of the coefficients $\beta^{+}$ and $\beta^-$, we
follow the ideas in \cite{1998ChenZou} to prove the following theorem.

\begin{thm}
\label{thm_regularity}
Let $\Omega$ be a convex domain and $\Gamma$ smooth enough. Given an integer $m\geqslant 0$, let $J_D\in H^{m+3/2}(\Gamma)$, $J_N\in H^{m+1/2}(\Gamma)$ and $f\in PH^{m}(\Omega)$, then the interface problem \eqref{model} admits a unique solution $ u\in PH^{m+2}(\Omega)$ such that
\begin{equation}
\label{u_regularity}
\sum_{k=1}^{m+2} | \beta u |_{PH^k(\Omega)} \leqslant C_r \left( \min\{\beta^+,\beta^-\} \|J_D\|_{H^{m+3/2}(\Gamma)} +  \|J_N\|_{H^{m+1/2}(\Gamma)} + \| f \|_{PH^{m}(\Omega)} \right),
\end{equation}
where $C_r$ is independent of $f$, $J_D$, $J_N$, $\beta^{+}$ and $\beta^-$.
\end{thm}
\begin{proof}
See Appendix \ref{append_regularity}.
\end{proof}

Inspired by the proof of this theorem, we derive in the next section
an algorithm for constructing IFE shape functions by solving local problems
that model discontinuous solution, flux and right-hand side accross the interface.



\section{IFE Functions and an Associated IFE Method}
\label{norm_equi}

In this section we present a procedure to construct IFE functions on an interface element $T$ and derive the enriched IFE method.

\subsection{IFE Functions}

The main goal is to find two polynomials
$w_p^+ \in \mathbb{P}_p(T_\lambda^+)$, $w_p^- ~\in~\mathbb{P}_p(T_\lambda^-)$ that satisfy the homogeneous interface conditions \ref{nonhomo_jump_cond_1}, \ref{nonhomo_jump_cond_2}, ($J_D=0,~ J_N=0$) in certain weak sense.
By Appendix \ref{append_cauchy}, given $z_p \in \mathbb{P}_p(T_\lambda)$ the extended interface conditions
across a linear interface $\Gamma$ yield a unique polynomial $v_p \in \mathbb{P}_k(T_\lambda)$ as a solution to the Cauchy problem in $T_\lambda^-$:
\begin{equation}\label{eq:cauchy}
  \triangle v_p = \frac{\beta^+}{\beta^-} \triangle z_p, ~~\text{in} ~T_\lambda^- \text{~~and~~} v_p = z_p, ~~
 \partial_{\mathbf{n}} v_p = \frac{\beta^+}{\beta^-}\partial_{\mathbf{n}} {z_p} ~~ \text{on}~ \Gamma_T^\lambda
\end{equation}
where $\partial_{\mathbf{n}} {v_p} = \nabla v_p \cdot \mathbf{n}$ and $\partial_{\mathbf{n}} {z_p} = \nabla z_p \cdot \mathbf{n}$ are the normal derivatives.
However, for general curved interfaces, given $z_p \in \mathbb{P}_p(T_\lambda)$, there is no $v_p \in \mathbb{P}_p(T_\lambda)$ to satisfy the above Cauchy problem. Instead we consider a least-squares approximation $v_p \in \mathbb{P}_p(T_\lambda)$ that minimizes
\begin{equation} \label{eq:min1}
||\triangle v_p - \frac{\beta^+}{\beta^-} \triangle {z_p}||_{T_\lambda^-}^2 +
\frac{1}{h_T^3} || v_p - z_p||_{\Gamma_T^\lambda}^2 + \frac{1}{h_T}
|| \partial_{\mathbf{n}} v_p - \frac{\beta^+}{\beta^-}\partial_{\mathbf{n}} z_p||_{\Gamma_T^\lambda}^2.
\end{equation}
To approximate particular solutions corresponding to the case $J_D \neq 0$, $J_N=0$ and $f$ smooth we determine $v_p\in\mathbb{P}_p(T_\lambda)$ that minimizes
\begin{equation} \label{eq:min2}
||\triangle v_p - \frac{\beta^+}{\beta^-} \triangle z_p||_{T_\lambda^-}^2 +
\frac{1}{h_T^3} || v_p - z_p-J_D||_{\Gamma_T^\lambda}^2 + \frac{1}{h_T}
|| \partial_{\mathbf{n}} v_p - \frac{\beta^+}{\beta^-}\partial_{\mathbf{n}} z_p||_{\Gamma_T^\lambda}^2.
\end{equation}
Similarly, for $J_D=0$, $J_N \neq 0$ and $f$ smooth  we compute $v_p \in \mathbb{P}_p(T_\lambda)$ that minimizes
\begin{equation}\label{eq:min3}
||\triangle v_p - \frac{\beta^+}{\beta^-} \triangle z_p||_{T_\lambda^-}^2 +
\frac{1}{h_T^3} || v_p - z_p||_{\Gamma_T^\lambda}^2 + \frac{1}{h_T}
|| \partial_{\mathbf{n}} v_p - \frac{\beta^+}{\beta^-}\partial_{\mathbf{n}} z_p-\frac{J_N}{\beta^-}||_{\Gamma_T^\lambda}^2.
\end{equation}
Finally, for $J_D=0$, $J_N=0$ and $f$ discontinuous we determine $v_p\in \mathbb{P}_p(T_\lambda^-)$ that minimizes
\begin{equation}\label{eq:min4}
||\triangle v_p - \frac{\beta^+}{\beta^-} \triangle z_p + \frac{\mathfrak{F}_{T_\lambda}^-(f) - \mathfrak{F}_{T_\lambda}^+(f)}{\beta^-} ||_{T_\lambda^-}^2 +
\frac{1}{h_T^3} || v_p - z_p||_{\Gamma_T^\lambda}^2 + \frac{1}{h_T}
|| \partial_{\mathbf{n}} v_p - \frac{\beta^+}{\beta^-}\partial_{\mathbf{n}} z_p||_{\Gamma_T^\lambda}^2.
\end{equation}
For each of the above cases, the corresponding macro $p^{th}$-degree (piecewise) polynomial IFE function $w_p=(w_p^+,w_p^-)$ on $T_\lambda$ is such
that $w_p^+= z_p|_{T_\lambda^+}$ and $w_p^- = v_p|_{T_\lambda^-}$. We further note that the procedure of solving least-squares problem described above is essentially equivalent to ``solve" Cauchy problems on the associated fictitious sub-element $T_{\lambda}^-$ with a least-squares finite element method \cite{2009BochevGunzburger,2019GuoLin,2017HUMUYE}. More specifically, we consider the following bilinear forms :
\begin{subequations}
\label{cauchy_local}
\begin{align}
    &  a_{\lambda}(r_p,q_p) =  \int_{T_{\lambda}^{-}} \triangle r_p \triangle q_p  dX + h_T^{-3} \int_{\Gamma^{\lambda}_T} r_p q_p ds + h_T^{-1} \int_{\Gamma^{\lambda}_T} \partial_{\mathbf{ n}}r_p ~ \partial_{\mathbf{ n}}q_p ds, ~~~ r_p,q_p\in \mathbb{P}_p(T_{\lambda}),  \label{cauchy_bilinear_form_l} \\
    & b_{\lambda}(r_p,q_p) =  \int_{T_{\lambda}^{-}} \frac{\beta^+}{\beta^-} \triangle r_p \triangle q_p  dX + h_T^{-3} \int_{\Gamma^{\lambda}_T} r_p q_p ds + h_T^{-1} \int_{\Gamma^{\lambda}_T} \frac{\beta^+}{\beta^-} \partial_{\mathbf{ n}}r_p ~ \partial_{\mathbf{ n}}q_p ds, ~~~ r_p,q_p\in \mathbb{P}_p(T_{\lambda}),  \label{cauchy_bilinear_form_r}
\end{align}
\end{subequations}
that induce the following energy type semi norms
\begin{equation}
\label{cauchy_norm}
\vertiii{v_p}^2_{a_{\lambda}} = a_{\lambda}(v_p,v_p), ~~~~~ \vertiii{v_p}^2_{b_{\lambda}} = b_{\lambda}(v_p,v_p), ~~~ \forall v_p\in \mathbb{P}_p(T_{\lambda}).
\end{equation}
By the assumption $\beta^-\geqslant\beta^+$ we have
\begin{equation}
\label{cauchy_norm_relation}
\frac{\beta^+}{\beta^-}  \vertiii{v_p}^2_{a_{\lambda}} \leqslant \vertiii{v_p}^2_{b_{\lambda}} \leqslant  \vertiii{v_p}^2_{a_{\lambda}}, ~~~ \forall v_p\in\mathbb{P}_p(T_\lambda).
\end{equation}
In fact, by Theorem 4.1 in \cite{2019GuoLin}, both $\vertiii{\cdot}_{a_{\lambda}}$ and $\vertiii{\cdot}_{b_{\lambda}}$ are norms on $\mathbb{P}_p(T_\lambda)$ with arbitrary degree $p$.
Furthermore, by Lemma 4.2 in \cite{2017HUMUYE}
, the bilinear form $a_{\lambda}(\cdot,\cdot)$ is coercive with respect to the norm $\vertiii{\cdot}_{a_{\lambda}}$ on the space $\mathbb{P}_p(T_\lambda)$. Then, for a given $z_p \in \mathbb{P}_p(T_\lambda)$ one can check that the least-squares solution $v_p \in \mathbb{P}_p (T_\lambda)$ of (\ref{eq:cauchy})
satisfies the discrete weak problem
\begin{equation}
a_{\lambda}(v_p,q_p) = b_{\lambda}(z_p,q_p), ~~ \forall ~ q_p \in \mathbb{P}_p(T_\lambda),
\end{equation}
from which we introduce the following mapping that maps $z_p$ to $v_p$.
\begin{definition}
\label{def:cauchy_ext}
The {\em Cauchy Mapping} $\mathfrak{C}_T: \mathbb{P}_p(T_\lambda) \rightarrow \mathbb{P}_p(T_\lambda)$ is such that, given $z_p \in \mathbb{P}_p(T_\lambda)$, its image $\mathfrak{C}_T(z_p) \in \mathbb{P}_p(T_\lambda) $ is the solution of the discrete local variational problem
\begin{equation}
\label{cauchy_exten}
a_{\lambda}(\mathfrak{C}_T(z_p),q_p) = b_{\lambda}(z_p,q_p), ~~~ \forall q_p\in\mathbb{P}_p(T_\lambda).
\end{equation}
\end{definition}
By the coercivity of $a_{\lambda}(\cdot,\cdot)$ and $b_{\lambda}(\cdot,\cdot)$, $\mathfrak{C}_T$ is bijective.
The Cauchy mapping $\mathfrak{C}_T$ plays a important role in both the theoretical analysis and the construction of the IFE spaces where on each interface element $T\in\mathcal{T}^i_h$, the local $p^{th}$-degree macro polynomial IFE space is given by
\begin{equation}
\label{loc_IFE_space_inter}
S^p_h(T) = \{ w_p=(w_p^+,w_p^-) ~:~ w_p^+ = z_p|_{T^+} ~~and~~ w_p^- =\mathfrak{C}_T(z_p)|_{T^-}, ~ \forall~ z_p\in \mathbb{P}_p(T_\lambda) \}, ~~ \forall~ T\in\mathcal{T}^i_h.
\end{equation}
These local IFE spaces on interface elements will be used to approximate the homogeneous
component of the exact solution.

Now we proceed to define the enrichment IFE functions for approximating the nonhomogeneous component
of the solution of  the interface problem \eqref{model}. The enrichment functions associated with the data on the interface are derived using the least-squares problems \eqref{eq:min2} and \eqref{eq:min3} for which we introduce the following two mappings.
\begin{definition}
\label{def:cauchy_ext_DN}
Let $\mathfrak{C}_{T,D}:\mathbb{P}_p(T_\lambda) \rightarrow \mathbb{P}_p(T_\lambda)$ and $\mathfrak{C}_{T,N}:\mathbb{P}_p(T_\lambda) \rightarrow \mathbb{P}_p(T_\lambda)$ be such that given $z_p \in \mathbb{P}_p(T_\lambda)$, its images $\mathfrak{C}_{T,D}(z_p)$ and $\mathfrak{C}_{T,N}(z_p) \in \mathbb{P}_p(T_\lambda)$ are solutions of the following local variational problems, respectively:
\begin{subequations}
\label{cauchy_ext_DN}
\begin{align}
  & a_{\lambda}(\mathfrak{C}_{T,D}(z_p),q_p) = b_{\lambda}(z_p,q_p) + h^{-3}_T \int_{\Gamma^{\lambda}_T} J_D q_p ds, ~~~ \forall~ q_p\in\mathbb{P}_p(T_\lambda), \label{cauchy_exten_D}   \\
    &   a_{\lambda}(\mathfrak{C}_{T,N}(z_p),q_p) = b_{\lambda}(z_p,q_p) + h^{-1}_T \int_{\Gamma^{\lambda}_T} \frac{ J_N \partial_{\mathbf{ n}}q_p }{\beta^-} ds, ~~~ \forall~ q_p\in\mathbb{P}_p(T_\lambda). \label{cauchy_exten_N}
\end{align}
\end{subequations}
\end{definition}
Since $\mathfrak{C}_{T,D}$ and $\mathfrak{C}_{T,N}$ are determined by modifying the right hand side of \eqref{cauchy_exten}, they are still bijective. By these two mappings, we introduce
two macro $p^{th}$-degree polynomials on every interface element $T$ as follows:
\begin{equation}
\label{enrich_IFE_fun_DN}
\phi_{T,D}=
\begin{cases}
      & 0 ~~~~~~~~~~~\, \text{on} ~ T^+, \\
      & \mathfrak{C}_{T,D}(0) ~~~ \text{on}~ T^-,
\end{cases}
~~\text{and}~~
\phi_{T,N}=
\begin{cases}
      & 0 ~~~~~~~~~~~\, \text{on} ~ T^+, \\
      & \mathfrak{C}_{T,N}(0) ~~~ \text{on}~ T^-,
\end{cases}~~\forall ~T \in \mathcal{T}_h^i.
\end{equation}
In fact, $\mathfrak{C}_{T,D}(0)$ and $\mathfrak{C}_{T,N}(0)$, respectively, are minimizers of \eqref{eq:min2} and \eqref{eq:min3} with $z_p=0$.
These two functions will be used to homogenize the interface jump conditions \eqref{nonhomo_jump_cond_1} and \eqref{nonhomo_jump_cond_2} in a weak sense.

For $p\geqslant2$, we resolve the singularity induced by a rough right-hand side function $f$ by
another mapping defined by the least-square problem \eqref{eq:min4}.

\commentout{
we obtain optimal convergence rates by resolving the singularity
induced by a rough right-hand side function $f$ for which we derive the following mapping from the
least-square problem \eqref{eq:min4}.
}
\begin{subequations} \label{eq:CT_phi_Tf}
\begin{definition}
\label{def:cauchy_ext_Lap}
For every function $\varphi\in L^2(T_{\lambda})$, we let $\mathfrak{C}_{T,\varphi}:\mathbb{P}_p(T_\lambda) \rightarrow \mathbb{P}_p(T_\lambda)$ be such that given $z_p \in \mathbb{P}_p(T_\lambda)$, its image $\mathfrak{C}_{T,\varphi}(z_p)\in \mathbb{P}_p(T_\lambda)$
is the solution to the local variational problem
\begin{equation}
\label{cauchy_ext_Lap}
  a_{\lambda}(\mathfrak{C}_{T,\varphi}(z_p),q_p) = b_{\lambda}(z_p,q_p) +  \int_{T^-_{\lambda}} \varphi \triangle q_p  dX, ~~~ \forall ~q_p\in\mathbb{P}_p(T_\lambda).
\end{equation}
\end{definition}
Again, $\mathfrak{C}_{T,\varphi}$ is determined by modifying the right hand side of
\eqref{cauchy_exten}
and is a bijective mapping.
In order to handle the singularity in $f$ across the interface we introduce the following function
\begin{equation}
\label{ctf}
 \varphi_f =\frac{ \Pi^{p-2}_{T^+_{\lambda}} f^+  - \Pi^{p-2}_{T^-_{\lambda}} f^- }{\beta^-} \in \mathbb{P}_{p-2}(T_\lambda),
\end{equation}
where $\Pi^{k}_{T^s_{\lambda}}f^s \in \mathbb{P}_k(T_\lambda)$ is the $L^2$ projection  satisfying
$(\Pi^{k}_{T^s_{\lambda}}f^s - f^s, q_k)_{T_\lambda^s}=0,~\forall q_k \in \mathbb{P}_k(T_\lambda)$ for $s=+,-$. Applying \eqref{cauchy_ext_Lap} with $\varphi=\varphi_f$ and $z_p=0$ we define
the following macro $p^{th}$-degree polynomial:
\begin{equation}
\label{enrich_IFE_fun_Lap}
\phi_{T,f}=
\begin{cases}
      & 0 ~~~~~~~~~~~~~ \text{on} ~ T^+, \\
      & \mathfrak{C}_{T,\varphi_f}(0) ~~~~ \text{on}~ T^-.
\end{cases}
\end{equation}
\end{subequations}
Let us note that $\mathfrak{C}_{T,\varphi_f}(0)$ is the solution of \eqref{eq:min4} with $\mathfrak{F}_{T_\lambda}^s(f)$ being replaced by $ \Pi_{T_{\lambda}^s}^{p-2} f^s $. Furthermore, the mappings defined by \eqref{cauchy_exten}, \eqref{cauchy_ext_DN} and \eqref{cauchy_ext_Lap} can be interpreted as discrete operators that extend polynomials from one side of an interface element to the other side according to the interface jump data. Following \cite{2019GuoLin},
we say that the macro polynomials described by \eqref{loc_IFE_space_inter}, \eqref{enrich_IFE_fun_DN} and \eqref{enrich_IFE_fun_Lap} are constructed by \textit{Cauchy extensions}.

We refer to the macro polynomials defined by \eqref{enrich_IFE_fun_DN} and \eqref{enrich_IFE_fun_Lap}
as enrichment IFE functions because we will use them to enrich the local IFE space $S_h^p(T)$ on every interface element $T \in \mathcal{T}^i_h$. Specifically, the local enrichment functions \eqref{enrich_IFE_fun_DN} and \eqref{enrich_IFE_fun_Lap}
are used to define the following  global enrichment IFE function:
\begin{equation}
\label{enrich_IFE_fun_glob}
\Phi_h =
\begin{cases}
\phi_{T,D} + \phi_{T,N} + \phi_{T,f},  & \text{on}  ~ T\in\mathcal{T}^i_h, \\
 0, & \text{on}  ~ T\in\mathcal{T}^n_h,
\end{cases}
\end{equation}
where the enrichment function $\phi_{T,f}$ is used only for $p\geqslant2$.

\begin{rem}
\label{rem:compare_homogenization}
The enrichment functions in \cite{2007GongTHESIS,2007GongLiLi,2010GongLi} are constructed to exactly satisfy the nonhomogeneous jump conditions through a level set method which are allowed to be
non polynomials. However, our enrichment IFE functions in \eqref{enrich_IFE_fun_DN} and \eqref{enrich_IFE_fun_Lap} are piecewise polynomials which satisfy the nonhomogeneous jump conditions in a weak but approximate sense. Moreover, as shown later, these enrichment IFE functions yield optimally converging IFE solutions of interface problems.
\end{rem}

Finally, following the general framework of the IFE methodology, on each non-interface element $T\in\mathcal{T}^n_h$, the local IFE space is simply the standard $p^{th}$-degree polynomial space
\begin{equation}
\label{loc_IFE_space_non_inter}
S^p_h(T) = \mathbb{P}_p(T),~~\forall T \in \mathcal{T}_h^n,
\end{equation}
i.e., the standard $p^{th}$-degree polynomials will be used on all the non-interface elements in the IFE method to be described later.


\subsection{A $p^{th}$-Degree Enriched IFE Method}
In this section we describe an enriched IFE method for solving the elliptic interface problem \eqref{model} by first introducing the following underlying function spaces
\begin{equation}
\begin{split}
\label{space_Vh}
V_h = \Big\{ v\in L^2(\Omega)~:~& v|_T\in H^1(T) \text{~if~} T\in\mathcal{T}^n_h, ~ v|_{T^{\pm}}\in H^1(T^{\pm})
\text{~if~} T\in \mathcal{T}^i_h, \\
& \text{and} ~ v~ \text{is continuous on each} ~ e\in\mathcal{E}^n_h, ~ v|_{\partial\Omega}=0 \Big\},
\end{split}
\end{equation}
\begin{equation}\label{eq:space_homogen}
V_{h,0}= \big \{ v ~\in ~V_h ~ : ~~ [v]|_\Gamma =0 ~~ \text{and} ~~ [\beta \partial_{\bf n} v ]|_{\Gamma}= 0 \big \}.
\end{equation}
Using the local IFE spaces \eqref{loc_IFE_space_non_inter}, \eqref{loc_IFE_space_inter}, and Appendix \ref{append_decomposition}, we introduce a global IFE subspace as follows:
\begin{equation}
\label{glob_IFE_space}
S^p_h(\Omega) = \Big\{ v\in L^2(\Omega)~:~ v|_{T}\in S^p_h(T), ~ \forall T\in\mathcal{T}_h ~ \text{and} ~ v ~ \text{is continuous on each} ~ e\in\mathcal{E}^n_h, ~ v|_{\partial\Omega}=0 \Big\}.
\end{equation}
We observe that $S^p_h(\Omega)\subseteq V_h$ and the functions in these two spaces may be discontinuous across the edges of interface elements as well as the interface itself. However, $S^p_h(\Omega)$ is not a subspace of $V_{h,0}$ for general curved interfaces since the homogeneous interface conditions are not exactly satisfied by polynomials.

Now, if the exact solution $u$ of \eqref{model} has the regularity shown in Theorem \ref{thm_regularity}, thus $u \in V_h$ and satisfies the following weak problem:
\begin{subequations}
\label{weak_form}
\begin{align}
     a_h(u,v) & =L_f(v),  ~~ \forall v \in V_{h,0}, \label{weak_form_1}
\end{align}
where
 \begin{align}
 a_h(u,v) & = \sum_{T\in\mathcal{T}_h} \int_T \beta \nabla u\cdot \nabla v dX \nonumber  \\
& - \sum_{e\in\mathcal{E}^i_h} \int_e \{ \beta \nabla u\cdot \mathbf{ n} \}_e [v]_e ds
   + \epsilon_0 \sum_{e\in\mathcal{E}^i_h} \int_e \{ \beta \nabla v \cdot \mathbf{ n} \}_e [u]_e ds + \sum_{e\in\mathcal{E}^i_h} \frac{\sigma^0 \gamma }{|e|^{\theta}} \int_e  [u]_e \, [v]_e ds  \label{weak_form_2} \\
& - \sum_{T\in\mathcal{T}^i_h} \int_{\Gamma_T} \{ \beta \nabla u\cdot \mathbf{ n} \}_{\Gamma} [v]_{\Gamma} ds
   + \epsilon_1  \sum_{T\in\mathcal{T}^i_h} \int_{\Gamma_T} \{ \beta \nabla v \cdot \mathbf{ n} \}_{\Gamma} [u]_{\Gamma} ds +  \sum_{T\in\mathcal{T}^i_h} \frac{\sigma^1 \gamma }{h^{\theta}_T} \int_{\Gamma_T}  [u]_{\Gamma} \, [v]_{\Gamma} ds, \nonumber
\end{align}
and
\begin{align}
   L_f(v)&=\int_{\Omega}fv dX + \int_{\Gamma}J_N\{v\}_{\Gamma} ds + \epsilon_1 \int_{\Gamma} J_D\{\beta\nabla v\cdot\mathbf{ n}\}_{\Gamma}ds + \sum_{T\in\mathcal{T}^i_h} \frac{\sigma^1 \gamma }{h^{\theta}_T} \int_{\Gamma_T}  J_D \, [v]_{\Gamma} ds,  \label{new_Lf}
\end{align}
\end{subequations}
with some positive constants $\sigma^0$ and $\sigma^1$ independent of the coefficients $\beta^{\pm}$. The remaining parameters are $\gamma=(\beta^-)^2/\beta^+$ and $\theta=1$. Similar penalty terms are also employed in other finite element methods based on unfitted meshes \cite{2015BurmanClaus,2012MASSJUNG,2016WangXiaoXu,2010WuXiao}.
Furthermore, we equip $V_h$ with the following broken norm
\begin{equation}
\begin{split}
\label{energy_norm_aux}
      \vertiii{ v }^2_h = & \sum_{T\in\mathcal{T}_h} \int_T \| \sqrt{\beta} \nabla v \|^2 dX + \sum_{e\in\mathcal{E}^i_h} \frac{\sigma^0 \gamma }{|e|} \int_e [v]^2_e ds  + \sum_{T\in\mathcal{T}^i_h} \frac{\sigma^1 \gamma}{h_T} \int_{\Gamma_T} [v]^2_{\Gamma} ds\\
      & + \frac{|e|}{\sigma^0 \gamma } \sum_{e\in\mathcal{E}^i_h} \int_e ( \{ \beta \nabla v \cdot \mathbf{ n} \}_e )^2 ds + \frac{h_T}{\sigma^1 \gamma }\sum_{T\in\mathcal{T}^i_h} \int_{\Gamma_T} ( \{ \beta \nabla v \cdot \mathbf{ n} \}_{\Gamma})^2 ds.
\end{split}
\end{equation}
We recall a few results about the equivalence of norms
from Lemmas 4.2, 3.5 and 3.7 in \cite{2019GuoLin} useful for the relevant analysis to be presented.
\begin{lemma}
\label{lem_norm_equiv}
The following hold on a  fictitious element $T_{\lambda}$, $\lambda>1$,  for every element
$T\in\mathcal{T}^i_h$ :
\begin{subequations}
\begin{equation}\label{lem_norm_equiv_eq1}
h^2_T \vertiii{ \cdot }_{{a_{\lambda}}} \simeq \| \cdot \|_{L^2(T^-_{\lambda})} ~~ \text{on } \mathbb{P}_p(T_\lambda),
\end{equation}
and
\begin{equation}
\label{lem_norm_equiv_eq2}
\| \cdot \|_{L^2(T^-_{\lambda})}  \simeq \| \cdot \|_{L^2(T^+_{\lambda})}  \simeq \| \cdot \|_{L^2(T_{\lambda})} \simeq \|\cdot\|_{L^2(T)} ~  \text{on} ~ \mathbb{P}_p(T_\lambda).
\end{equation}
\end{subequations}
\end{lemma}

Following discussions presented in Appendix \ref{append_decomposition},
we write the true solution as $u = {u}_0 + u_P$ where ${u}_0$ is the solution of a related interface problem with homogeneous interface conditions $J_D=0$ and $J_N=0$ and smooth right-hand side, while $u_P$ is the solution of
another related interface problem with nonhomogeneous interface conditions and a nonsmooth right-hand side.
If $u_P$ is known, then ${u}_0\in V_{h,0}$ can be determined by the following weak problem
\begin{equation} \label{eq:weak_homogen}
a_h({u}_0,v) = L_f(v) - a_h(u_P,v), ~~~ \forall ~ v \in V_{h,0}.
\end{equation}
Employing the global IFE space \eqref{glob_IFE_space} and the weak formulation \eqref{eq:weak_homogen}, we propose
an enriched IFE method for finding the enriched IFE solution $u_h = \tilde{u}_h + \Phi_h$ with $\tilde{u}_h\in S^p_h(\Omega)$ and $\Phi_h$ given by \eqref{enrich_IFE_fun_glob} such that
\begin{equation}
\label{enr_IFE_method_1}
a_h(u_h , v_h) = L_f(v_h), ~~~ \forall v_h\in S^p_h(\Omega).
\end{equation}
Since, according to \eqref{enrich_IFE_fun_glob}, $\Phi_h$ is completely determined by the data $J_D$, $J_N$ and $f$, the IFE method \eqref{enr_IFE_method_1} is further reduced to finding the homogeneous IFE solution $\tilde{u}_h \in S^p_h(\Omega)$ such that
\begin{equation}
\label{enr_IFE_method_2}
a_h(\tilde{u}_h,v_h) = L_f(v_h) - a_h(\Phi_h,v_h), ~~~ \forall ~v_h \in S^p_h(\Omega).
\end{equation}
We note that the IFE methods for homogeneous and nonhomogeneous interface jumps yield the same stiffness matrix. Therefore, when the stabilization constants $\sigma^0$ and $\sigma^1$ are large enough, the bilinear
form $a_h(\cdot,\cdot)$ is coercive as shown in \cite{2019GuoLin} Theorem 6.1, which guarantees the existence and uniqueness of both the homogeneous and enriched IFE solutions $\tilde{u}_h$ and $u_h = \tilde{u}_h + \Phi_h$.

%

\section{Error Analysis of the Enriched IFE Method}
\label{sec:enriched_IFE}
We present an error estimation for the enriched IFE method \eqref{enr_IFE_method_1} and
\eqref{enr_IFE_method_2} by first studying the approximation capabilities of the proposed IFE spaces and
enrichment IFE functions. Without loss of generality, we only discuss the case $p\geqslant2$, and the linear case $p=1$ (the enrichment $\phi_{T,f}$ is not needed) can readily handled by similar arguments.
Since the approximation properties of the IFE spaces on non-interface elements are well established,
our focus is on analyzing the approximation capabilities of both the IFE spaces \eqref{loc_IFE_space_inter}
and the enrichment IFE functions \eqref{enrich_IFE_fun_DN} and \eqref{enrich_IFE_fun_Lap} on interface
elements. For this purpose, we consider the following operator
\begin{equation}
\label{cauchy_exten_sum}
\tilde{\mathfrak{C}}_T(v) = \mathfrak{C}_T(v) + \mathfrak{C}_{T,D}(0) + \mathfrak{C}_{T,N}(0) + \mathfrak{C}_{T,\varphi_f}(0), ~~~ \forall v\in\mathbb{P}_p(T_\lambda^+).
\end{equation}
Next we apply the $L^2$ projection operator $\Pi^p_{T_{\lambda}}$ onto $\mathbb{P}_p(T_{\lambda})$ to define an auxiliary operator $Q_T$ on a fictitious element $T_{\lambda}$ associated with an interface element $T \in \mathcal{T}_h^i$ as
\begin{equation}
\label{J_interpolation_2}
Q_T u =
\begin{cases}
 Q^+_Tu :=   \Pi^p_{T_{\lambda}} \big ( \mathfrak{E}^+ (u) \big )  & \text{on} ~ T^+, \\
 Q^-_Tu :=  \widetilde{\mathfrak{C}}_T( \Pi^p_{T_{\lambda}} (\mathfrak{E}^+ (u)))  & \text{on} ~ T^-,
\end{cases}~~~~\forall u \in PH^{p+1}(T_{\lambda}),
\end{equation}
where $\mathfrak{E}^+ (u)$ is the Sobolev extension of $u^+$ to $\Omega$.
The diagram in Figure \ref{fig:diagram_all} shows the
relationship between associated approximations considered in this manuscript. We first note that
$Q_Tu|_{T^+} = {\Pi^p_{T_{\lambda}}(\mathfrak{E}^+ (u))}|_{T^+}$ and $Q_Tu|_{T^-} =\widetilde{\mathfrak{C}}_T( {\Pi^p_{T_{\lambda}}( \mathfrak{E}^+ (u)}))|_{T^-}$, hence $u^+ - Q_T^+u$ is a projection error. The error
$u^- - Q_T^-u = \mathfrak{E}^- (u) - \widetilde{\mathfrak{C}}_T( {\Pi^p_{T_{\lambda}} (\mathfrak{E}^+ (u))}) =
\mathfrak{E}^- (u) - \Pi^p_{T_{\lambda}} (\mathfrak{E}^- (u)) + \Pi^p_{T_{\lambda}} (\mathfrak{E}^- (u)) - \widetilde{\mathfrak{C}}_T( {\Pi^p_{T_{\lambda}} (\mathfrak{E}^+ (u))})$ which is the sum of a projection error and
the difference between the projection
$\Pi^p_{T_{\lambda}}(\mathfrak{E}^- (u))$ and
the extension $\widetilde{\mathfrak{C}}_T( \Pi^p_{T_{\lambda}}( \mathfrak{E}^+ (u)))$ on $T^-$, and
this difference is indicated by a dashed line in Figure \ref{fig:diagram_all}.
Then, we follow the analysis framework presented in \cite{2019GuoLin} to establish estimates of
$\Pi^p_{T_{\lambda}} (\mathfrak{E}^- (u)) - \widetilde{\mathfrak{C}}_T( \Pi^p_{T_{\lambda}} (\mathfrak{E}^+ (u)))$.
We begin with the following lemma about the polynomial projection operator on $T^s_{\lambda},~s = \pm$ for $T \in \mathcal{T}_h^i$.

\begin{figure}[H]
\centering
    \includegraphics[width=2.7in]{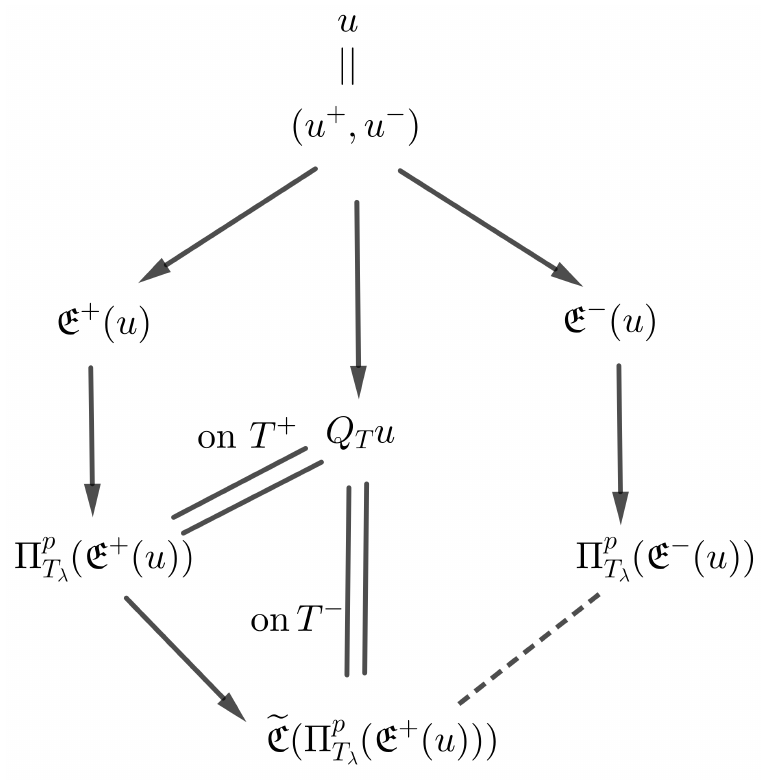}
  \caption{Diagram for analyzing approximation capabilities.}
  \label{fig:diagram_all} 
\end{figure}


\begin{lemma}
\label{lem_opti_projec_ptslambda}
The following estimate holds for $\lambda > 1$:
\begin{equation}
\label{lem_opti_projec_ptslambda_eq0}
\| \Pi^{p}_{T^s_{\lambda}} (w) - w \|_{L^2(T_{\lambda})} \lesssim h^{p+1}_T |w|_{H^{p+1}(T_{\lambda})}, ~~ s = \pm, ~~
\forall T \in \mathcal{T}_h^i, ~~ \forall w\in H^{p+1}(T_{\lambda}), ~~ p\geqslant 0.
\end{equation}
\end{lemma}
\begin{proof}
By the equivalence of norms $\|\cdot\|_{T^s_{\lambda}} \simeq \|\cdot\|_{T_{\lambda}}$ given by Lemma \ref{lem_norm_equiv}, we have the boundedness of $\Pi^{p}_{T^s_{\lambda}}$ on $T_{\lambda}$: for every $w\in L^2(T_{\lambda})$, there holds
\begin{equation}
\label{lem_opti_projec_ptslambda_eq1}
\| \Pi^{p}_{T^s_{\lambda}} (w) \|_{L^2(T_{\lambda})} \lesssim \| \Pi^{p}_{T^s_{\lambda}} (w) \|_{L^2(T^s_{\lambda})} \leqslant \| w \|_{L^2(T^s_{\lambda})} \leqslant \| w \|_{L^2(T_{\lambda})}, ~~s = \pm.
\end{equation}
Next, we note that $\Pi^{p}_{T_{\lambda}}$ has the optimal error bound on the whole $T_{\lambda}$, so we can use \eqref{lem_opti_projec_ptslambda_eq1} to obtain
\begin{equation}
\label{lem_opti_projec_ptslambda_eq2}
\| \Pi^{p}_{T_{\lambda}}w - \Pi^{p}_{T^s_{\lambda}} (w)  \|_{L^2(T_{\lambda})} = \|  \Pi^{p}_{T^s_{\lambda}}(\Pi^{p}_{T_{\lambda}} (w) - w) \big ) \|_{L^2(T_{\lambda})} \lesssim \| \Pi^{p}_{T_{\lambda}} (w) - w  \|_{L^2(T_{\lambda})} \lesssim h^{p+1}_T |w|_{H^{p+1}(T_{\lambda})}.
\end{equation}
Finally, by the triangular inequality, we have
\begin{equation}
\label{lem_opti_projec_ptslambda_eq3}
\| \Pi^{p}_{T^s_{\lambda}} (w) - w \|_{L^2(T_{\lambda})} \leqslant \| \Pi^{p}_{T^s_{\lambda}} (w) - \Pi^{p}_{T_{\lambda}}(w) \|_{L^2(T_{\lambda})} + \| \Pi^{p}_{T_{\lambda}}w - w  \|_{L^2(T_{\lambda})} \lesssim h^{p+1}_T |w|_{H^{p+1}(T_{\lambda})}.
\end{equation}
\end{proof}
In the next theorem we establish an estimate of the error between $\tilde{\mathfrak{C}}_T(\Pi^p_{T_{\lambda}} (\mathfrak{E}^+(u)))$ and $\Pi^p_{T_{\lambda}} (\mathfrak{E}^-(u))$.
\begin{thm}
\label{cauchy_mapping_prop}
Let $u\in PH^{p+1}(\Omega)$ satisfy \eqref{inter_PDE}, \eqref{nonhomo_jump_cond_1} and \eqref{nonhomo_jump_cond_2} with Sobolev extensions $\mathfrak{E}^s(u), ~s=+,-$. Then for $\lambda > 1$ and $p\geq 2$ we have
\begin{equation}
\begin{split}
\label{cauchy_mapping_prop0}
\vertiii{ \tilde{\mathfrak{C}}_T(\Pi^p_{T_{\lambda}} (\mathfrak{E}^+(u))) - \Pi^p_{T_{\lambda}} (\mathfrak{E}^-(u))}_{a_{\lambda}}   \lesssim h^{p-1}_T \left( |\mathfrak{E}^+(u)|_{H^{p+1}(T_{\lambda})} + |\mathfrak{E}^-(u)|_{H^{p+1}(T_{\lambda})}\right), ~~\forall T \in \mathcal{T}_h^i.
\end{split}
\end{equation}
\end{thm}
\begin{proof}
We let $w=\tilde{\mathfrak{C}}_T(\Pi^p_{T_{\lambda}} (\mathfrak{E}^+ (u))) - \Pi^p_{T_{\lambda}} (\mathfrak{E}^- (u)) \in\mathbb{P}_p(T_\lambda)$ and use \eqref{cauchy_exten}, \eqref{cauchy_ext_DN} and \eqref{cauchy_ext_Lap} to write
\begin{equation}
\begin{split}
\label{cauchy_mapping_prop1}
& \vertiii{ w }^2_{a_{\lambda}} = a_{\lambda} ( \tilde{\mathfrak{C}}_T( \Pi^p_{T_{\lambda}} (\mathfrak{E}^+ (u)))- \Pi^p_{T_{\lambda}} (\mathfrak{E}^- (u)) , w ) \\
= &~b_{\lambda}( \Pi^p_{T_{\lambda}} (\mathfrak{E}^+ (u)),  w)  - a_{\lambda}( \Pi^p_{T_{\lambda}} (\mathfrak{E}^- (u)),  w ) + h^{-3}_T \int_{\Gamma^{\lambda}_T} J_D w ds + h^{-1}_T \int_{\Gamma^{\lambda}_T} \frac{J_N \partial_{\mathbf{ n}}w}{\beta^-} ds + \int_{T^-_{\lambda}}  \varphi_f \triangle w  dX.
\end{split}
\end{equation}
For the first two terms on the right hand side of \eqref{cauchy_mapping_prop1}, we note that
\begin{align}
  b_{\lambda}( \Pi^p_{T_{\lambda}} (\mathfrak{E}^+ (u)),  w)  - a_{\lambda}( \Pi^p_{T_{\lambda}} (\mathfrak{E}^- (u)),  w ) =   \int_{T_{\lambda}^{-}} \left( \frac{\beta^+}{\beta^-}  \triangle \Pi^p_{T_{\lambda}} (\mathfrak{E}^+ (u)) -  \triangle \Pi^p_{T_{\lambda}} (\mathfrak{E}^- (u)) \right) \triangle w  dX \nonumber \\
   + h_T^{-3} \int_{\Gamma^{\lambda}_T} \left(  \Pi^p_{T_{\lambda}} (\mathfrak{E}^+ (u)) -  \Pi^p_{T_{\lambda}} (\mathfrak{E}^- (u))\right) w ds 
  +  h_T^{-1} \int_{\Gamma^{\lambda}_T} \left( \frac{\beta^+}{\beta^-}  \partial_{\mathbf{ n}}   \Pi^p_{T_{\lambda}} (\mathfrak{E}^+ (u))- \partial_{\mathbf{ n}}  \Pi^p_{T_{\lambda}}  (\mathfrak{E}^- (u))\right) ~ \partial_{\mathbf{ n}}w ds.  \label{cauchy_mapping_prop2}
\end{align}
Combining \eqref{cauchy_mapping_prop1} and \eqref{cauchy_mapping_prop2}, and applying H\"older's inequality, we obtain
\begin{equation}
\begin{split}
\label{cauchy_mapping_prop3}
\vertiii{ w }_{a_{\lambda}} \leqslant & \norm{\frac{\beta^+}{\beta^-}  \triangle \Pi^p_{T_{\lambda}} (\mathfrak{E}^+ (u)) -  \triangle \Pi^p_{T_{\lambda}} (\mathfrak{E}^- (u)) + \varphi_f }_{L^2(T^-_{\lambda})}
+  h^{-3/2}_T \norm{\Pi^p_{T_{\lambda}} (\mathfrak{E}^+ (u)) -  \Pi^p_{T_{\lambda}} (\mathfrak{E}^- (u)) + J_D}_{L^2(\Gamma^{\lambda}_T)} \\
& + h^{-1/2}_T \norm{\frac{\beta^+}{\beta^-}  \partial_{\mathbf{ n}}   \Pi^p_{T_{\lambda}} (\mathfrak{E}^+ (u))- \partial_{\mathbf{ n}}  \Pi^p_{T_{\lambda}}  (\mathfrak{E}^- (u)) + \frac{J_N}{\beta^-}}_{L^2(\Gamma^{\lambda}_T)}  :=  I + II +III.
\end{split}
\end{equation}
For the term $I$, recalling that $-\beta^+\triangle \mathfrak{E}^+ (u)|_{T_\lambda^+} = f^+$
and $-\beta^-\triangle \mathfrak{E}^- (u)|_{T_\lambda^-} = f^-$, using the definition \eqref{ctf}, the triangular inequality, and the error bound \eqref{lem_opti_projec_ptslambda_eq0} with $p$ replaced by $p-2$, we obtain
\begin{align}
I  \leqslant & \norm{ \frac{\beta^+}{\beta^-} \left( \triangle \Pi^p_{T_{\lambda}} (\mathfrak{E}^+ (u))  -   \triangle \mathfrak{E}^+ (u)  \right)}_{L^2(T^-_{\lambda})} + \norm{ \triangle \mathfrak{E}^- (u) -  \triangle \Pi^p_{T_{\lambda}} (\mathfrak{E}^- (u))}_{L^2(T^-_{\lambda})} \nonumber \\
+ & \norm{\frac{\beta^+}{\beta^-} \left( \triangle \mathfrak{E}^+ (u)  - \Pi^{p-2}_{T^+_{\lambda}} (\triangle \mathfrak{E}^+ (u)) \right)}_{L^2(T^-_{\lambda})} + \norm{ \Pi^{p-2}_{T^-_{\lambda}} (\triangle \mathfrak{E}^- (u)) -  \triangle \mathfrak{E}^- (u)}_{L^2(T^-_{\lambda})} \nonumber \\
 \lesssim & h^{p-1}_T \left( \frac{\beta^+}{\beta^-} |\mathfrak{E}^+ (u)|_{H^{p+1}(T_{\lambda})} + |\mathfrak{E}^- (u)|_{H^{p+1}(T_{\lambda})} \right). \label{cauchy_mapping_prop4}
\end{align}
For the second term $II$, using the trace inequality on the interface given by Lemma 3.2 in \cite{2016WangXiaoXu}, the error bound \eqref{lem_opti_projec_ptslambda_eq0}, and the nonhomogeneous jump condition \eqref{nonhomo_jump_cond_1}, we have
\begin{equation}
\begin{split}
\label{cauchy_mapping_prop6}
II  \leqslant & h_T^{-3/2} \left( \|  \Pi^p_{T_{\lambda}}(\mathfrak{E}^+ (u)) -  \mathfrak{E}^+ (u) \|_{L^2(\Gamma^{\lambda}_T)} + \|  \Pi^p_{T_{\lambda}}(\mathfrak{E}^- (u)) -  \mathfrak{E}^- (u) \|_{L^2(\Gamma^{\lambda}_T)} + \| \mathfrak{E}^+ (u) - \mathfrak{E}^- (u) + J_D \|_{L^2(\Gamma^{\lambda}_T)} \right) \\
\lesssim & h^{p-1}_T \left( |\mathfrak{E}^+ (u)|_{H^{p+1}(T_{\lambda})} + |\mathfrak{E}^- (u)|_{H^{p+1}(T_{\lambda})} \right).
\end{split}
\end{equation}
An estimate of $III$ can be established in a similar manner as
\begin{align}
III \leqslant h^{p-1}_T \left( \frac{\beta^+}{\beta^-} |\mathfrak{E}^+ (u)|_{H^{p+1}(T_{\lambda})} +  |\mathfrak{E}^- (u)|_{H^{p+1}(T_{\lambda})} \right).
\label{cauchy_mapping_prop7}
 \end{align}
Finally, combining  \eqref{cauchy_mapping_prop4}, \eqref{cauchy_mapping_prop6}, \eqref{cauchy_mapping_prop7}
and \eqref{cauchy_mapping_prop3} completes the proof.
\end{proof}
Now we are ready to analyze the approximation capabilities of the local IFE spaces in \eqref{loc_IFE_space_inter} augmented by the enrichment IFE functions \eqref{enrich_IFE_fun_DN} and \eqref{enrich_IFE_fun_Lap} on interface elements.
We state the following theorem about the approximation properties of $Q_{T}$.

\begin{thm}
\label{thm_loc_approx}
Under the assumptions of Theorem \ref{cauchy_mapping_prop}, we have
\begin{subequations}
\label{thm_loc_approx_eq0}
\begin{align}
  &  \sum_{j=0}^2 h^j_T \abs{\mathfrak{E}^+ (u) - Q^+_T (u)}_{H^j(T)}  \lesssim h^{p+1}_T |\mathfrak{E}^+ (u)|_{H^{p+1}(T_{\lambda})},  ~~\forall T \in \mathcal{T}_h^i, \label{thm_loc_approx_eq01}  \\
  &  \sum_{j=0}^2 h^j_T \abs{\mathfrak{E}^- (u) - Q^-_T (u) }_{H^j(T)}  \lesssim h^{p +1}_T \left( |\mathfrak{E}^+ (u)|_{H^{p+1}(T_{\lambda})} + |\mathfrak{E}^- (u)|_{H^{p+1}(T_{\lambda})} \right),
  ~~\forall T \in \mathcal{T}_h^i.  \label{thm_loc_approx_eq02}
\end{align}
\end{subequations}
\end{thm}
\begin{proof}
Since $Q^+_{T}u=\Pi^p_{T_{\lambda}} (\mathfrak{E}^+ (u))$, the estimate \eqref{thm_loc_approx_eq01} simply follows from the standard error bounds for the $L^2$ projection operator $\Pi^p_{T_{\lambda}}$.

In order to prove \eqref{thm_loc_approx_eq02}, we add and subtract $\Pi^p_{T_{\lambda}}(\mathfrak{E}^- (u))$ and apply the triangular inequality to write
\begin{equation}
\label{thm_loc_approx_eq1}
\abs{\mathfrak{E}^- (u) - Q^-_T u)}_{H^j(T)} \leqslant \abs{\Pi^p_{T_{\lambda}}(\mathfrak{E}^- (u)) - \widetilde{\mathfrak{C}}_T( \Pi^p_{T_{\lambda}} (\mathfrak{E}^+ (u)))}_{H^j(T)}  + \abs{ \mathfrak{E}^- (u) - \Pi^p_{T_{\lambda}}(\mathfrak{E}^- (u))}_{H^j(T)}.
\end{equation}
To estimate the first term in \eqref{thm_loc_approx_eq1}, we use the fact
$\widetilde{\mathfrak{C}}( \Pi^p_{T_{\lambda}} (\mathfrak{E}^+ (u)))  - \Pi^p_{T_{\lambda}} (\mathfrak{E}^- (u)) \in \mathbb{P}_p(T_\lambda)$ together with the equivalence of norms in Lemma \ref{lem_norm_equiv} and Theorem \ref{cauchy_mapping_prop}, to obtain
\begin{equation*}
\begin{split}
\label{thm_loc_approx_eq2}
\norm{ \widetilde{\mathfrak{C}}_T ( \Pi^p_{T_{\lambda}} (\mathfrak{E}^+ (u))) - \Pi^p_{T_{\lambda}} (\mathfrak{E}^- (u))}_{L^2(T)}
 \lesssim  h_T^2  \vertiii{ \widetilde{\mathfrak{C}}_T ( \Pi^p_{T_{\lambda}} (\mathfrak{E}^+ (u))) - \Pi^p_{T_{\lambda}} (\mathfrak{E}^- (u)) }_{a_{\lambda}}
 \end{split}
 \end{equation*}
 \begin{equation*}
  \lesssim h_T^{p+1} \left(  |\mathfrak{E}^+ (u)|_{H^{p+1}(T_{\lambda})} + |\mathfrak{E}^- (u)|_{H^{p+1}(T_{\lambda})} \right).
\end{equation*}
Then, the standard inverse inequality yields
\begin{align}
\label{thm_loc_approx_eq3}
h^j_T \abs{ \widetilde{\mathfrak{C}}_T ( \Pi^p_{T_{\lambda}} (\mathfrak{E}^+ (u))) - \Pi^p_{T_{\lambda}} (\mathfrak{E}^- (u))}_{H^j(T)} &\lesssim \norm{ \widetilde{\mathfrak{C}}_T ( \Pi^p_{T_{\lambda}}( \mathfrak{E}^+ (u))) - \Pi^p_{T_{\lambda}} (\mathfrak{E}^- (u))}_{L^2(T)} \nonumber \\
& \lesssim  h^{p+1}_T \left(  |\mathfrak{E}^+ (u)|_{H^{p+1}(T_{\lambda})} + |\mathfrak{E}^- (u)|_{H^{p+1}(T_{\lambda})} \right).
\end{align}
Applying the standard error estimates of the $L^2$ projection to the second term in \eqref{thm_loc_approx_eq1} and combining it with \eqref{thm_loc_approx_eq3}, we establish \eqref{thm_loc_approx_eq02}.
\end{proof}

To investigate the approximation capabilities of the global IFE space \eqref{glob_IFE_space},
we consider the approximation operator  $Q_h :  PH^{p+1}(\Omega) \to S^p_h(\Omega)$ such that
for each $u \in PH^{p+1}(\Omega)$, we let $Q_hu|_T \in S_h^p(T) = \mathbb{P}_p(T)$ be the standard Lagrange
interpolation of $u$ on every non-interface element $T \in \mathcal{T}_h^n$ and let
$Q_hu|_T = Q_Tu $ be the IFE function given by \eqref{J_interpolation_2} on every interface element $T \in \mathcal{T}_h^i$.
The next theorem provides \emph{a priori} bounds for the error $u - Q_hu$ in the energy norm.

\begin{thm}
\label{thm:energy_norm_approx}
Under the assumptions of Theorem \ref{cauchy_mapping_prop}, we have
\begin{equation}
\label{thm:energy_norm_approx_eq0}
\vertiii{ u - Q_h u }_h \lesssim h^{p}  \frac{\beta^-}{ \sqrt{\beta^+ } } \sum_{k=1}^{p+1} \left( | u^- |_{H^{k}(\Omega^-)} + | u^+ |_{H^{k}(\Omega^+)} \right) .
\end{equation}
\end{thm}
\begin{proof}
The arguments essentially are the same as those for Theorem 5.2 in \cite{2019GuoLin}.
Over non-interface elements, the standard estimate for Lagrange interpolation yields
\begin{equation}
\label{thm:energy_norm_approx_eq1}
\sum_{T\in\mathcal{T}^n_h} \int_T \beta \norm{ \nabla (u - Q_T u) }^2 dX \lesssim h^{2p} \beta^- \sum_{T\in\mathcal{T}^n_h}  |u|^2_{H^{p+1}(T)},
\end{equation}
where we have assumed $\beta^-\geqslant\beta^+$.  Then, we apply Theorem \ref{thm_loc_approx} on interface elements to have
\begin{equation}
\label{thm:energy_norm_approx_eq2}
\sum_{T\in\mathcal{T}^i_h} \int_T \beta \norm{ \nabla (u - Q_T u) }^2 dX  \lesssim h^{2p} \beta^- \sum_{T\in\mathcal{T}^i_h} \left( |\mathfrak{E}^+ (u)|^2_{H^{p+1}(T_{\lambda})} + |\mathfrak{E}^- (u)|^2_{H^{p+1}(T_{\lambda})} \right).
\end{equation}
Next, we proceed to estimate the edge penalty terms. On each interface edge $e\in\mathcal{E}^i_h$, we define $e^{\pm}=e\cap\Omega^{\pm}$ and $\tilde{h}=\max\{h_{T^1}, h_{T^2}\}$ where $T^1$ and $T^2$ are the two elements sharing $e$. Then by Theorem \ref{thm_loc_approx} and the standard trace inequality, we have
\begin{equation}
\begin{split}
\label{thm:energy_norm_approx_eq3}
\sum_{s=\pm} \frac{\sigma^0 \gamma}{|e|} \int_{e^{s}} \left[ u^{s} -Q^{s}_hu \right]^2_e ds & \lesssim  \sum_{s=\pm}  \sum_{j=1,2} \tilde{h}^{-1} \gamma \int_{e}
\left( (\mathfrak{E}^s(u) -Q^{s}_h u)|_{T^j} \right)^2 ds  \\
&\lesssim \sum_{s=\pm} \sum_{j=1,2} \gamma  \left( \tilde{h}^{-2} \| \mathfrak{E}^s(u)_E -Q^{s}_{T^j} u \|^2_{L^2(T^j)} + | \mathfrak{E}^s(u) - Q^{s}_{T^j} u |^2_{H^1(T^j)} \right)\\
&\lesssim  \tilde{h}^{2p} \gamma \left( |\mathfrak{E}^+ (u)|^2_{H^{p+1}( T^1_{\lambda}\cup T^2_{\lambda} )} + |\mathfrak{E}^- (u)|^2_{H^{p+1}( T^1_{\lambda}\cup T^2_{\lambda} )}
\right).
\end{split}
\end{equation}
Similar estimates also hold for the penalty terms on the non-interface edges, the interface, and the flux terms. Summing these estimates and applying the finite overlapping Assumption (\textbf{A2}), we obtain
\begin{equation}
\label{ thm:energy_norm_approx_eq5}
\vertiii{ u - Q_h u }_h \lesssim h^{p} \sqrt{\gamma} \left( | \mathfrak{E}^- (u) |_{H^{p+1}(\Omega)} + | \mathfrak{E}^+ (u) |_{H^{p+1}(\Omega)}  \right),
\end{equation}
which leads to the desired estimate \eqref{thm:energy_norm_approx_eq0} after applying \eqref{soblev_bound}.
\end{proof}

\begin{rem}
\label{rem:less_regularity}
When the enrichment functions $\mathfrak{C}_{T,D}(0)$, $\mathfrak{C}_{T,N}(0)$ and $\mathfrak{C}_{T,\varphi_f}(0)$ are not used in \eqref{cauchy_exten_sum}, the proposed approximation operator \eqref{J_interpolation_2} reduces to the one in \cite{2019GuoLin} which can be used to estimate approximation capabilities of the IFE spaces in \eqref{loc_IFE_space_inter} for functions satisfying homogeneous jump conditions and the extended jump conditions \eqref{normal_jump_cond}. In particular, by directly estimating the term I in \eqref{cauchy_mapping_prop3},
we can see that if the enrichment function $\mathfrak{C}_{T,\varphi_f}(0)$ is omitted from the definition \eqref{cauchy_exten_sum}, then the resulting approximation operator
is suboptimal with an $\mathcal{O}(h^{2-j})$ accuracy in the norms $|\cdot|_{H^j}$, $j=0,1$ and $\mathcal{O}(h)$ in $\vertiii{\cdot}_h$, for $p\geqslant 2$.
\end{rem}

Now, we are ready to estimate the error of the enriched IFE solution of the proposed scheme \eqref{enr_IFE_method_1}.
\begin{thm}
\label{thm:energy_error}
Let $u\in PH^{p+1}(\Omega)$ be the exact solution of the interface problem \eqref{model}. Assume that the mesh $\mathcal{T}_h$ is fine enough and $\sigma^0$ and $\sigma^1$ are large enough such that the coercivity holds for the bilinear form \eqref{weak_form_2} (Theorem 6.1 in \cite{2019GuoLin}). Then the enriched IFE solution $u_h = \tilde{u}_h + \Phi_h$ has the following error bound:
\begin{equation}
\label{thm:energy_error_eq0}
\vertiii{ u - u_h }_h \lesssim \frac{\beta^-}{\sqrt{\beta^+}} h^p \sum_{k=1}^{p+1} \left( | u^- |_{H^{k}(\Omega^-)} + | u^+ |_{H^{k}(\Omega^+)} \right).
\end{equation}
\end{thm}
\begin{proof}
First, since the exact solution $u$ satisfies the weak formulation \eqref{enr_IFE_method_1}, we have
\begin{equation}
\label{thm:energy_error_eq1}
a_h(u_h - Q_h u ,v_h)  = a_h(u - Q_h u ,v_h) , ~~~ \forall v_h \in S^p_h(\Omega).
\end{equation}
Moreover, since $u_h - Q_h u\in S^p_h(\Omega)$, the coercivity and continuity established in Theorems 6.1, 6.2 in \cite{2019GuoLin} lead to
\begin{equation}
\begin{split}
\label{thm:energy_error_eq2}
\vertiii{ u_h - Q_h u}^2_h & \lesssim a_h( u_h - Q_h u , u_h - Q_h u)  = a_h ( u - Q_hu, u_h - Q_h u) \lesssim \vertiii{u - Q_hu}_h \vertiii{ {u}_h - Q_hu}_h,
\end{split}
\end{equation}
which yields $\vertiii{ u_h - Q_h u}_h \lesssim  \vertiii{u - Q_hu}_h$.
Then, using Theorem \ref{thm:energy_norm_approx}, we obtain \eqref{thm:energy_error_eq0} as follows
\begin{equation}
\label{thm:energy_error_eq4}
\vertiii{ u -  u_h}_h \lesssim  \vertiii{ u - Q_hu}_h + \vertiii{  Q_hu - u_h }_h \lesssim h^{p}  \frac{\beta^-}{ \sqrt{\beta^+ } }\sum_{k=1}^{p+1} \left( | u^- |_{H^{k}(\Omega^-)} + | u^+ |_{H^{k}(\Omega^+)} \right).
\end{equation}
\end{proof}

In the next theorem we state and establish an optimal $L^2$ error estimate for the IFE error.
\begin{thm}
\label{thm:L2_error}
Under the assumptions of Theorem \ref{thm:energy_error}, the enriched IFE solution $u_h = \tilde{u}_h + \Phi_h$ has the following error bound:
\begin{equation}
\label{thm:L2_error_eq0}
\| u - u_h \|_{L^2(\Omega)}  \lesssim \left( \frac{\beta^-}{\beta^+} \right)^2 h^{p+1} \sum_{k=1}^{p+1} \left( | u^- |_{H^{k}(\Omega^-)} + | u^+ |_{H^{k}(\Omega^+)} \right).
\end{equation}
\end{thm}
\begin{proof}
We apply the standard duality argument in which we let $z\in PH^2(\Omega)$ be the solution to the interface problem \eqref{inter_PDE},\eqref{b_condition} with source term $u-{u}_h$ and homogeneous interface
 conditions \eqref{nonhomo_jump_cond_1} and \eqref{nonhomo_jump_cond_2}.
Then, for every $T \in \mathcal{T}_h^i$, let $I_T$ be the operator defined by
\begin{equation}
\label{J_interpolation_3}
I_T v =
\begin{cases}
 I^+_Tv :=   \Pi^p_{T_{\lambda}} \big ( \mathfrak{E}^+ (v) \big )  & \text{on} ~ T^+, \\
 I^-_Tv :=  {\mathfrak{C}}_T( \Pi^p_{T_{\lambda}} (\mathfrak{E}^+ (v)))  & \text{on} ~ T^-,
\end{cases}~~~~\forall v \in PH^{2}(T_{\lambda}), ~~T \in \mathcal{T}_h^i.
\end{equation}
We note that operator $I_T$ is similar to $Q_T$ but is defined only by the Cauchy mapping $\mathfrak{C}_T$. Then we define $I_hz \in S_h^p(\Omega)$ piecewise such that $I_hz|_T = I_Tz$ when $T \in \mathcal{T}_h^i$; otherwise, $I_hz|_T$ is the standard Lagrange interpolation when $T \in \mathcal{T}_h^n$. We note $z$ satisfies the weak formulation \eqref{enr_IFE_method_1} for all test functions in $V_h$. Therefore, combining the orthogonality condition $a_h(I_hz,u-u_h)=0$ and the continuity of the bilinear form $a_h(\cdot,\cdot)$ established in Theorem 6.2 in \cite{2019GuoLin}, we have
\begin{equation}
\begin{split}
\label{thm:L2_error_eq1}
\| u - u_h \|^2_{L^2(\Omega)}  = a_h(z, u - u_h) = a_h(z - I_hz, u - u_h) \lesssim \vertiii{z - I_hz }_h \vertiii{u - u_h}_h.
\end{split}
\end{equation}
By Theorem \ref{thm_regularity} we have following bound for $z$:
\begin{equation}
\label{thm:L2_error_eq20}
 \sum_{k=1}^{2} \left( \beta^-  |z^-|_{H^k(\Omega^-)} +  \beta^+ |z^+|_{H^k(\Omega^+)} \right) \lesssim \| u - u_h \|_{L^{2}(\Omega)}.
\end{equation}
Since $I_hz$ is derived from $Q_hz$ by dropping the terms
$\mathfrak{C}_{T,\varphi_f}(0)$, $\mathfrak{C}_{T,D}(0)$ and $\mathfrak{C}_{T,N}(0)$, by
Remark \ref{rem:less_regularity}, Theorem \ref{thm:energy_norm_approx}, and Theorem (\ref{thm:L2_error_eq20}), we have
\begin{equation}
\label{thm:L2_error_eq2}
\vertiii{z - I_hz }_h \lesssim h \frac{\beta^-}{\sqrt{\beta^+}} \sum_{k=1}^{2} \left( | z^- |_{H^{k}(\Omega^-)} + | z^+ |_{H^{k}(\Omega^+)} \right)\lesssim \frac{ \beta^-}{(\beta^+)^{3/2}} h \| u - u_h \|_{L^2(\Omega)}.
\end{equation}
Finally, combining \eqref{thm:L2_error_eq1}, \eqref{thm:L2_error_eq2} and Theorem \ref{thm:energy_error}, we have the estimate
\eqref{thm:L2_error_eq0}.
\end{proof}

\begin{rem}
\label{rem_est_data}
Under the conditions of Theorem \ref{thm_regularity}, applying the regularity result \eqref{u_regularity} to the estimate \eqref{thm:energy_error_eq0} and \eqref{thm:L2_error_eq0}, we have the following estimates in terms of data in the interface problem:
\begin{subequations}
\label{rem_est_data_eq0}
\begin{equation}
\label{rem_est_data_eq01}
\vertiii{ u - u_h }_h \lesssim \frac{\beta^-}{(\beta^+)^{3/2}} h^p ( \min\{\beta^+,\beta^-\} \|J_D\|_{H^{p+3/2}(\Gamma)} + \|J_N\|_{H^{p+1/2}(\Gamma)} + \| f \|_{PH^{p}(\Omega)} ),
\end{equation}
\begin{equation}
\label{rem_est_data_eq01}
\| u - u_h \|_{L^2(\Omega)} \lesssim \frac{(\beta^-)^2}{(\beta^+)^{3}} h^{p+1} ( \min\{\beta^+,\beta^-\} \|J_D\|_{H^{p+3/2}(\Gamma)} + \|J_N\|_{H^{p+1/2}(\Gamma)} + \| f \|_{PH^{p}(\Omega)} ).
\end{equation}
\end{subequations}
\end{rem}

\begin{rem}
Employing arguments similar to those for Remark 6.3 in \cite{2019GuoLin} and applying Remark \ref{rem:less_regularity}, we can obtain the following estimates for the enriched IFE solution provided that
the exact solution $u$ is in $PH^{m+1}(\Omega)$ with $m\geqslant 1$ being an integer and $m\leqslant p$:
\begin{subequations}
\label{est_low_regularity}
\begin{align}
& \vertiii{u-u_h}_h \lesssim \frac{\beta^-}{\sqrt{\beta^+}} h^m \sum_{k=1}^{m+1} \left( | u^- |_{H^{k}(\Omega^-)} + | u^+ |_{H^{k}(\Omega^+)} \right) , \label{est_low_regularity_eq1} \\
& \| u - u_h \|_{L^2(\Omega)}  \lesssim \left( \frac{ \beta^-}{\beta^+} \right)^2 h^{m+1} \sum_{k=1}^{m+1}
\left( | u^- |_{H^{k}(\Omega^-)} + | u^+ |_{H^{k}(\Omega^+)} \right). \label{est_low_regularity_eq2}
\end{align}
\end{subequations}
\end{rem}

\begin{rem}
\label{rem:extended_jumps_removed}
~\\
\begin{itemize}
\item We note that both the IFE functions and the method in \cite{2019GuoLin}
yield optimal IFE solutions for interface problems with homogeneous jump conditions and they can be viewed as a special cases of the IFE functions and the method proposed here.

\item We introduce enrichment IFE functions \eqref{enrich_IFE_fun_DN} and \eqref{enrich_IFE_fun_Lap} to handle the nonhomogeneous jump conditions and the singularity of $f$ across the interface to circumvent
the extended jump conditions \eqref{normal_jump_cond}. Hence, the proposed enriched IFE framework
enables us to solve a larger class of interface  elliptic problems than the existing IFE methods in the literature.
\commentout{
    On the other hand, the proposed method involves penalty terms on the interface which is not the case for existing IFE methods satisfying extended
    interface conditions \cite{2014AdjeridBenromdhaneLin,2018AdjeridRomdhaneLin,2016AdjeridGuoLin}.
    This later issue needs further investigation.
}

\item Another desirable feature of the proposed enriched IFE method is the fact that it allows us to
construct enrichment IFE functions for each nonhomogeneous interface conditions independently. For example,
when $J_D \neq 0$, while $J_N = 0$ and the source term $f$ is smooth across the interface, we only need to construct $\phi_{T,N}$ and set $\phi_{T,D}=\phi_{T,f} = 0$.
\end{itemize}
\end{rem}

\section{Stability of the Enriched IFE Method}
\label{sec:conditioning}

In this section, we address the stabilities in the computations of the proposed IFE method by estimating the spectral condition numbers of both the local linear systems for computing IFE functions and the global linear system for computing enriched IFE solutions. In this section $\bfv$ will denote a column vector with its transpose $\bfv'$. We recall the spectral condition number, the maximum, and the minimum eigenvalues of a symmetric positive definite matrix $\mathbf{ M}$ as
\begin{equation}
\label{spectral_cond}
\kappa(\mathbf{ M}) = \frac{\mu_{\max}(\mathbf{ M})}{\mu_{\min}(\mathbf{ M})}, ~~~ \text{and} ~~~ \mu_{\max}(\mathbf{ M})= \max_{\bfv\in\mathbb{R}^n, \bfv\neq\mathbf{ 0}}\frac{\bfv' \mathbf{ M} \bfv}{\bfv'\bfv}, ~~~ \mu_{\min}(\mathbf{ M})= \min_{\bfv\in\mathbb{R}^n, \bfv\neq\mathbf{ 0}}\frac{\bfv' \mathbf{ M} \bfv}{\bfv'\bfv}.
\end{equation}

\subsection{Local IFE Basis Functions and Condition Numbers}

Let us first describe the procedure for computing the local IFE basis functions and then investigate the conditioning of the resulting algebraic problem. For every interface element $T \in \mathcal{T}_h^i$, let $\{ \zeta_i, \ i=1,2,\ldots,n \}$ be a basis for the polynomial space $\mathbb{P}_p(T_\lambda)$ with $n=(p+1)(p+2)/2$. Thus, for each $v=\sum_{i=1}^n \alpha^+_i \zeta_i$
in $\mathbb{P}_p(T_\lambda)$ with $\bfalpha^+=[\alpha^+_{1},\alpha^+_{2},\ldots,\alpha^+_{n}]'$, there exist vectors $\bfalpha^-=[\alpha^-_{1},\alpha^-_{2},\ldots,\alpha^-_{n}]'$, $\bfalpha^-_r=[\alpha^-_{1,r},\alpha^-_{2,r},\ldots,\alpha^-_{n,r}]'$, $r=D,N,\phi_f$
such that
\begin{subequations} \label{vCT}
\begin{equation}
\label{vCTold}
\mathfrak{C}_T(v)=\sum_{i=1}^n \alpha^-_{i} \zeta_i ~~~~~ \text{and} ~~~~~
 \mathfrak{C}_{T,r}(0)=\sum_{i=1}^n \alpha^-_{i,r} \zeta_i, ~~~ r=D,N,\phi_f.
\end{equation}
By \eqref{cauchy_exten}, \eqref{cauchy_ext_DN} and \eqref{eq:CT_phi_Tf} 
the vectors $\bfalpha^-$ and $\bfalpha^-_r$ are determined by solving the following linear algebraic   systems
\begin{equation}
\label{loc_linear_sys}
\mathbf{ A}_T \bfalpha^-=\mathbf{ B}_T \bfalpha^+ ~~~~ \text{and} ~~~~ \mathbf{ A}_T \bfalpha^-_r=\mathbf{ b}_{T,r}, ~ r=D,N,\phi_f,
\end{equation}
where the matrices $\mathbf{A}_T$, $\mathbf{B}_T$ and the vectors $\mathbf{ b}_{T,r}$, $r=D,N,\phi_f,$ are given by
\begin{equation}
\label{AB_T}
\mathbf{ A}_{T}=\left[ a_{\lambda}(\zeta_i,\zeta_j) \right]^n_{i,j=1}, ~~ \mathbf{ B}_{T}=\left[ b_{\lambda}(\zeta_i,\zeta_j) \right]^n_{i,j=1},
\end{equation}
\begin{equation}
\label{bT_DNf}
\mathbf{ b}_{T,D} = h^{-3}_T \left[ \int_{\Gamma^{\lambda}_T} J_D\zeta_i ds \right]^n_{i=1}, ~~
\mathbf{ b}_{T,N} = \frac{h^{-1}_T}{\beta^-} \left[ \int_{\Gamma^{\lambda}_T} J_N \partial_{\mathbf{ n}}\zeta_i ds\right]^n_{i=1},~~
\mathbf{ b}_{T,\phi_f} =  \left[ \int_{\Gamma^{\lambda}_T} \varphi_f \triangle\zeta_i dX \right]^n_{i=1}.
\end{equation}
\end{subequations}
Since both bilinear forms $a_{\lambda}(\cdot,\cdot)$ and $b_{\lambda}(\cdot,\cdot)$ are coercive under the related energy norms $\vertiii{\cdot}_{a_{\lambda}}$ and $\vertiii{\cdot}_{b_{\lambda}}$, $\mathbf{A}_{T}$ and $\mathbf{B}_{T}$ are symmetric positive definite, and thus the linear systems in \eqref{loc_linear_sys} have unique solutions. In particular, we let $\bfalpha^+=\mathbf{ e}_i$ (the $i$-th unit vector), i.e., $v=\zeta_i$, $i=1,2,...,n$, compute the corresponding $\bfalpha^-$ according to \eqref{loc_linear_sys} and form the polynomial $\mathfrak{C}_{T}(\zeta_i)$ according to \eqref{vCTold}. Thus, a set of basis functions spanning the local IFE space $S_h^p(T),~T \in \mathcal{T}_h^i$ are given by
\begin{equation}
\label{local_IFE_bas_fun}
\phi_{T,i}=
\begin{cases}
      & \zeta_i ~~~~~~~~~ \text{on} ~ T^+, \\
      & \mathfrak{C}_{T}(\zeta_i) ~~~ \text{on}~ T^-,
\end{cases}
~~~~ i=1,2,\ldots,n.
\end{equation}
Now, we proceed to estimate the spectral condition numbers of $\mathbf{A}_T$ and begin by finding upper and lower bounds of the eigenvalues of $\mathbf{A}_T$ and $\mathbf{B}_T$.
\begin{lemma}
\label{lem_AT_eig}
There exist constants $0 < c_{\mu}\leqslant C_{\mu}$ independent of $h_T$, $\beta^{\pm}$ and interface location such that the following estimates hold for every $T \in \mathcal{T}_h^i$:
\begin{subequations}
\label{lem_AT_eig_eq1}
\begin{align}
&   c_{\mu}h^{-2}_T \leqslant  \mu_{\min}(\mathbf{A}_T) \leqslant \mu_{\max}(\mathbf{A}_T) \leqslant C_{\mu} h^{-2}_T,   \label{lem_AT_eig_eq1_A} \\
 &  \frac{\beta^+}{\beta^-}c_{\mu}h^{-2}_T \leqslant  \mu_{\min}(\mathbf{B}_T) \leqslant \mu_{\max}(\mathbf{B}_T) \leqslant C_{\mu} h^{-2}_T. \label{lem_AT_eig_eq1_B}
\end{align}
\end{subequations}
\end{lemma}
\begin{proof}
Let $T$ be an arbitrary interface element and let $T_{\lambda}$ be its associated fictitious element. We consider the estimate for
the matrix $\mathbf{A}_T$ first. For each $\bfv=[v_i]^n_{i=1}\neq0$, letting $v=\sum_{i=1}^nv_i\zeta_i$, we have
\begin{equation}
\label{lem_AT_eig_eq3}
 \frac{\bfv'\mathbf{A}_T\bfv}{\bfv'\bfv} = \frac{a_{\lambda}(v,v)}{ \| v \|^2_{L^2(T^-_{\lambda})}} \cdot \frac{ \| v \|^2_{L^2(T^-_{\lambda})}}{\bfv'\bfv}.
\end{equation}
Let $\mathbf{ M}_T=\left[ (\zeta_i,\zeta_j)_{L^2(T)} \right]^n_{i,j=1}$ be the symmetric mass matrix corresponding to the polynomial basis $\{\zeta_i\}^n_{i=1}$ on $T$. The equivalence of norms $d_1\|v\|^2_{L^2(T)}\leqslant \| v \|^2_{L^2(T^-_{\lambda})} \leqslant d_2\|v\|^2_{L^2(T)}$ from Lemma \ref{lem_norm_equiv} yields
\begin{equation}
\label{lem_AT_eig_eq5}
d_1 \mu_{\min}(\mathbf{ M}_T) \leqslant d_1 \frac{ \bfv' \mathbf{ M}_T \bfv }{\bfv'\bfv} \leqslant \frac{ \| v \|^2_{L^2(T^-_{\lambda})}}{\bfv'\bfv} \leqslant d_2 \frac{ \bfv' \mathbf{ M}_T \bfv }{\bfv'\bfv}  \leqslant d_2 \mu_{\max}(\mathbf{ M}_T).
\end{equation}
We substitute the equivalence of norms $c_1 h^{-4}_T \| v \|^2_{L^2(T^-_{\lambda})} \leqslant a_{\lambda}(v,v) \leqslant c_2 h^{-4}_T \| v \|^2_{L^2(T^-_{\lambda})}$ from Lemma \ref{lem_norm_equiv} and \eqref{lem_AT_eig_eq5} into \eqref{lem_AT_eig_eq3}. Then, we apply the bounds $\mu_{\min}(\mathbf{ M}_T)$ and $\mu_{\max}(\mathbf{ M}_T)$ from Lemma A.1 in \cite{2006ErnGuermond} with shape regularity of the mesh to have
\begin{equation}
\label{lem_AT_eig_eq6}
c_{\mu} h^{-2}_T \leqslant d_1c_1\mu_{\min}(\mathbf{ M}_T)h^{-4}_T \leqslant \frac{\bfv'\mathbf{A}_T\bfv}{\bfv'\bfv} \leqslant d_2c_2 \mu_{\max}(\mathbf{ M}_T)h^{-4}_T \leqslant  C_{\mu} h^{-2}_T,
\end{equation}
which yields \eqref{lem_AT_eig_eq1_A} according to \eqref{spectral_cond}. The constants $c_1$, $c_2$, $d_1$ and $d_2$ in the discussions above are independent of both the interface location and element size $h_T$ but they depend
on the scaling factor $\lambda$ and the degree $p$ \cite{2019GuoLin}. The estimate for the matrix $\mathbf{B}_T$ given in
\eqref{lem_AT_eig_eq1_B} follows from the same line of reasoning plus the usage of \eqref{cauchy_norm_relation}.
\end{proof}

In the next theorem we state a stability result for the local IFE problem.
\begin{thm}
\label{thm_condition_AT}
There exists a constant $C_{\kappa}$ independent of $h_T$, $\beta^{\pm}$ and the interface location such that
\begin{equation}
\label{thm_condition_AT_eq0}
\kappa(\mathbf{A}_T) \leqslant C_{\kappa}, ~\forall T \in \mathcal{T}_h^i.
\end{equation}
\end{thm}
\begin{proof}
The proof immediately follows from Lemma \ref{lem_AT_eig}.
\end{proof}

\begin{rem}
The scaling factor $\lambda$ is critical to improving the conditioning of $\mathbf{A}_T$ since the geometric analysis in \cite{2019GuoLin} shows as $\lambda\rightarrow1$, i.e., $T_{\lambda}$ reduces to the original element $T$ and the constants in the norm equivalence may blow up when the interface cuts some elements with extremely small-cut subelements, and consequently, the generic constants in the estimates given by Lemma \ref{lem_AT_eig} and Theorem
\ref{thm_condition_AT} can be extremely large. Therefore, $\lambda$ acts like a regularization parameter for solving local Cauchy problems with better conditioning.
\end{rem}

\begin{rem}
\label{rem_mass_cond}
By a similar argument we can further show that $\kappa(\mathbf{B}_T) \leqslant C_{\kappa}\frac{\beta^-}{\beta^+}$. In addition, Lemma \ref{lem_AT_eig} yields $\kappa(\mathbf{A}_T) \lesssim \kappa(\mathbf{ M}_T)$ and $\kappa(\mathbf{B}_T) \lesssim \frac{\beta^-}{\beta^+} \kappa(\mathbf{ M}_T)$.
\end{rem}


\subsection{Stability of the Proposed IFE Method}

Here, we establish upper bounds of the spectral condition number of the stiffness matrix corresponding to
the symmetric enriched IFE method, \emph{i.,e.}, $\epsilon_0=\epsilon_1=-1$ in the bilinear form \eqref{weak_form_2}. The IFE space $S^p_h(\Omega)$ is the set of piecewise polynomials spanned by the standard nodal Lagrange basis functions on
non-interface elements and the IFE shape functions $\phi_{T,i}$ defined by \eqref{local_IFE_bas_fun}
on interface elements where we assume the underling functions $\{\zeta_i\}^n_{i=1}$ are also the standard Lagrange nodal polynomials.
By the standard finite element procedure we use the local shape functions for $S_h^p(T)$ to construct $\{ \psi_1, \psi_2, \ldots, \psi_{N} \}$ as a global basis for $S^p_h(\Omega)$. Thus, the enriched IFE scheme \eqref{enr_IFE_method_2} yields
\begin{equation}
\label{Kh_mat}
\mathbf{K}_h=\left[a_h(\psi_i,\psi_j)\right]^N_{i,j=1} ~~~~~ \text{and} ~~~~~ F_h=\left[ L_f(\psi_i) - a_h(\psi_i,\Phi_h) \right]^N_{i=1},
\end{equation}
and the coefficient vector $\mathbf{ u}_h$  of the enriched IFE solution $u_h$  is determined by the linear system
\begin{equation}
\label{uh_solu}
\mathbf{K}_h \mathbf{ u}_h = F_h.
\end{equation}
We first state and prove a few preliminary results that help estimate $\kappa(\mathbf{K}_h)$.

\begin{lemma}
\label{alpha_bound}
Let $\bfalpha^-$ and $\bfalpha^+$ be defined by \eqref{vCT} with any polynomial basis $\{ \zeta_i \}_{i=1}^n$, then
\begin{equation}
\label{alpha_bound_eq0}
\frac{c_{\mu}}{C_{\mu}} \frac{\beta^+}{\beta^-} \| \bfalpha^+ \|_2   \leqslant \| \bfalpha^- \|_{2} \leqslant \frac{C_{\mu}}{c_{\mu}} \| \bfalpha^+ \|_2,
\end{equation}
where $c_{\mu}$ and $C_{\mu}$ inherit from Lemma \ref{lem_AT_eig}.
\end{lemma}
\begin{proof}
Using \eqref{loc_linear_sys} and Lemma \ref{lem_AT_eig}, we have
\begin{equation}
\label{alpha_bound_eq1}
\| \bfalpha^+ \|_2 = \| \mathbf{B}^{-1}_T \mathbf{A}_T \bfalpha^- \|_2  \leqslant \| \mathbf{B}^{-1}_T \|_2 \| \mathbf{A}_T \|_2 \| \bfalpha^- \|_2 \leqslant \frac{\mu_{\max}(\mathbf{A}_T)}{\mu_{\min}(\mathbf{B}_T)} \| \bfalpha^- \|_2 \leqslant \frac{C_{\mu}}{c_{\mu}} \frac{\beta^-}{\beta^+} \| \bfalpha^- \|_2
\end{equation}
which gives the left inequality of \eqref{alpha_bound_eq0}. The inequality on the right can be proved similarly.
\end{proof}

\begin{lemma}
\label{lem_C_stability}
There exist constants $c_b>0$ and $C_b>0$ such that the following estimates hold for all interface element $T \in \mathcal{T}_h^i$:
\begin{align}
   c_b \frac{\beta^+}{\beta^-} \| v \|_{L^2(T^s)}  \leqslant \| \mathfrak{C}_T(v) \|_{L^2(T^s)} \leqslant  C_b \| v \|_{L^2(T^s)}, ~~~~s=\pm, ~~\forall v\in \mathbb{P}_p(T_\lambda). \label{lem_C_stability_eq1}
\end{align}
\commentout{
Consequently, they imply that the Cauchy mapping $\mathfrak{C}_T$ and its inverse $\mathfrak{C}^{-1}_T$ on every interface element are bounded in terms of the $L^2$-norm
on $\mathbb{P}_p(T^s)$.
}
\end{lemma}
\begin{proof}
First, we prove \eqref{lem_C_stability_eq1} for $s = -$ by specifically selecting a basis $\{ \zeta_i \}_{i=1}^n$ of $\mathbb{P}_p(T_\lambda)$ that is orthogonal on $T^-$ with respect to the $L^2$ inner product.
Let $\bfalpha^{+}$ and $\bfalpha^-$ be the coordinates of $v$ and $\mathfrak{C}_T(v)$, respectively, obtained from  \eqref{vCT}. Then,
Lemma \ref{alpha_bound} yields
\begin{equation}
\label{lem_C_stability_eq2}
\| \mathfrak{C}_T(v) \|_{L^2(T^-)}  = \| \bfalpha^- \|_2 \leqslant \frac{C_{\mu}}{c_{\mu}} \| \bfalpha^+ \|_2 = \frac{C_{\mu}}{c_{\mu}} \| v \|_{L^2(T^-)},
\end{equation}
which establishes the second inequality in \eqref{lem_C_stability_eq1} for $s = -$.
The first inequality in \eqref{lem_C_stability_eq1} for $s = -$ can be proved similarly. Similar arguments can be applied to prove
\eqref{lem_C_stability_eq1} for $s = +$ by selecting a basis $\{ \zeta_i \}_{i=1}^n$ of $\mathbb{P}_p(T_\lambda)$ which is orthogonal on $T^+$ with respect to the $L^2$ inner product.
\end{proof}

\begin{figure}[H]
\centering
    \includegraphics[width=1.5in]{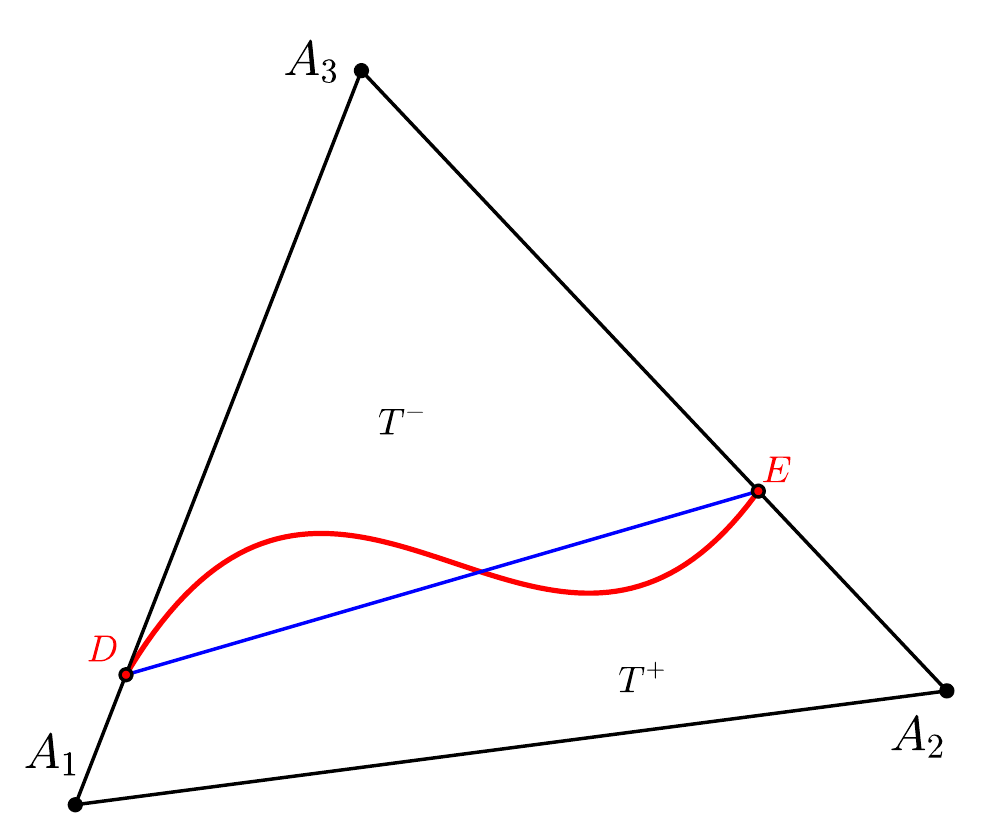}
  \caption{Illustration of an interface element.}
  \label{fig:trace_inequa} 
\end{figure}

As an important consequence, Lemma \ref{lem_C_stability} above actually establishes the stability of the Cauchy mapping $\mathfrak{C}_T$ and $\mathfrak{C}^{-1}_T$ in the $L^2$ norm on $\mathbb{P}_p(T^s)$ which is  a key ingredient for
establishing the following trace inequalities of IFE functions.

\begin{lemma}
\label{thm_L2_trace}
There exists a constant $C_t > 0$ such that, for every interface element $T \in \mathcal{T}_h^i$ and its edge $e$,
the following estimates hold:
\begin{equation}
\label{thm_L2_trace_eq0}
h_T^{1/2} \| \phi_T \|_{L^2(e)} \leqslant C_t \frac{\beta^-}{\beta^+} \| \phi_T \|_{L^2(T)},~~ h^{1/2}_T \|  \phi_T^s \|_{L^2(\Gamma_T)} \leqslant C_t \frac{\beta^-}{\beta^+} \| \phi_T \|_{L^2(T)}, ~s=+,-,~ \forall \phi_T \in S^p_h(T).
\end{equation}
\end{lemma}
\begin{proof}
Without loss of generality, consider an interface element $T$ as shown in Figure \ref{fig:trace_inequa}.
In order to prove \eqref{thm_L2_trace_eq0} on the edge $e=A_1A_3$
we let $T^-$ and $T^+$, respectively, be the curved-edge triangle $A_3DE$ and quadrilateral $A_1A_2ED$, and let $\widetilde{T}^-$ and $\widetilde{T}^+$, respectively, be the straight-edge triangle $A_3DE$ and quadrilateral $A_1A_2ED$.

By (\ref{loc_IFE_space_inter}) every IFE function $\phi_T \in S_h^p(T)$ on an interface element
$T$ (with the fictitious element $T_\lambda$) can be written as
$\phi_T = (\phi_T^+,\phi_T^-)$ where $\phi_T^+ = z_p$ on $T^+$ and $\phi_T^-=\mathfrak{C}_T(z_p)$
on $T^-$ with $z_p, \mathfrak{C}_T(z_p) \in \mathbb{P}_p(T_\lambda)$.

If $E$ and $D$ are such that $|A_3D|\geqslant |A_3A_1|/2$ and $|A_3E|\geqslant |A_3A_2|/2$,
then the standard trace inequality for polynomials \cite{2003WarburtonHesthaven} and the equivalence of norms (3.11) in Lemma 3.6 \cite{2019GuoLin} lead to
\begin{equation}
\label{thm_L2_trace_eq1}
\| \phi^-_T \|_{L^2(A_3D)} \lesssim h^{-1/2}_T \| \mathfrak{C}_T(z_p) \|_{L^2(\widetilde{T}^-)} \lesssim h^{-1/2}_T \| \mathfrak{C}_T(z_p) \|_{L^2(T^-)}= h^{-1/2}_T \| \phi^-_T \|_{L^2(T^-)}.
\end{equation}
For the sub-edge $DA_1$, the standard trace inequality for polynomials \cite{2003WarburtonHesthaven} and the stability \eqref{lem_C_stability_eq1} yield
\begin{equation}
\begin{split}
\label{thm_L2_trace_eq2}
\| \phi^+_T \|_{L^2(DA_1)} & \lesssim h^{-1/2}_T \|  z_p \|_{L^2(T)} \lesssim h^{-1/2}_T \left( \| z_p \|_{L^2(T^+)} + \| \mathfrak{C}^{-1}_T(\mathfrak{C}_T(z_p)) \|_{L^2(T^-)} \right) \\
& \lesssim h^{-1/2}_T \left( \| z_p \|_{L^2(T^+)} + \frac{\beta^-}{\beta^+} \| \mathfrak{C}_T(z_p) \|_{L^2(T^-)} \right) \lesssim h^{-1/2}_T \frac{\beta^-}{\beta^+} \| \phi_T \|_{L^2(T)}.
\end{split}
\end{equation}
Combining \eqref{thm_L2_trace_eq1} and \eqref{thm_L2_trace_eq2} leads to the first estimate in \eqref{thm_L2_trace_eq0}. On the other hand, when $|A_3D|\leqslant |A_3A_1|/2$ or $|A_3E|\leqslant |A_3A_2|/2$, by similar argument to \eqref{thm_L2_trace_eq2}, we have
\begin{equation}
\begin{split}
\label{thm_L2_trace_eq3}
\| \phi^-_T \|_{L^2(A_3D)} & \lesssim h^{-1/2}_T \|  \mathfrak{C}_T(z_p) \|_{L^2(T)} \lesssim h^{-1/2}_T \left( \| \mathfrak{C}_T(z_p) \|_{L^2(T^-)} + \| \mathfrak{C}_T(z_p) \|_{L^2(T^+)} \right) \\
& \lesssim h^{-1/2}_T \left( \| \mathfrak{C}_T(z_p) \|_{L^2(T^-)} +  \| z_p \|_{L^2(T^+)} \right) \lesssim h^{-1/2}_T  \| \phi_T \|_{L^2(T)}.
\end{split}
\end{equation}
For the estimate on $DA_1$, we use an argument similar to \eqref{thm_L2_trace_eq1} to get
\begin{equation}
\label{thm_L2_trace_eq4}
\| \phi^+_T \|_{L^2(DA_1)} \lesssim h^{-1/2}_T \| z_p \|_{L^2(\triangle A_1A_2D)} \lesssim h^{-1/2}_T \|  z_p \|_{L^2(\widetilde{T}^+)} \lesssim h^{-1/2}_T \| \phi^+_T \|_{L^2(T^+)},
\end{equation}
where in the second inequality we used the fact that $\triangle A_1 A_2 D\subset \widetilde{T}^+$
and in the last inequality we used the equivalence of norms given by (3.12) in Lemma 3.6 \cite{2019GuoLin}.
Combining \eqref{thm_L2_trace_eq3} and \eqref{thm_L2_trace_eq4} leads to the first estimate in \eqref{thm_L2_trace_eq0}.

For the trace inequality on the interface $\Gamma_T$, we start with the case that $|A_3D|\geqslant |A_3A_1|/2$ and $|A_3E|\geqslant |A_3A_2|/2$. For $-$, using the trace inequality (3.14b) in \cite{2019GuoLin} and the equivalence of norms (3.11) in \cite{2019GuoLin}, we have
\begin{equation}
\label{thm_L2_trace_eq5}
\| \phi_T^- \|_{L^2(\Gamma_T)} \leqslant \| \phi_T^- \|_{L^2(\Gamma^{\lambda}_T)} \lesssim h^{-1/2}_T \| \phi_T^- \|_{L^2(T^-_{\lambda})} \lesssim h^{-1/2}_T \| \phi_T^- \|_{L^2(T^-)}.
\end{equation}
For $+$, using the trace inequality (3.14b) in \cite{2019GuoLin}, the equivalence of norms \eqref{lem_norm_equiv_eq2} and the stability \eqref{lem_C_stability_eq1}, we obtain
\begin{equation}
\begin{split}
\label{thm_L2_trace_eq6}
\| \phi_T^+ \|_{L^2(\Gamma_T)} & \leqslant \| \phi_T^+ \|_{L^2(\Gamma^{\lambda}_T)} \lesssim h^{-1/2}_T \| \phi_T^+ \|_{L^2(T^+_{\lambda})} \lesssim h^{-1/2}_T \| z_p \|_{L^2(T_{\lambda})} \lesssim h^{-1/2}_T \| z_p \|_{L^2(T)} \\
&  \lesssim h^{-1/2}_T \left( \| z_p \|_{L^2(T^+)} + \| \mathfrak{C}^{-1}_T(\mathfrak{C}_T(z_p)) \|_{L^2(T^-)} \right) \\
& \lesssim h^{-1/2}_T \left( \| z_p \|_{L^2(T^+)} + \frac{\beta^-}{\beta^+} \| \mathfrak{C}_T(z_p) \|_{L^2(T^-)} \right) \lesssim h^{-1/2}_T \frac{\beta^-}{\beta^+} \| \phi_T \|_{L^2(T)}.
\end{split}
\end{equation}
Combining \eqref{thm_L2_trace_eq5} and \eqref{thm_L2_trace_eq6}, we have the second inequality in \eqref{thm_L2_trace_eq0}. On the other hand, when $|A_3D|\leqslant |A_3A_1|/2$ or $|A_3E|\leqslant |A_3A_2|/2$, we follow the steps used to derive \eqref{thm_L2_trace_eq6}
to write
\begin{equation}
\begin{split}
\label{thm_L2_trace_eq7}
\| \phi_T^- \|_{L^2(\Gamma_T)} & \leqslant \| \phi_T^- \|_{L^2(\Gamma^{\lambda}_T)} \lesssim h^{-1/2}_T \| \phi_T^- \|_{L^2(T^-_{\lambda})} \lesssim h^{-1/2}_T \|  \mathfrak{C}_T(z_p) \|_{L^2(T_{\lambda})} \lesssim h^{-1/2}_T \| \mathfrak{C}_T(z_p) \|_{L^2(T)} \\
&  \lesssim h^{-1/2}_T \left( \| \mathfrak{C}_T(z_p) \|_{L^2(T^+)} + \| \mathfrak{C}_T(z_p) \|_{L^2(T^-)} \right) \\
& \lesssim h^{-1/2}_T \left( \| z_p \|_{L^2(T^+)} + \| \mathfrak{C}_T(z_p) \|_{L^2(T^-)} \right) \lesssim h^{-1/2}_T  \| \phi_T \|_{L^2(T)}.
\end{split}
\end{equation}
Also, following the same line of reasoning  used for \eqref{thm_L2_trace_eq5}, we apply
the equivalence of norms (3.12) in \cite{2019GuoLin} to have
\begin{equation}
\label{thm_L2_trace_eq8}
\| \phi_T^+ \|_{L^2(\Gamma_T)} \leqslant \| \phi_T^+ \|_{L^2(\Gamma^{\lambda}_T)} \lesssim h^{-1/2}_T \| \phi_T^+ \|_{L^2(T^+_{\lambda})} \lesssim h^{-1/2}_T \| \phi_T^+ \|_{L^2(T^+)}.
\end{equation}
Combining \eqref{thm_L2_trace_eq7} and \eqref{thm_L2_trace_eq8} completes the proof.
\end{proof}

By the stability result \eqref{lem_C_stability_eq1}, we can also establish the following inverse inequalities for IFE functions.

\begin{lemma}
\label{lemma_IFE_inverse}
There exists a constant $C_i >0$ such that the following estimate holds for every interface element
$T \in \mathcal{T}_h^i$:
\begin{equation}
\label{lemma_IFE_inverse_eq0}
\| \nabla \phi_T \|_{L^2(T)} \leqslant C_{i} \frac{\beta^-}{\beta^+} h^{-1}_T \| \phi_T \|_{L^2(T)},
~~~~~ \forall \phi_T \in S^p_h(T).
\end{equation}
\end{lemma}
\begin{proof}
Without loss of generality, again, we consider the interface element $T$ as shown in Figure \ref{fig:trace_inequa}.
If $|A_3D|\geqslant |A_3A_1|/2$ and $|A_3E|\geqslant |A_3A_2|/2$ and $\phi_T=(z_p, \mathfrak{C}_T(z_p))$ for $z_p$, $\mathfrak{C}_T(z_p) \in \mathbb{P}_p(T_\lambda)$, then
the equivalence of norms (3.11) in Lemma 3.6 in \cite{2019GuoLin} gives
\begin{equation}
\label{lemma_IFE_inverse_eq1}
\| \nabla \phi^-_T \|_{L^2(T^-)} \leqslant \| \nabla \mathfrak{C}_T(z_p) \|_{L^2(T)} \lesssim h^{-1}_T \| \mathfrak{C}_T(z_p) \|_{L^2(T)} \lesssim h^{-1}_T \| \phi^-_T \|_{L^2(T^-)}.
 \end{equation}
 Applying the stability result \eqref{lem_C_stability_eq1} yields
 \begin{equation}
 \begin{split}
\label{lemma_IFE_inverse_eq2}
\| \nabla \phi^+_T \|_{L^2(T^+)} \leqslant \| \nabla z_p \|_{L^2(T)} &  \lesssim h^{-1}_T  \| z_p \|_{L^2(T)}
\lesssim h^{-1}_T \left(  \| z_p \|_{L^2(T^+)} +  \| \mathfrak{C}^{-1}_T(\mathfrak{C}_T(z_p)) \|_{L^2(T^-)}  \right) \\
& \lesssim h^{-1}_T \left(  \| \phi^+_T \|_{L^2(T^+)} + \frac{\beta^-}{\beta^+} \| \phi^-_T \|_{L^2(T^-)}  \right) \lesssim \frac{\beta^-}{\beta^+} h^{-1}_T \| \phi_T \|_{L^2(T)}.
\end{split}
\end{equation}
Combining (\ref{lemma_IFE_inverse_eq1}) and (\ref{lemma_IFE_inverse_eq2}) yields
(\ref{lemma_IFE_inverse_eq0}).

For the case $|A_3D|\leqslant |A_3A_1|/2$ or $|A_3E|\leqslant |A_3A_2|/2$, we follow the same reasoning used to establish \eqref{lemma_IFE_inverse_eq2} and apply \eqref{lem_C_stability_eq1} to write
 \begin{equation}
 \begin{split}
\label{lemma_IFE_inverse_eq3}
\| \nabla \phi^-_T \|_{L^2(T^-)} &\leqslant \| \nabla \mathfrak{C}_T(z_p) \|_{L^2(T)} \lesssim h^{-1}_T  \| \mathfrak{C}_T(z_p) \|_{L^2(T)} \\
&\lesssim h^{-1}_T \left(  \| \mathfrak{C}_T(z_p) \|_{L^2(T^+)} +  \| \mathfrak{C}_T(z_p) \|_{L^2(T^-)}  \right) \\
& \lesssim h^{-1}_T \left(  \| z_p \|_{L^2(T^+)} +\| \mathfrak{C}_T(z_p) \|_{L^2(T^-)}  \right) \lesssim  h^{-1}_T \| \phi_T \|_{L^2(T)}.
\end{split}
\end{equation}
On $T^+$, we use the equivalence of norms (3.12) in Lemma 3.6 \cite{2019GuoLin} to obtain
\begin{equation}
\label{lemma_IFE_inverse_eq4}
\| \nabla \phi^+_T \|_{L^2(T^+)} \leqslant \| \nabla z_p \|_{L^2(T)} \lesssim h^{-1}_T \| z_p \|_{L^2(T)} \lesssim h^{-1}_T \| \phi^+_T \|_{L^2(T^+)}.
\end{equation}
Combining (\ref{lemma_IFE_inverse_eq3}) and (\ref{lemma_IFE_inverse_eq4}) completes the proof.
\end{proof}

In the following lemma we prove a discrete Poincar\'e inequality.
\begin{lemma}
\label{lem_poc}
For all $v\in V_h$, there holds
\begin{equation}
\label{lem_poc_eq0}
\| v \|_{L^2(\Omega)} \lesssim \|  \nabla v \|_{L^2(\Omega)}  + \sum_{e\in\mathcal{E}^i_h} |e|^{-1/2} \|[v]_e\|_{L^2(e)} +  \sum_{T\in\mathcal{T}^i_h} h^{-1/2}_T \| [v]_{\Gamma} \|_{L^2(\Gamma_T)}  .
\end{equation}
\end{lemma}
\begin{proof}
We apply a standard argument used for Lemma 2.1 in \cite{1982Arnold}. Given a $v\in V_h$,
we define an auxiliary function $z=(z^+,z^-) \in PH^2(\Omega)$ as the solution of the
interface problem \eqref{model} satisfying homogeneous jump conditions across the interface with $f=v$. By Theorem  \ref{thm_regularity} we have
\begin{equation}
\label{lem_poc_eq1}
 \sum_{k=1}^{2} \left( \beta^-  |z^-|_{H^k(\Omega^-)} +  \beta^+ |z^+|_{H^k(\Omega^+)} \right) \lesssim \| v \|_{L^{2}(\Omega)}.
\end{equation}
We let $\mathfrak{E}^s(z) \in H^2(\Omega)$ be the Sobolev extensions of $z$, $s=\pm$ and  use H\"older's inequality to write
\begin{equation}
\begin{split}
\label{lem_poc_eq3}
\| v \|^2_{L^2(\Omega)} &= \int_{\Omega} \beta \nabla z \cdot \nabla v dX - \sum_{e\in\mathcal{E}^i_h} \int_e \{ \beta \nabla z\cdot \mathbf{ n} \}_e [v]_e ds - \sum_{T\in\mathcal{T}^i_h} \int_{\Gamma_T} \{ \beta \nabla z\cdot \mathbf{ n} \}_{\Gamma} [v]_{\Gamma} ds \\
&\leqslant \left( \| \ \nabla v \|^2_{L^2(\Omega)} +  \sum_{e\in\mathcal{E}^i_h} |e|^{-1} \| [v]_e \|^2_{L^2(e)} +   \sum_{T\in\mathcal{T}^i_h} h^{-1}_T \| [v]_{\Gamma} \|^2_{L^2(\Gamma_T)} \right)^{1/2} \\
& \cdot  \left( \| \beta \nabla z \|^2_{L^2(\Omega)} + \sum_{e\in\mathcal{E}^i_h}  |e| \| \{ \beta \nabla z\cdot \mathbf{ n} \}_e \|^2_{L^2(e)} + \sum_{T\in\mathcal{T}^i_h} h_T \|  \{ \beta \nabla z\cdot \mathbf{ n} \}_{\Gamma} \|^2_{L^2(\Gamma_T)} \right)^{1/2}.
\end{split}
\end{equation}
Using the facts that $z^+=z^-$ and $\beta^+\nabla z^+\cdot\mathbf{ n} = \beta^-\nabla z^-\cdot\mathbf{ n}$ on $\Gamma$ and that $h < C$ depending only on $\Omega$, we can apply the trace inequality on $\Omega^-$ to obtain
\begin{equation}
\label{lem_poc_eq4}
\sum_{T\in\mathcal{T}^i_h} h^{1/2}_T  \|  \{ \beta \nabla z\cdot \mathbf{ n} \}_{\Gamma} \|_{L^2(\Gamma_T)} \leqslant h^{1/2}  \|   \beta^- \nabla z^-\cdot \mathbf{ n}  \|_{L^2(\Gamma)} \lesssim   \beta^{-}\left( | z^- |_{H^1(\Omega^-)} +  | z^- |_{H^2(\Omega^-)} \right).
\end{equation}
In addition, given a non-interface edge $e\in\mathcal{E}^i_h$, there exists a non-interface element $T$ in, say $\Omega^-$, containing $e$, then the trace inequality on $e$ and $T$ and the mesh regularity yield
\begin{align}
\| |e|^{1/2} \{ \beta \nabla z\cdot \mathbf{ n} \}_{e} \|_{L^2(e)} & =  |e|^{1/2} \|  \beta^{-} \nabla z^{-}\cdot \mathbf{ n} \|_{L^2(e)}   \lesssim  \beta^{-}\left( | z^{-} |_{H^1(T)} + h_T | z^{-} |_{H^2(T)} \right). \label{lem_poc_eq5}
\end{align}
If $e \in \mathcal{E}_h^i$ is an interface edge, then, by letting $T$ be an interface element containing $e$, using the Sobolev extensions $\mathfrak{E}^\pm (z)$ and the mesh regularity, we obtain
\begin{equation}
\begin{split}
\label{lem_poc_eq6}
|e|^{1/2}  \| \{ \beta \nabla z\cdot \mathbf{ n} \}_{e} \|_{L^2(e)} &  \lesssim \sum_{s=\pm} \beta^s  \left( | \mathfrak{E}^s(z) |_{H^1(T)} + h_T | \mathfrak{E}^s(z) |_{H^2(T)} \right).
\end{split}
\end{equation}
We first sum \eqref{lem_poc_eq5} and \eqref{lem_poc_eq6} over all the edges in $\mathcal{E}^i_h$ to obtain an upper bound for the jump terms on all edges in $\mathcal{E}^i_h$. Then, we combine the resulting bound with \eqref{lem_poc_eq4} and apply \eqref{lem_poc_eq1} to
obtain
\begin{equation}
\begin{split}
\label{lem_poc_eq7}
& \left( \| \beta \nabla z \|^2_{L^2(\Omega)} +  \sum_{e\in\mathcal{E}^i_h} |e| \| \{ \beta \nabla z\cdot \mathbf{ n} \}_e \|^2_{L^2(e)} +  \sum_{T\in\mathcal{T}^i_h} h_T \| \{ \beta \nabla z\cdot \mathbf{ n} \}_{\Gamma} \|^2_{L^2(\Gamma_T)} \right)^{1/2} \\
 \lesssim &  \sum_{s=\pm} \beta^s  (| z^s |_{H^1(\Omega^s)} + | z^s |_{H^2(\Omega^s)} ) \lesssim \| v \|_{L^2(\Omega)},
\end{split}
\end{equation}
where we also applied the boundedness of Soblev extensions \eqref{soblev_bound}. Finally, substituting \eqref{lem_poc_eq7} into \eqref{lem_poc_eq3} leads to the estimate in \eqref{lem_poc_eq0}.
\end{proof}

Now, we are ready to state and prove an equivalence between $\|\cdot\|_{L^2(\Omega)}$ and $\vertiii{\cdot}_h$ norms on $S^p_h(\Omega)$.

\begin{lemma}
\label{lem_glob_energy_norm_equiv}
For every  $v_h\in S^p_h(\Omega)$, there holds
\begin{equation}
\label{rem_L2_energy_equiv_eq}
(\beta^+)^{1/2} \| v_h \|_{L^2(\Omega)}\lesssim  \vertiii{v_h}_h \lesssim   \frac{(\beta^-)^{2}}{(\beta^+)^{3/2}} h^{-1} \| v_h \|_{L^2(\Omega)}.
\end{equation}
\end{lemma}
\begin{proof}
We note that the first inequality in \eqref{rem_L2_energy_equiv_eq} readily follows from \eqref{lem_poc_eq0};
hence, we proceed to prove the second inequality in \eqref{rem_L2_energy_equiv_eq} which is, in fact, a discrete inverse inequality. For each $e\in\mathcal{E}^i_h$, we let $T^1_e$ and $T^2_e$ be the two elements sharing $e$. First, we recall the inequalities (6.2) and (6.3) in \cite{2019GuoLin} for every $v_h \in S_h^p(\Omega)$
 \begin{align}
  & \sum_{e\in\mathcal{E}^i_h} \frac{|e|}{\sigma^0 \gamma } \int_e ( \{ \beta \nabla v_h \cdot \mathbf{ n}_e \}_e)^2 ds  \lesssim \sum_{e\in\mathcal{E}^i_h}  ( \| \sqrt{\beta} \nabla v_h \|^2_{L^2(T^1_e)} +  \| \sqrt{\beta} \nabla v_h \|^2_{L^2(T^2_e)} ) \lesssim \beta^- \| \nabla v_h \|^2_{L^2(\Omega)} , \label{lem_glob_energy_norm_equiv_eq1} \\
 & \sum_{T\in\mathcal{T}^i_h} \frac{h_T}{\sigma^1 \gamma } \int_{\Gamma_T} ( \{ \beta \nabla v_h \cdot \mathbf{ n}_{\Gamma} \}_{\Gamma})^2 ds  \lesssim \sum_{T\in\mathcal{T}^i_h}  \| \sqrt{\beta} \nabla v_h \|^2_{L^2(T)} \lesssim \beta^- \| \nabla v_h \|^2_{L^2(\Omega)}, \label{lem_glob_energy_norm_equiv_eq2}
\end{align}
which are actually the consequence of the trace inequalities in (5.38) of \cite{2019GuoLin}. In addition, applying \eqref{thm_L2_trace_eq0} and the mesh quasi-uniformity  and regularity assumptions
with $\gamma=(\beta^-)^2/\beta^+$, we have
\begin{align}
    & \sum_{e\in\mathcal{E}^i_h} \frac{\sigma^0 \gamma }{|e|} \int_e [v_h]^2_e ds \lesssim \sum_{e\in\mathcal{E}^i_h} |e|^{-2}  \frac{(\beta^-)^4}{(\beta^+)^3} \left( \|v_h\|^2_{L^2(T^1_e)}  +  \|v_h\|^2_{L^2(T^2_e)} \right) \lesssim h^{-2} \frac{(\beta^-)^4}{(\beta^+)^3} \| v_h \|^2_{L^2(\Omega)},  \label{lem_glob_energy_norm_equiv_eq3} \\
    &  \sum_{T\in\mathcal{T}^i_h} \frac{\sigma^1 \gamma}{h_T} \int_{\Gamma_T} [v_h]^2_{\Gamma} ds \lesssim \sum_{T\in\mathcal{T}^i_h} h^{-2}_T \frac{(\beta^-)^4}{(\beta^+)^3}  \|v_h\|^2_{L^2(T)} \lesssim h^{-2} \frac{(\beta^-)^4}{(\beta^+)^3}  \|v_h\|^2_{L^2(\Omega)}. \label{lem_glob_energy_norm_equiv_eq4}
\end{align}
Then, substituting \eqref{lem_glob_energy_norm_equiv_eq1}-\eqref{lem_glob_energy_norm_equiv_eq4} into \eqref{energy_norm_aux} and applying the standard inverse inequalities on non-interface elements and the inverse inequalities in Lemma \ref{lemma_IFE_inverse} on interface elements, we obtain the
second inequality in \eqref{rem_L2_energy_equiv_eq}.
\end{proof}
%
In the following lemma, we estimate the $L^2$ norm of functions in $S_h^p(\Omega)$ in terms of their coordinates.
\begin{lemma}
\label{IFE_mass}
For every $v_h\in S^p_h(\Omega)$ having coordinates $\bfv$ with respect to the global basis functions $\{ \psi_i \}_{i=1}^N$, there holds
\begin{equation}
\label{IFE_mass_eq0}
\frac{(\beta^+)^2}{(\beta^-)^2} h^2 \bfv'\bfv \lesssim \|v_h\|^2_{L^2(\Omega)}  \lesssim   h^2 \bfv'\bfv.
\end{equation}
\end{lemma}
\begin{proof}
Let $T$ be an element and its local IFE space $S_h^p(T)$ be equipped with the basis functions $\{\phi_{T,i}\}_{i=1}^n$, and let $\bfv_T$
be the coordinates of $\phi_T \in S^p_h(T)$. We only need to prove the local version of \eqref{IFE_mass_eq0}:
\begin{equation}
\label{IFE_mass_eq1}
\frac{(\beta^+)^2}{(\beta^-)^2} h^2 \bfv'_T\bfv_T  \lesssim \| \phi_T \|^2_{L^2(T)} \lesssim  h^2_{T} \bfv'_T \bfv_T.
\end{equation}
When $T$ is a non-interface element, then $\{\phi_{T,i}\}_{i=1}^n$ are Lagrange polynomials; hence, by Lemma A.1 in \cite{2006ErnGuermond}, the following standard result holds
\begin{equation}
\label{IFE_mass_eq1_1}
h^2_{T} \bfv'_T \bfv_T  \lesssim  \| \phi_T \|^2_{L^2(T)} \lesssim  h^2_{T} \bfv'_T \bfv_T,
\end{equation}
which leads to \eqref{IFE_mass_eq1} since $\beta^+/\beta^- \leq 1$.

When $T$ is an interface element, by the construction approach of local IFE basis functions on interface elements \eqref{local_IFE_bas_fun}, each function in $S^p(T)$ can be written as
$\phi_T = (\phi_T^+,\phi_T^-)$ where $\phi_T^+ = z_p$ on $T^+$, $\phi_T^-=\mathfrak{C}_T(z_p)$
on $T^-$ with  $z_p, \mathfrak{C}_T(z_p) \in \mathbb{P}_p(T_\lambda)$. Furthermore, we have $\phi^+_{T,i}=\phi_{T,i}|_{T^+}=\zeta_i$ on $T^+$ and $\phi^-_{T,i}=\phi_{T,i}|_{T^-}=\mathfrak{C}_T(\zeta_i)$ on $T^-$, $i=1,2,\ldots,n$; and thus for each $z_p = \sum_{i=1}^n\alpha^{+}_i\zeta_{i}$, we have $\bfv_T=\bfalpha^+$. Note that $\{ \zeta_i \}_{i=1}^n$ is a basis of $\mathbb{P}_p(T_\lambda)$ such that its restriction to
$T$ is a Lagrange polynomial basis of $\mathbb{P}_p(T)$.
Then, \eqref{lem_C_stability_eq1} and \eqref{IFE_mass_eq1_1} yield
\begin{equation}
\begin{split}
\label{IFE_mass_eq2}
\| \phi_T \|^2_{L^2(T)}  = \| \mathfrak{C}_T(z_p) \|^2_{L^2(T^-)} + \| z_p \|^2_{L^2(T^+)}  \lesssim \| z_p \|^2_{L^2(T^-)} + \| z_p \|^2_{L^2(T^+)} \lesssim  \| z_p \|^2_{L^2(T)}  \lesssim  h^2_T \bfv'_T \bfv_T,
\end{split}
\end{equation}
which implies that the second inequality in \eqref{IFE_mass_eq1} also holds for an interface element $T$. For the first inequality in \eqref{IFE_mass_eq1}, we also use \eqref{lem_C_stability_eq1} and \eqref{IFE_mass_eq1_1} to obtain
\begin{equation}
\begin{split}
\label{IFE_mass_eq3}
\| \phi_T \|^2_{L^2(T)}  = \| \mathfrak{C}_T(z_p) \|^2_{L^2(T^-)} + \| z_p \|^2_{L^2(T^+)}
\gtrsim \frac{(\beta^+)^2}{(\beta^-)^2} \| z_p \|^2_{L^2(T^-)} + \| z_p \|^2_{L^2(T^+)} \\ \gtrsim  \frac{(\beta^+)^2}{(\beta^-)^2}  \| z_p \|^2_{L^2(T)}  \gtrsim \frac{(\beta^+)^2}{(\beta^-)^2}   h^2_T \bfv'_T \bfv_T,
\end{split}
\end{equation}
Finally, summing \eqref{IFE_mass_eq1} over all elements leads to \eqref{IFE_mass_eq0}.
\end{proof}

We are now ready to state and prove the main theorem of this section about the spectral condition number of the stiffness matrix $\mathbf{K}_h$.
\begin{thm}
\label{thm_Kh_cond}
Suppose $\sigma^0$ and $\sigma^1$ are large enough, then
\begin{equation}
\label{thm_Kh_cond_eq0}
\kappa(\mathbf{K}_h) \lesssim \left( \frac{ \beta^- }{ \beta^+ } \right)^6 h^{-2}.
\end{equation}
\end{thm}
\begin{proof}
Since $\sigma^0$ and $\sigma^1$ are large enough, according to Lemma 6.1 and Theorem 6.1 in \cite{2019GuoLin}, the coercivity of $a_h(\cdot,\cdot)$ holds for the norm $\vertiii{\cdot}_h$, i.e., $a_h(v_h,v_h)\geqslant\frac{1}{4}\vertiii{v_h}^2_h$, $\forall v_h\in S^p_h(\Omega)$.
Let us recall that $S_h^p(\Omega)$ is equipped with the global IFE basis functions  $\psi_i$, $i=1,2,\ldots,N$. Hence, for every $v_h \in S_h^p(\Omega)$, written as $v_h=\sum^N_{i=1} v_i \psi_i$ with
coordinates $\bfv=[v_i]^N_{i=1}$,  we have
\begin{equation}
\label{thm_Kh_cond_eq1}
\frac{\bfv' \mathbf{K}_h \bfv}{\bfv'\bfv} = \frac{a_h(v_h,v_h)}{ \| v_h \|^2_{L^2(\Omega)} } \cdot \frac{ \|v_h \|^2_{L^2(\Omega)} }{\bfv'\bfv}.
\end{equation}
Applying the coercivity and continuity of $a_h(\cdot,\cdot)$ with \eqref{rem_L2_energy_equiv_eq}, we obtain
\begin{equation}
\label{thm_Kh_cond_eq2}
\frac{(\beta^-)^{4}}{(\beta^+)^{3}} h^{-2} \gtrsim \frac{7\vertiii{v_h}^2_h}{\|v_h\|^2_{L^2(\Omega)}} \geqslant \frac{a_h(v_h,v_h)}{ \| v_h \|^2_{L^2(\Omega)} } \geqslant \frac{\vertiii{v_h}^2_h}{4\|v_h\|^2_{L^2(\Omega)}} \gtrsim \beta^+.
\end{equation}
Combining \eqref{IFE_mass_eq0}, \eqref{thm_Kh_cond_eq1} and \eqref{thm_Kh_cond_eq2}  completes the proof.
\end{proof}

Finally we estimate the scaled condition number $\kappa_S(\mathbf{K}_h) = \kappa(\mathbf{ D}^{-1/2}_K\mathbf{ K}_h\mathbf{ D}^{-1/2}_K)$ where $\mathbf{ D}_K$ is the diagonal matrix of $\mathbf{ K}_h$. We know that $\mathbf{ D}_K$ is considered as a kind of preconditioner for some finite element methods for interface problems on unfitted meshes \cite{2012BabuskaBanerjee,2015BurmanClaus,2017LehrenfeldReusken}.
\begin{thm}
\label{thm_PKh_cond}
Suppose $\sigma^0$ and $\sigma^1$ are large enough, then
\begin{equation}
\label{thm_PKh_cond_eq0}
\kappa(\mathbf{ D}^{-1/2}_K\mathbf{K}_h\mathbf{ D}^{-1/2}_K) \lesssim  \left( \frac{\beta^-}{\beta^+} \right)^6 h^{-2}.
\end{equation}
\end{thm}
\begin{proof}
Following the notations and reasoning used in the proof of Theorem \ref{thm_Kh_cond}, for every $v_h\in S^p_h(\Omega)$ written as $v_h=\sum_{i=1}^Nv_i\psi_i$, we let $\bfv=[v_i]_{i=1}^N\in\mathbb{R}^N$; in addition we define $\tilde{\bfv}=\mathbf{ D}^{-1/2}_K\bfv=[\tilde{v}_i]_{i=1}^N$ and $\tilde{v}_h=\sum_{i=1}^N\tilde{v}_i\psi_i$. Then we can write
\begin{equation}
\label{thm_PKh_cond_eq1}
\frac{\bfv'\mathbf{ D}^{-1/2}_K\mathbf{K}_h \mathbf{ D}^{-1/2}_K\bfv}{\bfv'\bfv} = \frac{\tilde{\bfv}'\mathbf{ K}_h\tilde{\bfv}}{\tilde{\bfv}'\mathbf{ D}_K\tilde{\bfv}} = \frac{ \vertiii{\sum_{i=1}^N\tilde{v}_i\psi_i}^2_h }{ \sum_{i=1}^N\tilde{v}_i^2\vertiii{\psi_i}_h^2 } .
\end{equation}
We first note that
\begin{equation}
\label{thm_PKh_cond_eq1_1}
\vertiii{\sum_{i=1}^N\tilde{v}_i\psi_i}^2_h \leqslant \sum_{i=1}^N\tilde{v}_i^2\vertiii{\psi_i}_h^2.
\end{equation}
Then it only remains to estimate the lower bound for \eqref{thm_PKh_cond_eq1}. Lemma \ref{lem_glob_energy_norm_equiv} yields
\begin{equation}
\label{thm_PKh_cond_eq2}
\frac{ \vertiii{\sum_{i=1}^N\tilde{v}_i\psi_i}^2_h }{\sum_{i=1}^N\tilde{v}_i^2\vertiii{\psi_i}_h^2} \gtrsim \left(\frac{\beta^+}{\beta^-}\right)^4 h^2 \frac{ \|\sum_{i=1}^N\tilde{v}_i\psi_i\|^2_{L^2(\Omega)} }{\sum_{i=1}^N\tilde{v}_i^2 \|\psi_i \|^2_{L^2(\Omega)}}.
\end{equation}
Now on an element $T$, given each IFE function $\phi_T\in S^p_h(T)$, we write $\phi_T=\sum_{i=1}^n v_{T,i}\phi_{T,i}$ where $n=(p+1)(p+2)/2$ and $\{\phi_{T,i}\}_{i=1}^n$ are local IFE basis functions on $T$. Since they are linearly independent, we can define a special norm
\begin{equation}
\label{thm_PKh_cond_eq3}
\| \phi_T \|_{h,T} = \sum_{i=1}^n | v_{T,i} | \| \phi_{T,i} \|_{L^2(T)}.
\end{equation}
In order to estimate the lower estimate of the right hand side of \eqref{thm_PKh_cond_eq2}, we only need to prove the following inequality on each element
\begin{equation}
\label{thm_PKh_cond_eq3_1}
\| \phi_T \|_{h,T}  \lesssim \frac{\beta^-}{\beta^+} \| \phi_T \|_{L^2(T)} ,  ~~~~~ \forall \phi_T \in S^p_h(T).
\end{equation}
If $T$ is a non-interface element, by the standard scaling argument, we have $\|\cdot\|_{h,T}\simeq \| \cdot \|_{L^2(T)}$ on $S^p_h(T)=\mathbb{P}_p$ which yields \eqref{thm_PKh_cond_eq3_1} since $\beta^-\geqslant\beta^+$. However, if $T$ is an interface element, following the similar argument to Lemma \ref{IFE_mass}, we have $\phi^+_{T,i}=\phi_{T,i}|_{T^+}=\zeta_i$ on $T^+$ and $\phi^-_{T,i}=\phi_{T,i}|_{T^-}=\mathfrak{C}_T(\zeta_i)$ on $T^-$, $i=1,2,\ldots,n$. Thus there holds $z_p = \sum_{i=1}^n\alpha^{+}_i\zeta_{i} = \sum_{i=1}^n v_{T,i}\zeta_{i}$. Then the second inequality in Lemma \ref{lem_C_stability} yields
\begin{equation}
\begin{split}
\label{thm_PKh_cond_eq4}
\| \phi_T \|_{h,T} &=\sum_{i=1}^n | v_{T,i} | \| \phi_{T,i} \|_{L^2(T)} \lesssim  \sum_{i=1}^n | v_{T,i} | ( \| \zeta_{i} \|_{L^2(T^+)} + \| \mathfrak{C}_T(\zeta_{i}) \|_{L^2(T^-)} ) \\
&\lesssim  \sum_{i=1}^n | v_{T,i} | ( \| \zeta_{i} \|_{L^2(T^+)} + \| \zeta_{i} \|_{L^2(T^-)} ) \\
&\lesssim \sum_{i=1}^n | v_{T,i} |  \| \zeta_{i} \|_{L^2(T)} = \| z_p \|_{h,T} \lesssim \| z_p \|_{L^2(T)}, 
\end{split}
\end{equation}
where in the last two inequalities we have also used the equivalence $\|\cdot\|_{h,T}\simeq \| \cdot \|_{L^2(T)}$ on $\mathbb{P}_p(T)$. Then we note that $\phi^+_T=\mathfrak{C}^{-1}_T(\phi^-_T)$ and further use the first inequality in Lemma \ref{lem_C_stability} to obtain
\begin{equation}
\label{thm_PKh_cond_eq4_1}
\| z_p \|_{L^2(T)} \lesssim \| z_p \|_{L^2(T^+)} + \| \mathfrak{C}^{-1}_T (\mathfrak{C}_T (z_p)) \|_{L^2(T^-)} \lesssim \| \phi^+_T \|_{L^2(T^+)} + \frac{\beta^-}{\beta^+}\| \phi^-_T \|_{L^2(T^-)} \lesssim \frac{\beta^-}{\beta^+}\| \phi_T \|_{L^2(T)}
\end{equation}
Combining \eqref{thm_PKh_cond_eq4} and \eqref{thm_PKh_cond_eq4_1}, we obtain \eqref{thm_PKh_cond_eq3_1}. Squaring \eqref{thm_PKh_cond_eq3_1} and summing over all the elements yield
\begin{equation}
\label{thm_PKh_cond_eq5}
\frac{ \|\sum_{i=1}^N\tilde{v}_i\psi_i\|^2_{L^2(\Omega)} }{\sum_{i=1}^N\tilde{v}_i^2 \|\psi_i\|^2_{L^2(\Omega)}} \gtrsim \left(\frac{\beta^+}{\beta^-}\right)^2.
\end{equation}
Combining \eqref{thm_PKh_cond_eq5} and \eqref{thm_PKh_cond_eq2} yields \eqref{thm_PKh_cond_eq1} which, in turns, gives \eqref{thm_PKh_cond_eq0}.
\end{proof}
Note that $\mathbf{ D}^{-1}_K\mathbf{K}_h = \mathbf{ D}^{-1/2}_K( \mathbf{ D}^{-1/2}_K\mathbf{K}_h \mathbf{ D}^{-1/2}_K ) \mathbf{ D}^{1/2}_K$; so the result \eqref{thm_PKh_cond_eq0} is also true for $\mathbf{ D}^{-1}_K\mathbf{K}_h $.

\begin{rem}
\label{rem:cond_beta}
~\\
\vspace{-0.2in}
\begin{itemize}
\item We note that Theorems \ref{thm_Kh_cond} and \ref{thm_PKh_cond} indicate the dependence of the upper bound of the condition number on the mesh size $h$ and the contrast $\beta^-/\beta^+$. Here the growth with respect to $h$, i.e., $h^{-2}$ is optimal and sharp in the sense that it is comparable with the standard fitted mesh finite element methods and it can be also verified by numerical experiments, see the results in Section \ref{sec:num_example} below.

\item On the other hand, we have observed from numerical experiments (some of them are reported in the next section) that the dependence of the condition number on the contrast is not worse than quadratic which is much smaller than the bound $(\beta^-/\beta^+)^6$ derived in Theorems \ref{thm_Kh_cond} and \ref{thm_PKh_cond}. However its rigorous proof is still open and will be our research topic in the future. Moreover designing an IFE method with a condition number that depends weakly
on the contrast $\rho = \beta^+/\beta^-$ is also worth investigating.

\item
Unlike some methods on unfitted meshes \cite{2017LehrenfeldReusken} which require a preconditioner (such as the diagonal preconditioner described above) to obtain uniform $O(h^{-2})$ bounds for condition numbers, the estimate \eqref{thm_Kh_cond_eq0} implies that condition number of the proposed IFE method itself has the $O(h^{-2})$ bounds, and so does the preconditioned one \eqref{thm_PKh_cond_eq0}. Furthermore the numerical results in the next section suggest that diagonal scaling can sometimes slightly reduce the condition number, especially for linear elements. However, suitable preconditioners for IFE methods that deal with large contrast and high polynomial degree are still needed.
\end{itemize}
\end{rem}

\section{Numerical Examples}
\label{sec:num_example}

In this section, we perform several numerical experiments to corroborate the theoretical results
of the previous sections and to investigate the numerical performance of the proposed IFE method
\eqref{enr_IFE_method_1}. We note that the original interface problem \eqref{model} can be normalized by dividing
the elliptic equation on bother sides of the interface with the larger coefficient, i.e., by $\beta^-$ if $\beta^-\geqslant\beta^+$.
Then the normalized interface problem has the new coefficients $1$ and $\beta^+/\beta^-\leqslant 1$, and the stability parameter $\gamma$ becomes $1^2/(\beta^+/\beta^-)=\beta^-/\beta^+$ which is exactly the contrast $\rho$. Since the original interface problem and the resulted IFE scheme are mathematically equivalent to their normalized counterparts, this suggests that choosing $\gamma=\beta^-/\beta^+$ is sufficient to guarantee the optimal convergence and stability of the proposed IFE method here and also the one in \cite{2019GuoLin} for interface problems with homogeneous jump conditions. Actually this choice has also been widely used in the IFE literature \cite{2016GuoLin,2015LinLinZhang,2015LinSheenZhang}.

We perform all the computations on the domain $\Omega=(-1,1)\times(-1,1)$ for the symmetric IFE formulation. The meshes for the IFE method are generated by partitioning $\Omega$ into $N\times N$ uniform squares in which each square is cut into two triangles by joining the upper-left and lower-right vertices with a mesh size $h=2/N$. In all computations the IFE functions are constructed on fictitious elements $T_{1.5}$ unless it is specified otherwise.
\vskip 0.2in
\noindent\textbf{Example 1}.

The first set of numerical experiments are for investigating the stability of the local problem
for constructing the IFE functions stated in Theorem \ref{thm_condition_AT}.
More specifically, we want to comfirm that the upper bounds of $\kappa(A_T)$ are independent of both
interface location and element size $h$.
For this purpose, we consider circular interfaces centered at the origin with radius $r$,
meshes sizes $h=2/10,2/40,2/80$ and degree $p=3$.
Each circular interface splits $\Omega$ into
the interior $\Omega^-$ 
and the exterior $\Omega^+$ subdomains. 
On a mesh of size $h=2/10$, we select the
interface element $T=\triangle_1$ having vertices $A_1=(0.6,0)$, $A_2=(0.8,0)$, $A_3=(0.6,0.2)$,
as illustrated in Figure \ref{fig:loc_cond_num} and let $d_r$
denote the distance between the vertex $A_2$ and the intersection
point $\Gamma \cap A_1A_2$, i.e., $d_r = 0.8-r$.
Thus, as $d_r$ decreases to $0$, the intersection point $\Gamma\cap A_1A_2$
approaches the vertex $A_2$ creating small-cut elements which might cause numerical difficulties and instabilities.
Similarly, on a mesh of size $h=2/40$,
we select the interface element $T=\triangle_2$ having vertices $A_1=(0.75,0)$, $A_2=(0.8,0)$, $A_3=(0.75,0.05)$ while
on a mesh of size $h=2/80$ we select
the interface element $T=\triangle_3$ having
vertices $A_1=(0.775,0)$, $A_2=(0.8,0)$, $A_3=(0.775,0.025)$.

For $\triangle_1$ we select $r$ such that
$d_r=10^{-1},5\,10^{-2},10^{-2}$, $10^{-3},10^{-4},10^{-5},10^{-6},10^{-7}$ and compute $\kappa(\mathbf{A}_T)$
on a fictitious element $T_{1.5}$ and repeat these computations for $\triangle_2$ with
$d_r= 5 \,10^{-2},10^{-2}$, $10^{-3},10^{-4}$, $10^{-5},10^{-6},10^{-7}$
 and for $\triangle_3$ with $d_r= 10^{-2}$, $10^{-3},10^{-4}$, $10^{-5},10^{-6},10^{-7}$.
In the middle of Figure \ref{fig:loc_cond_num} we plot $\kappa(\mathbf{A}_T)$ versus $d_r$, and in the right of Figure \ref{fig:loc_cond_num} we further vary $\lambda$ and plot $\kappa(\mathbf{A}_T)$ versus $\lambda=1,1.1,\ldots,2$ for $d_r=10^{-7}$.
Clearly, these results show that $\kappa(\mathbf{A}_T)$ approaches a constant as $d_r$ approches $0$ with no blow-up in full agreement with Theorem \ref{thm_condition_AT}.
Moreover, $\kappa(\mathbf{A}_T)$ is extremely ill-conditioned on the original element $T_1$ and
becomes significantly better conditioned with increasing scaling parameter $\lambda$.

\begin{figure}[htbp]
\centering
    \includegraphics[width=1.9in]{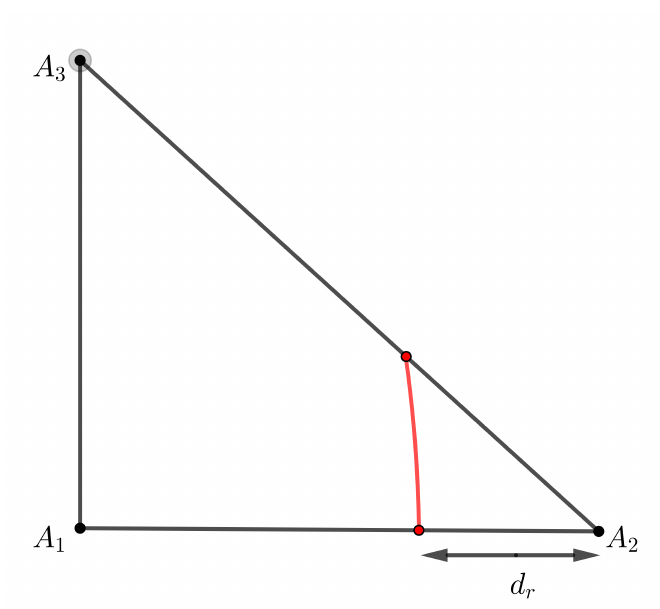}
    \hspace{-0.2in}
    \includegraphics[width=2.7in]{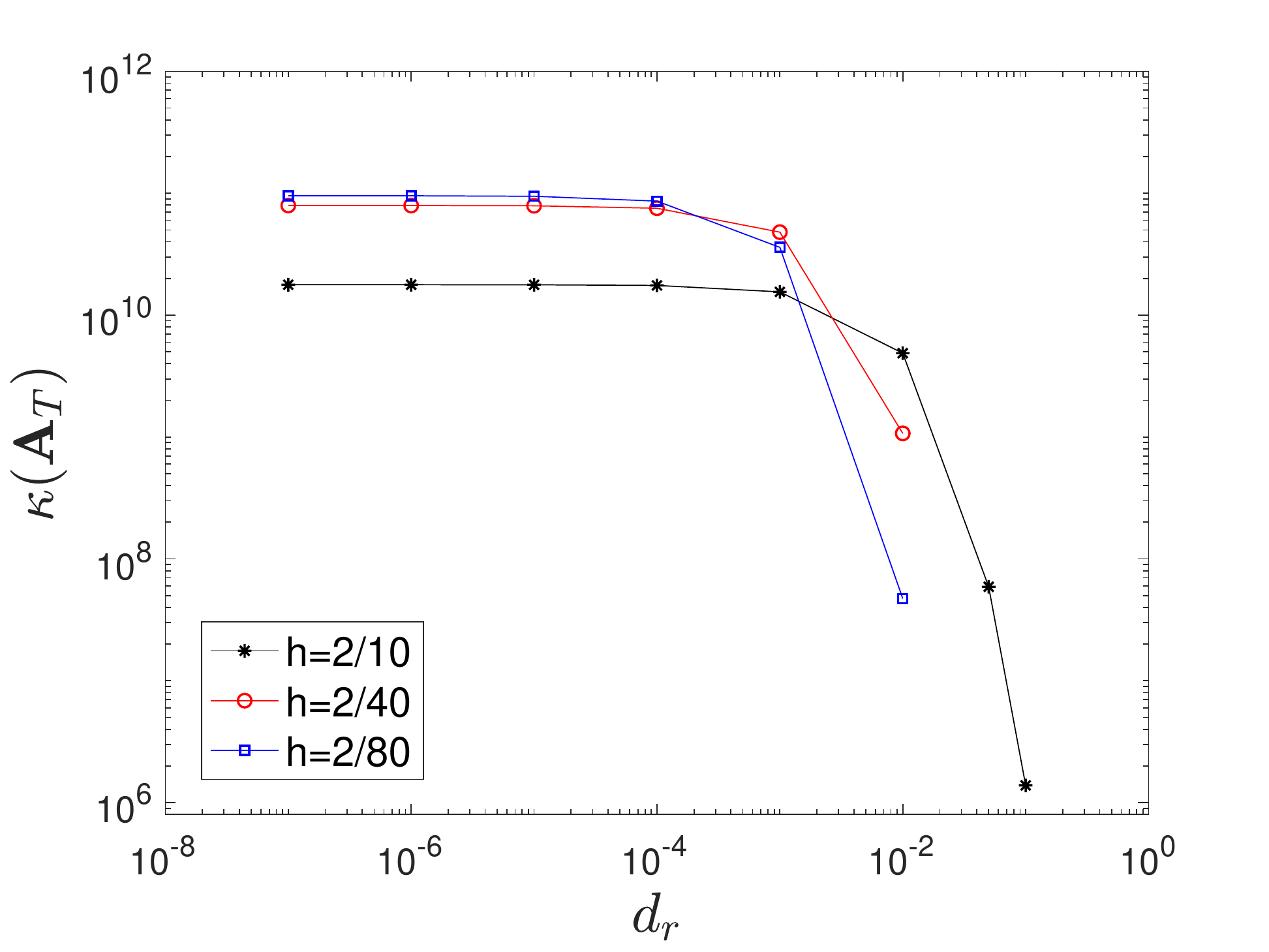}
     \hspace{-0.3in}
    \includegraphics[width=2.7in]{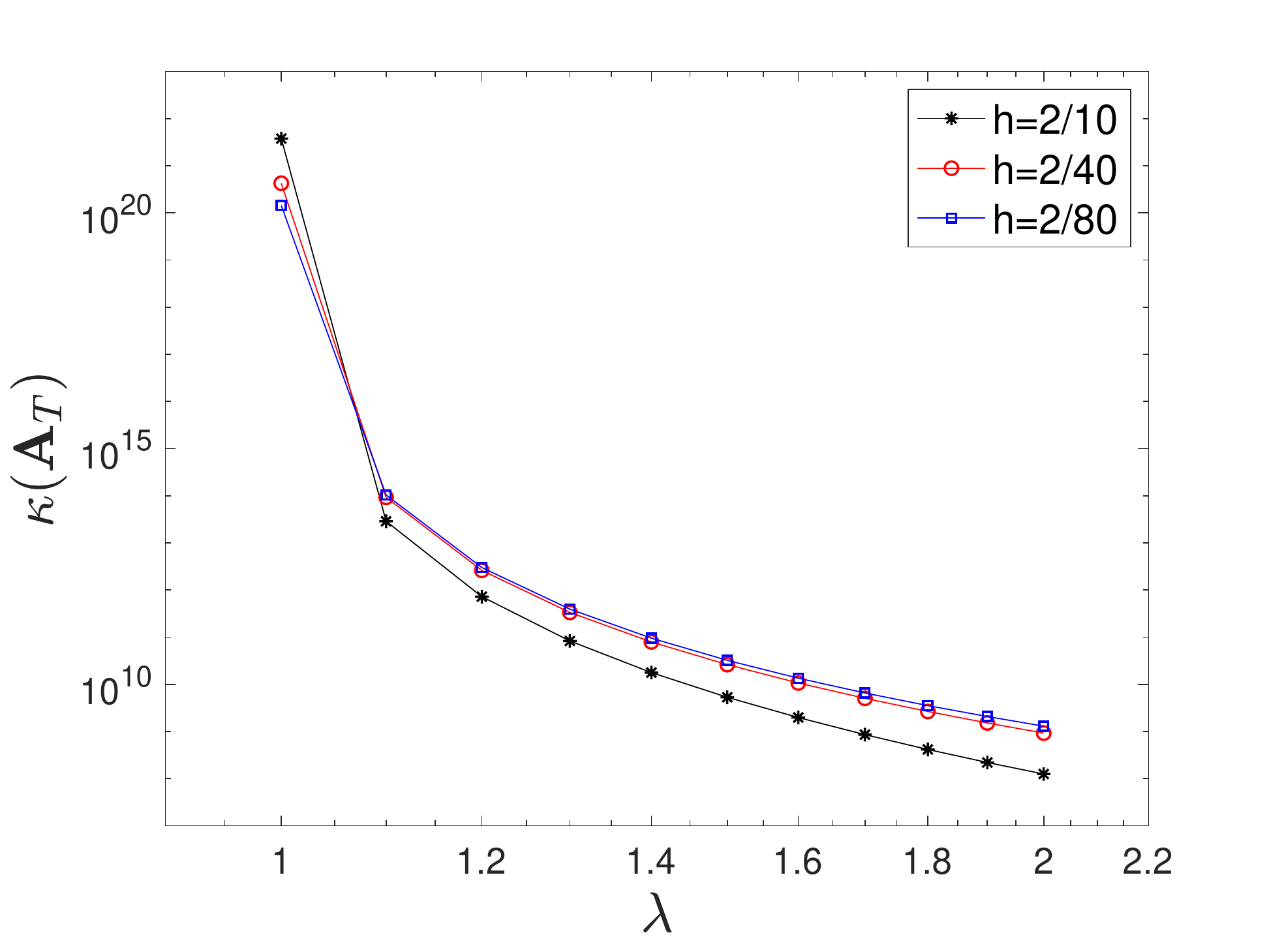}
    \caption{Interface element (left), $\kappa(\mathbf{A}_T)$ versus $d_r$ (center) and $\kappa(\mathbf{A}_T)$ versus $\lambda$ (right).}
    \label{fig:loc_cond_num}
\end{figure}

~\\
\noindent\textbf{Example 2}.

Inspired by \cite{2016WangXiaoXu,2010WuXiao} we perform a second set of numerical experiments
for the model problem \eqref{model} with a circular interface spliting  $\Omega=(-1,1)^2$ into
$\Omega^+=\{(x,y)\in\Omega, x^2+y^2<(\pi/4)^2\}$ and $\Omega^-=\{(x,y)\in \Omega, x^2+y^2>(\pi/4)^2\}$,
with the exact solution
\begin{equation}
\label{exact_solu}
u(x,y)=
\begin{cases}
      & \exp(xy)/\beta^-, ~~~~~~~~~~~\, X=(x,y)\in\Omega^-, \\
      &  \sin(\pi x)\sin(\pi y)/\beta^+, ~~~ X=(x,y)\in\Omega^+,
\end{cases}
\end{equation}
with the appropriate interface jumps $J_D$, $J_N$, source term $f$ and Dirichlet boundary conditions.
Since this solution has nonhomogeneous jump conditions and
does not satisfy the homogeneous extended jump conditions \eqref{normal_jump_cond}, we need all
three enriched IFE functions $\phi_{T,D}$, $\phi_{T,N}$ and $\phi_{T,f}$, for $p\geqslant2$ on each
interface element $T$ as in \eqref{enrich_IFE_fun_glob}. We use meshes of
size $h=2/N$, $N=10,20,\ldots,80$ to solve \eqref{model} with $\beta=(\beta^+,\beta^-)=(2,1)$, and we
plot the $L^2$ and $H^1$ errors versus $N$ in  Figure \ref{fig:convergence_rates_bp2_examp1}.
In order to show the effect of the contrast we repeat the previous experiment with all parameters unchanged except for $\beta=(500,1)$
and show the errors in Figure \ref{fig:convergence_rates_bp500_examp1}. The dashed lines are reference lines passing through the most right data point and having $slope=p+1$ for the $L^2$ errors and $slope=p$ for the $H^1$ errors.
These results are in full agreement with the theory and show optimal convergence rates. For the relatively large contrast $\beta=(500,1)$
the errors of Figure \ref{fig:convergence_rates_bp500_examp1} wiggle slightly while converging
with optimal rates obtained by least-squares fitting.

\begin{figure}[H]
\centering
    \subfigure{
    \includegraphics[width=3in]{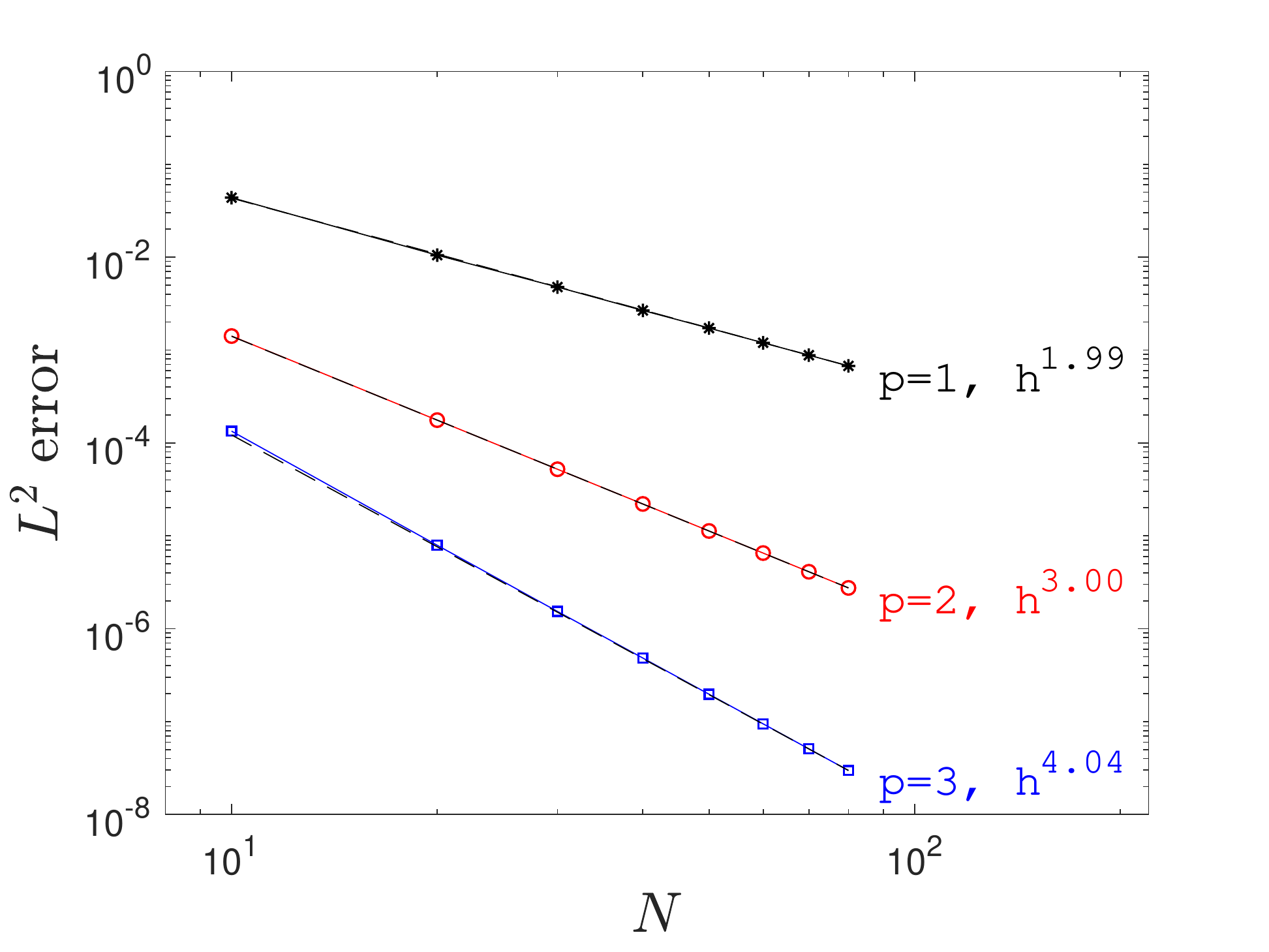}}
     \subfigure{
    \includegraphics[width=3in]{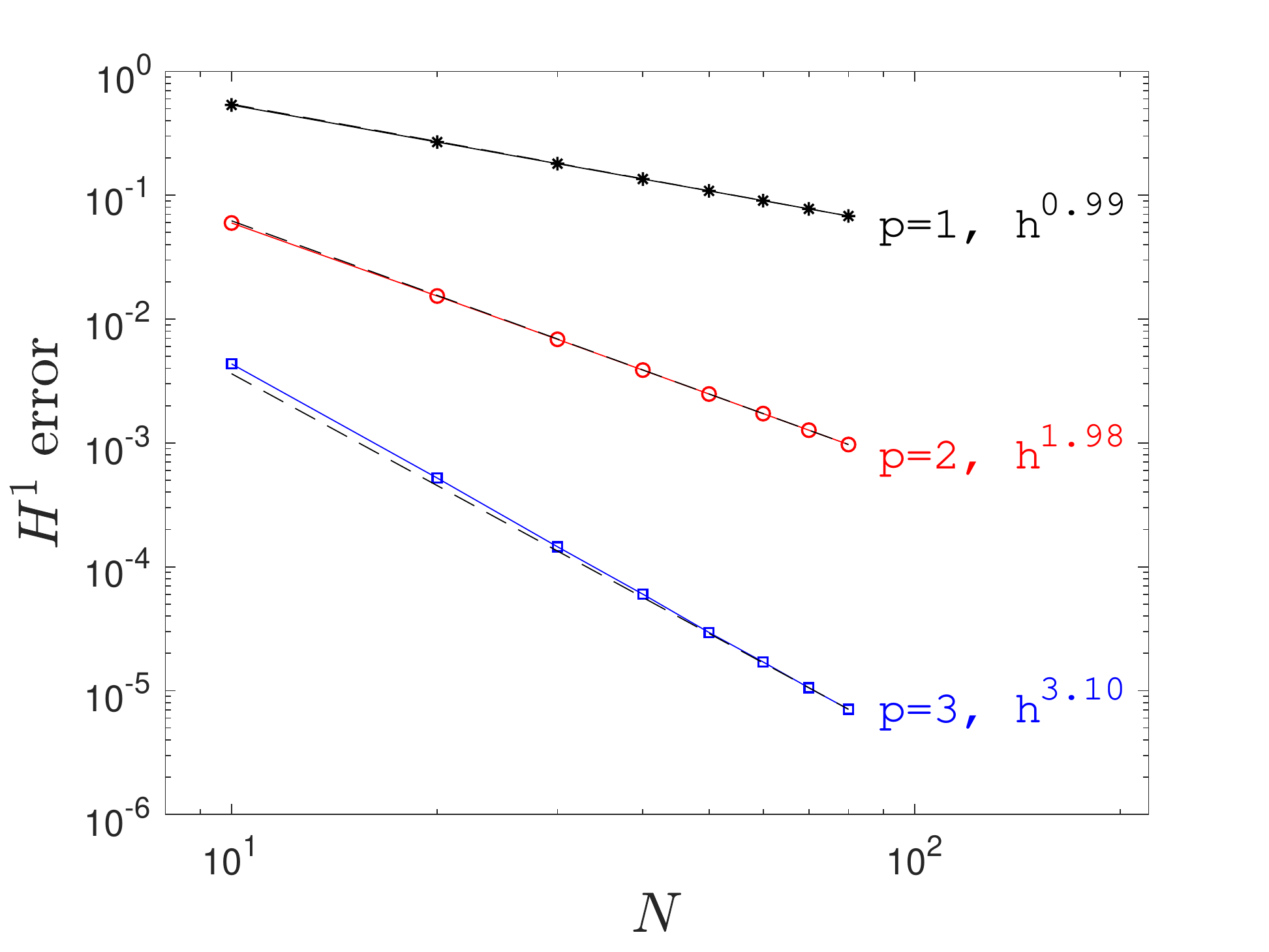}}
  \caption{Errors in $L^2$ (left) and $H^1$ (right) norms versus $N$ for
  $\beta=(2,1)$.}

  \label{fig:convergence_rates_bp2_examp1} 
\end{figure}

\begin{figure}[H]
\centering
    \subfigure{
    \includegraphics[width=3in]{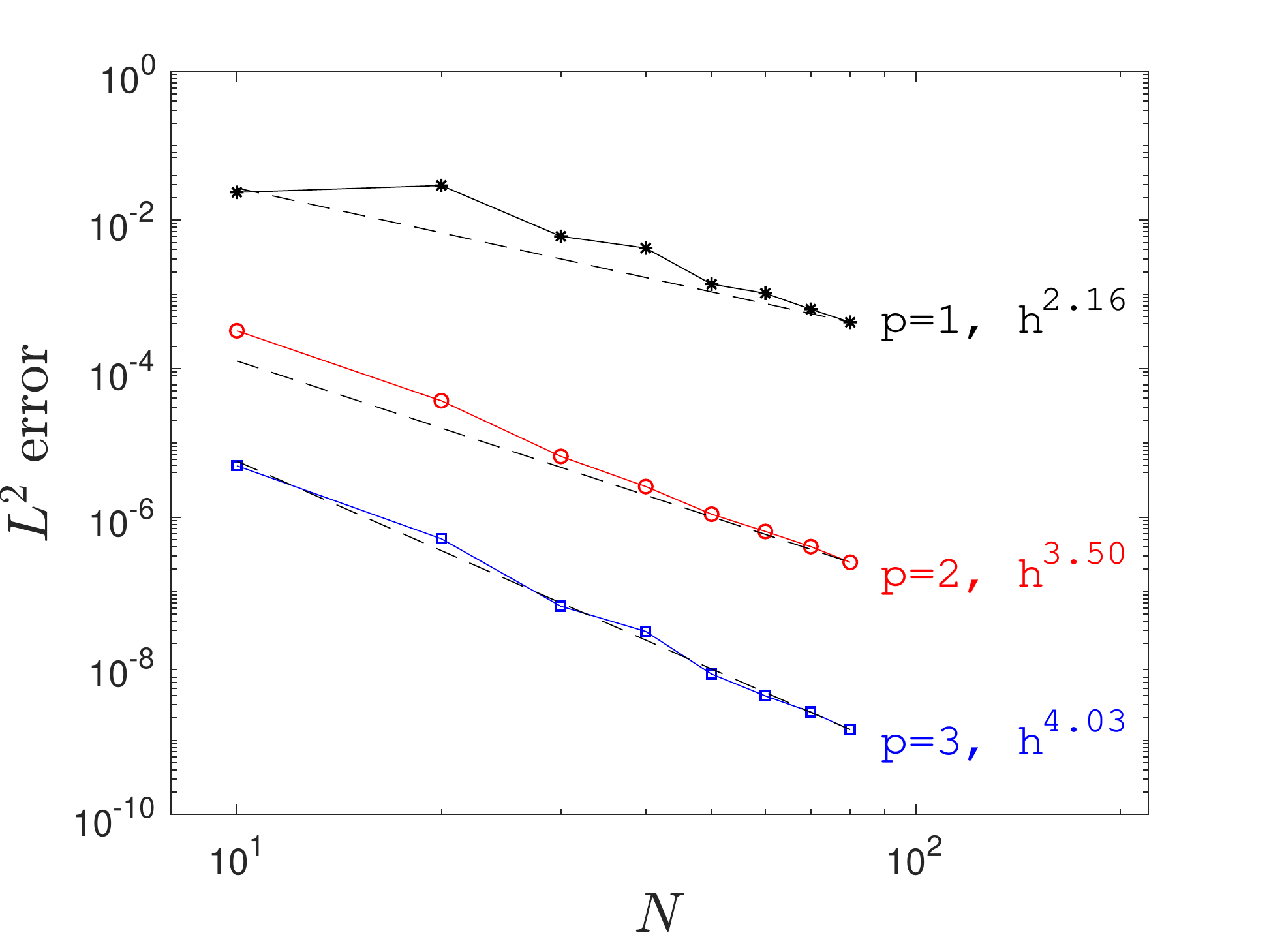}}
     \subfigure{
    \includegraphics[width=3in]{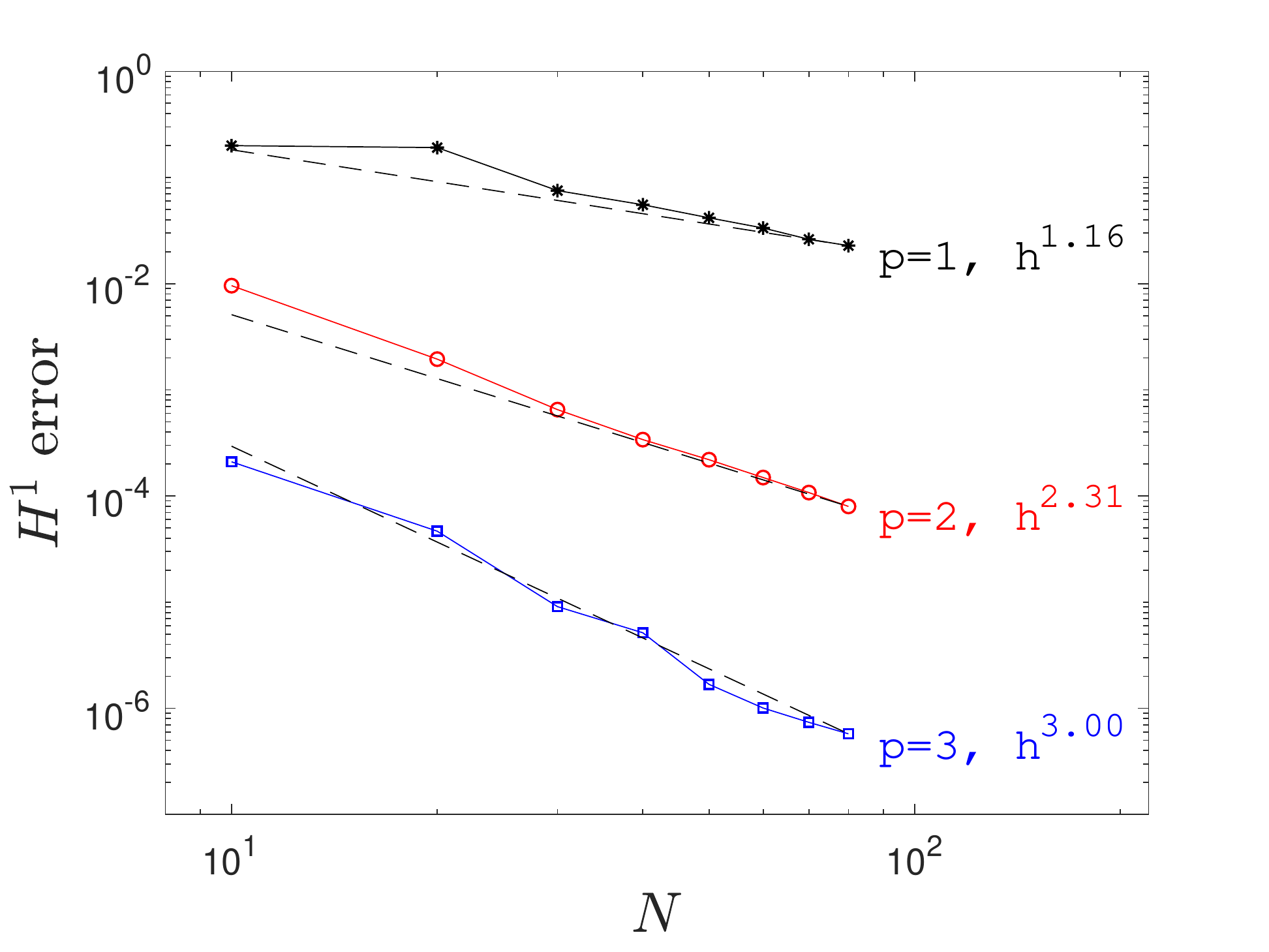}}
  \caption{Errors in $L^2$ (left) and $H^1$ (right) norms versus $N$ for
  $\beta=(500,1)$.}
  \label{fig:convergence_rates_bp500_examp1} 
\end{figure}

Next, we investigate the effect of the penalty terms on the numerical performance of the proposed IFE method. First, we note that the edge penalty terms in \eqref{weak_form} are required to guarantee both the invertibility of the stiffness matrix and the convergence of the IFE solution.
We further study the effect of the interface penalty terms by performing the numerical experiments using a modified IFE method by removing all interface penalty terms appearing in $a_h(\cdot,\cdot)$ as well as the term $\int_{\Gamma_T}J_D[v]_{\Gamma} ds$ appearing in $L_f$. However, the two terms $\int_{\Gamma}J_N\{v\}_{\Gamma} ds$ and $\int_{\Gamma}J_D \{\beta \nabla v\cdot\mathbf{ n} \}_{\Gamma} ds$ appearing in the right hand side \eqref{new_Lf} in $L_f$ are kept to enforce the nonhomogeneous interface conditions.
The $L^2$ and $H^1$ errors versus $N$ shown in Figure \ref{fig:convergence_rates_bp2_examp1_nopen} for the relatively small contrast
$\beta=(2,1)$ clearly show sub-optimal convergence rates and
much larger errors than those in Figure \ref{fig:convergence_rates_bp2_examp1} when
interface penalty terms are used. These experiments numerically confirm that the interface penalty terms are needed for the proposed scheme to attain optimal convergence rates.

\begin{figure}[H]
\centering
    \subfigure{
    \includegraphics[width=3in]{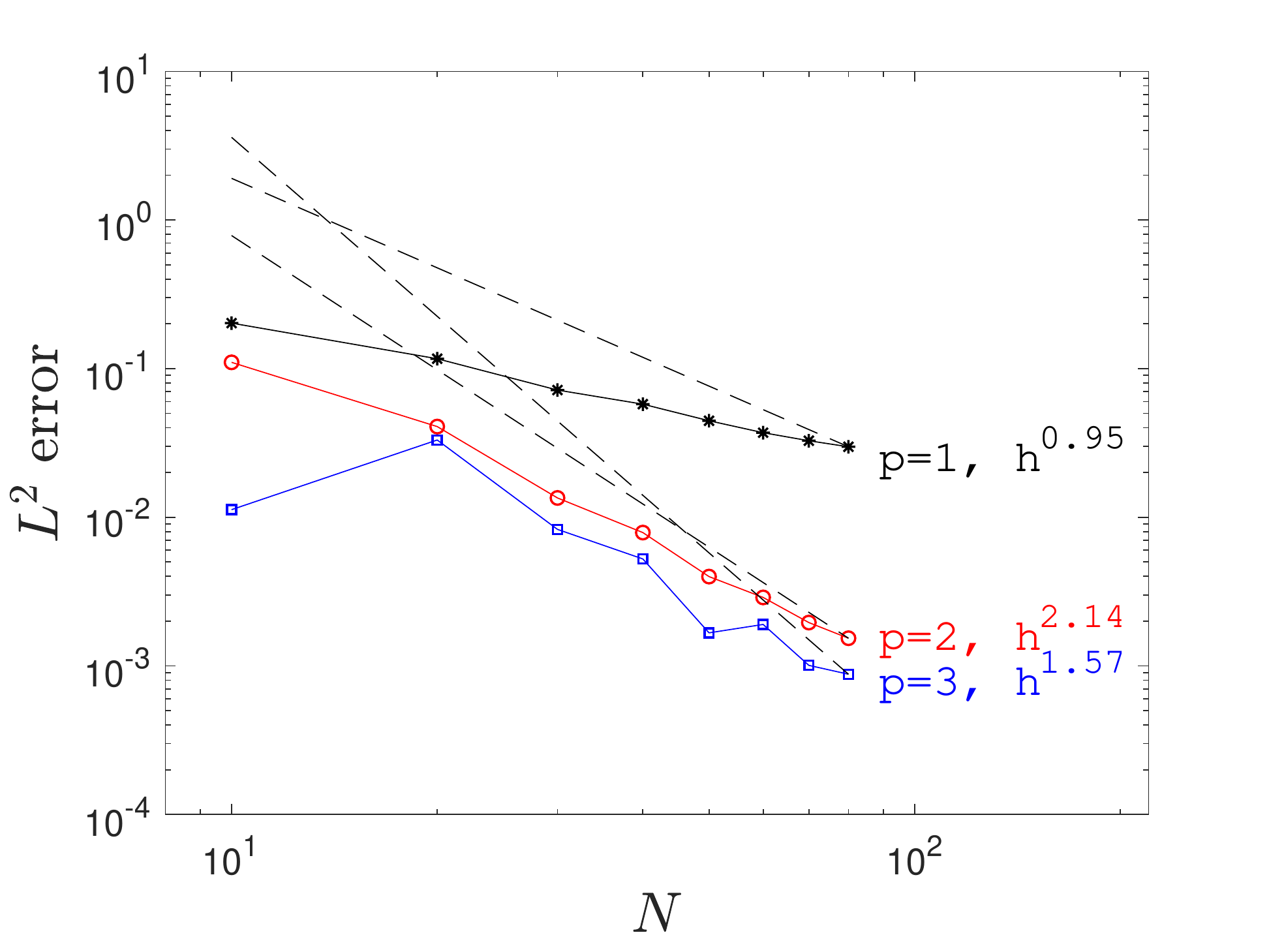}}
     \subfigure{
    \includegraphics[width=3in]{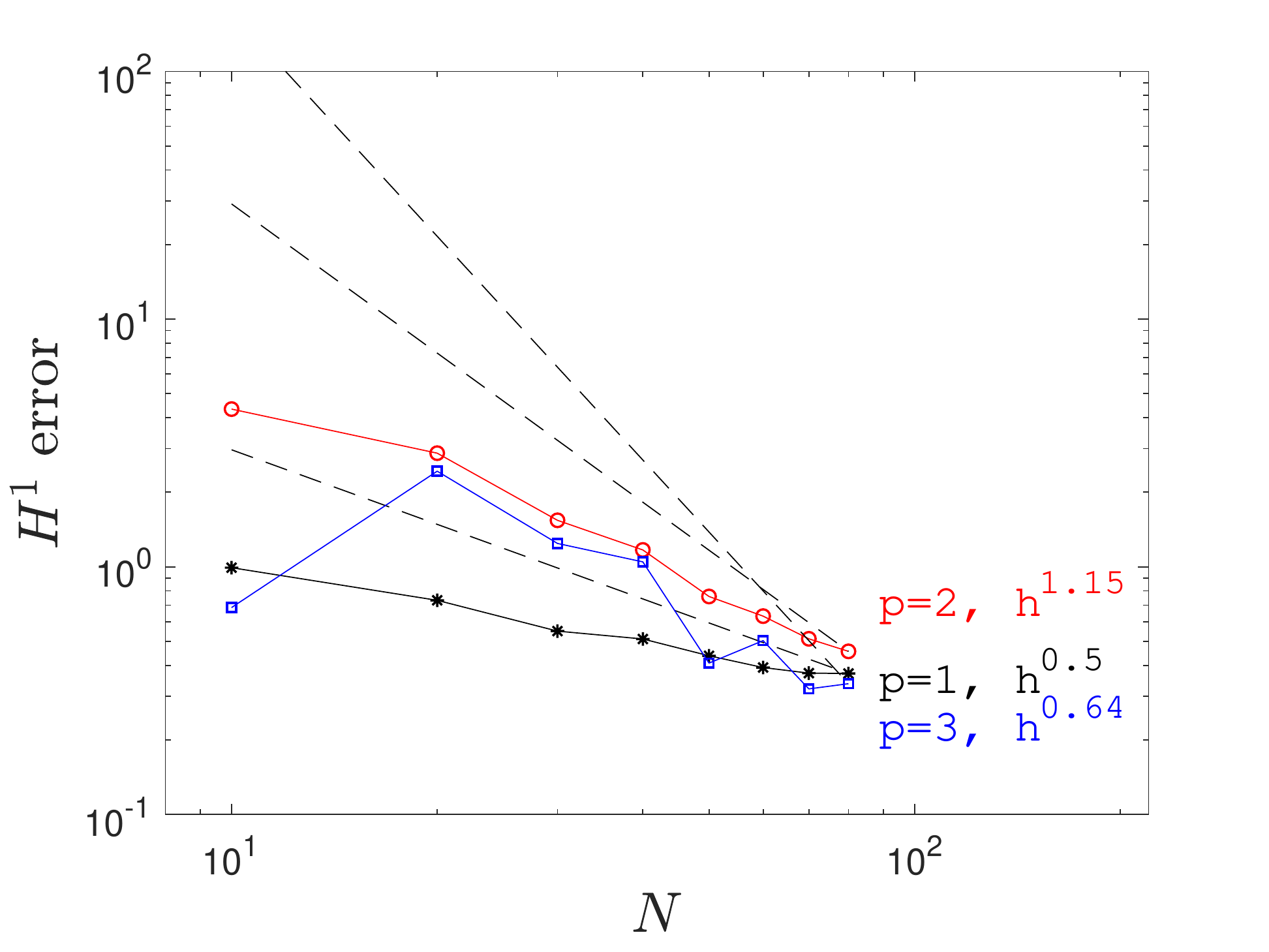}}
  \caption{Errors in $L^2$ (left) and $H^1$ (right) norms versus $N$ for $\beta=(2,1)$
   without interface penalties.}
 \label{fig:convergence_rates_bp2_examp1_nopen} 
\end{figure}

\begin{figure}[H]
\centering
    \subfigure{
    \includegraphics[width=3in]{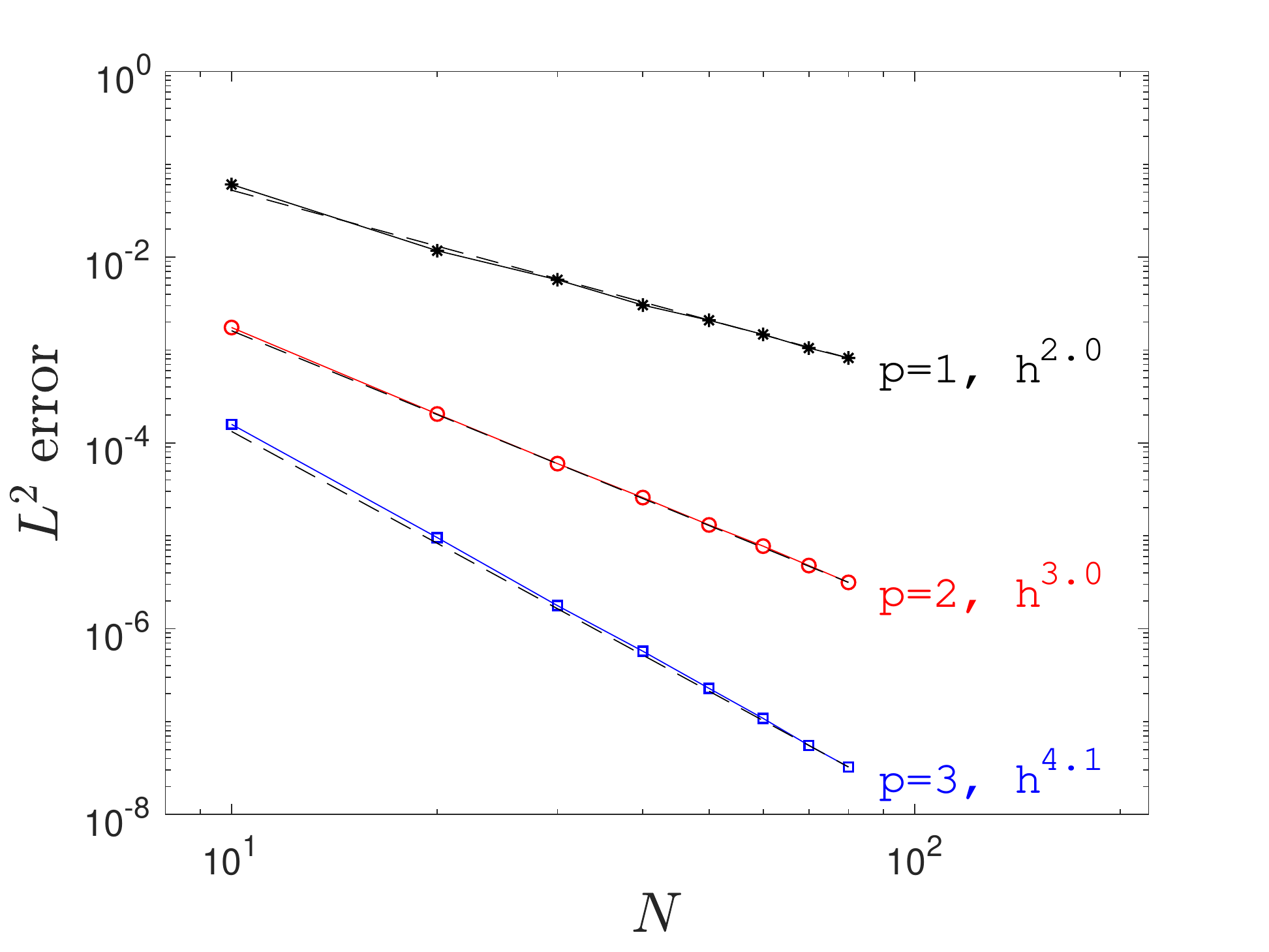}}
     \subfigure{
    \includegraphics[width=3in]{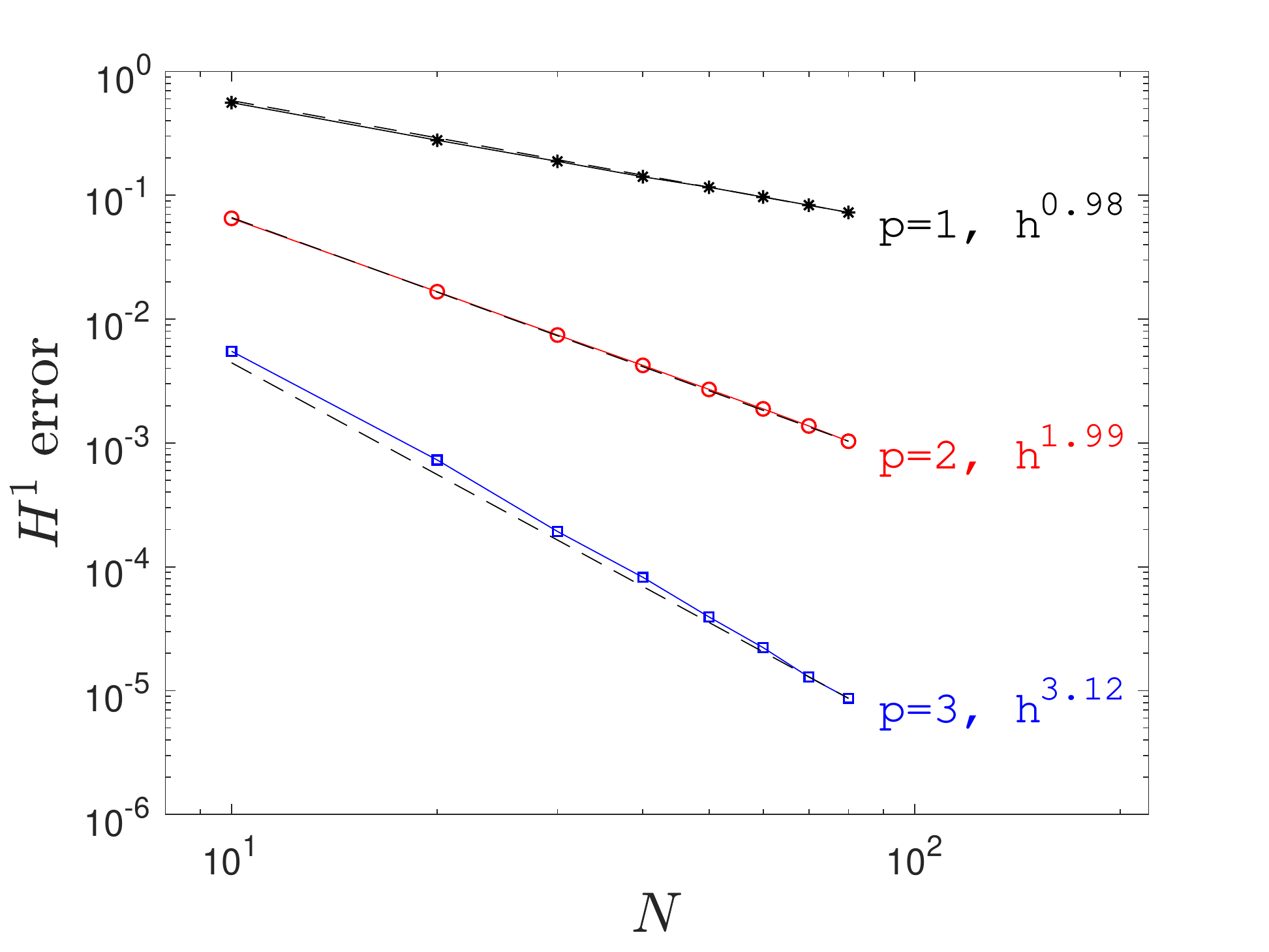}}
  \caption{Errors in $L^2$ (left) and $H^1$ (right) norms versus $N$ for
  $\beta=(2,1)$ with $\theta=2$.}
 \label{fig:convergence_rates_bp2_examp1_theta2} 
\end{figure}

Another parameter in the penalty terms which may effect the numerical performance of the proposed IFE method is the exponent $\theta$ of $|e|$ and $h_T$ in the stability
terms of the IFE method \eqref{weak_form}. The errors shown in Figure \ref{fig:convergence_rates_bp2_examp1_theta2} generated with $\theta=2$ for $\beta =(2,1)$ yield optimal convergence rates.
However for large contrasts, numerical experiments demonstrate that the errors produced by the IFE method with $\theta=2$ are much larger those with $\theta=1$, and they become more oscillatory. This unsatisfactory behavior is also observed for $\theta=1/2$ for large contrast problems. Thus, numerical results suggest that $\theta=1$ is a safe choice and that the conditioning of the IFE method deteriorates with larger values of $\theta$.
~\\

\noindent\textbf{Example 3}.

Now we investigate the stability of the IFE method for the interface problem in the previous example by computing the spectral condition number $\kappa(\mathbf{K}_h)$
of the stiffness matrix \eqref{uh_solu} and its scaled spectral condition number $\kappa_S(\mathbf{K}_h)=\kappa(\mathbf{ D}^{-1/2}_K\mathbf{ K}_h\mathbf{ D}^{-1/2}_K)$
where $\mathbf{D}_K$ is the diagonal of $\mathbf{K}_h$.
In Figure \ref{fig:global_cond_num_examp1} we plot $\kappa(\mathbf{K}_h)$ and  $\kappa_S(\mathbf{K}_h)$ versus $N$, for $p=1,2,3$ with $\beta=(2,1)$ (left) and $\beta=(500,1)$ (right).
The dashed lines indicate the expected growth rate $h^{-2}$ according to Theorem \ref{thm_Kh_cond}
while the average growth rates are shown at the right end of each curve.
We observe that both $\kappa(\mathbf{K}_h)$ and $\kappa_S(\mathbf{K}_h)$
grow like $O(h^{-2})$ under mesh refinement which is in agreement with Theorem \ref{thm_Kh_cond} and Theorem \ref{thm_PKh_cond}. We further observe that  both  $\kappa(\mathbf{K}_h)$  and $\kappa_S(\mathbf{K}_h)$ increase with increasing degree $p$ and contrast $\rho$.
We also observe that $\kappa_S(\mathbf{K}_h)$
is slightly smaller than $\kappa(\mathbf{K}_h)$.

\begin{figure}[H]
\centering
    \subfigure{
    \includegraphics[width=3in]{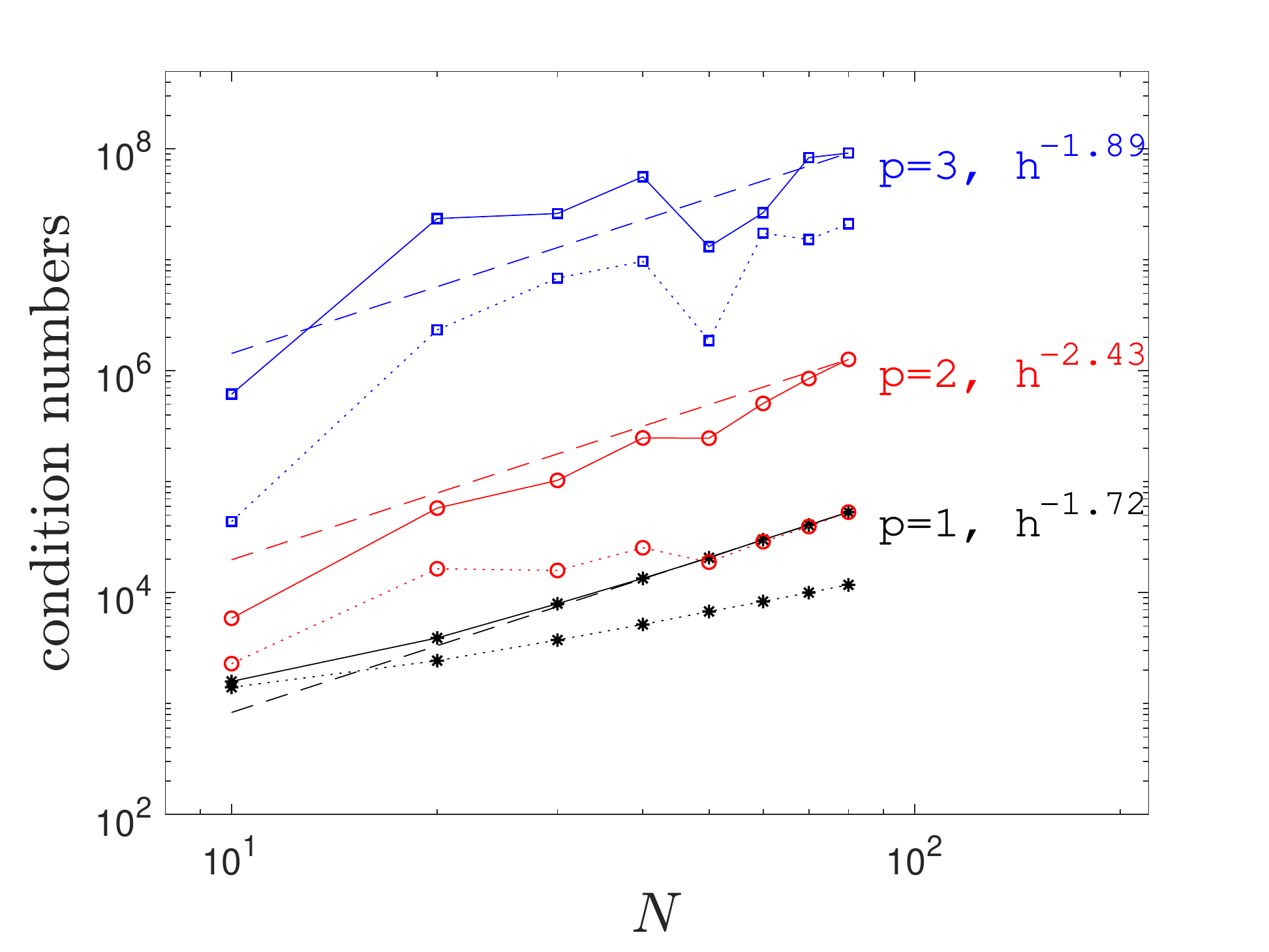}}
   \subfigure{
   \includegraphics[width=3in]{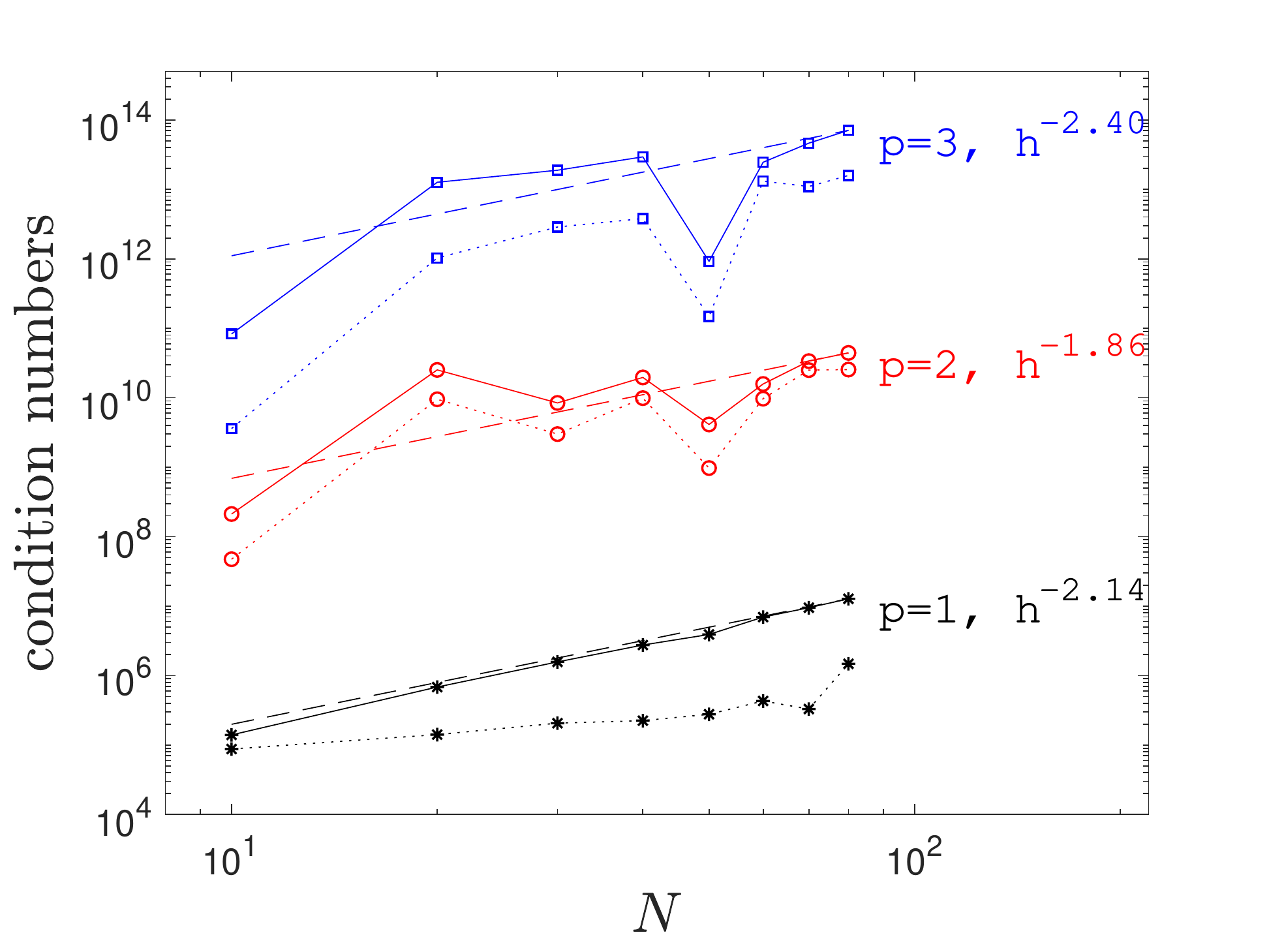}}
  \caption{{Condition numbers $\kappa(\mathbf{K}_h)$ (solid line) and
  $\kappa_S(\mathbf{K}_h)$
  (dotted line) for $\beta=(2,1)$ (left) and
  $\beta=(500,1)$ (right) versus $N$.
  Dashed line is the reference line for the growth $N^2$.}}
  \label{fig:global_cond_num_examp1} 
\end{figure}

Next we investigate how $\kappa(\mathbf{K}_h)$ depends on the contrast $\rho$ for
the model problem \eqref{model} with a circular interface centered at the origin and with radius
$r =\pi/4$ on a uniform mesh of size $h=2/40$, $\beta=(10^{\textrm{sgn}(i)}\cdot2^i,1)$, $i=-10,\ldots,-1,0,1,\ldots,10$ and $\textrm{sgn}(z)=-1$ if $z<0$, $\textrm{sgn}(0)=0$ and $\textrm{sgn}(z)=1$, if $z>0$.
We plot $\kappa(\mathbf{K}_h)$ versus $\rho$ in Figure \ref{fig:cond_IFE_vs_beta_rpi4} and show average growth rates with respect to $\rho$. From the numerical results, we can observe a linear growth rate for $p=1$ and rates not exceeding quadratic growth for $p\ge 2$. We would like to mention that the observed
growth rates of $\kappa(\mathbf{K}_h)$ with respect to $\rho$ are much better than the theoretical bounds of Theorem \ref{thm_Kh_cond} which deserves further investigation for an optimal bound. {The numerical results of Figure \ref{fig:global_cond_num_examp1} and \ref{fig:cond_IFE_vs_beta_rpi4} also show that the condition numbers $\kappa(\mathbf{K}_h)$ and $\kappa_S(\mathbf{K}_h)$ are comparable.}

\begin{figure}[H]
\centering
    \includegraphics[width=3in]{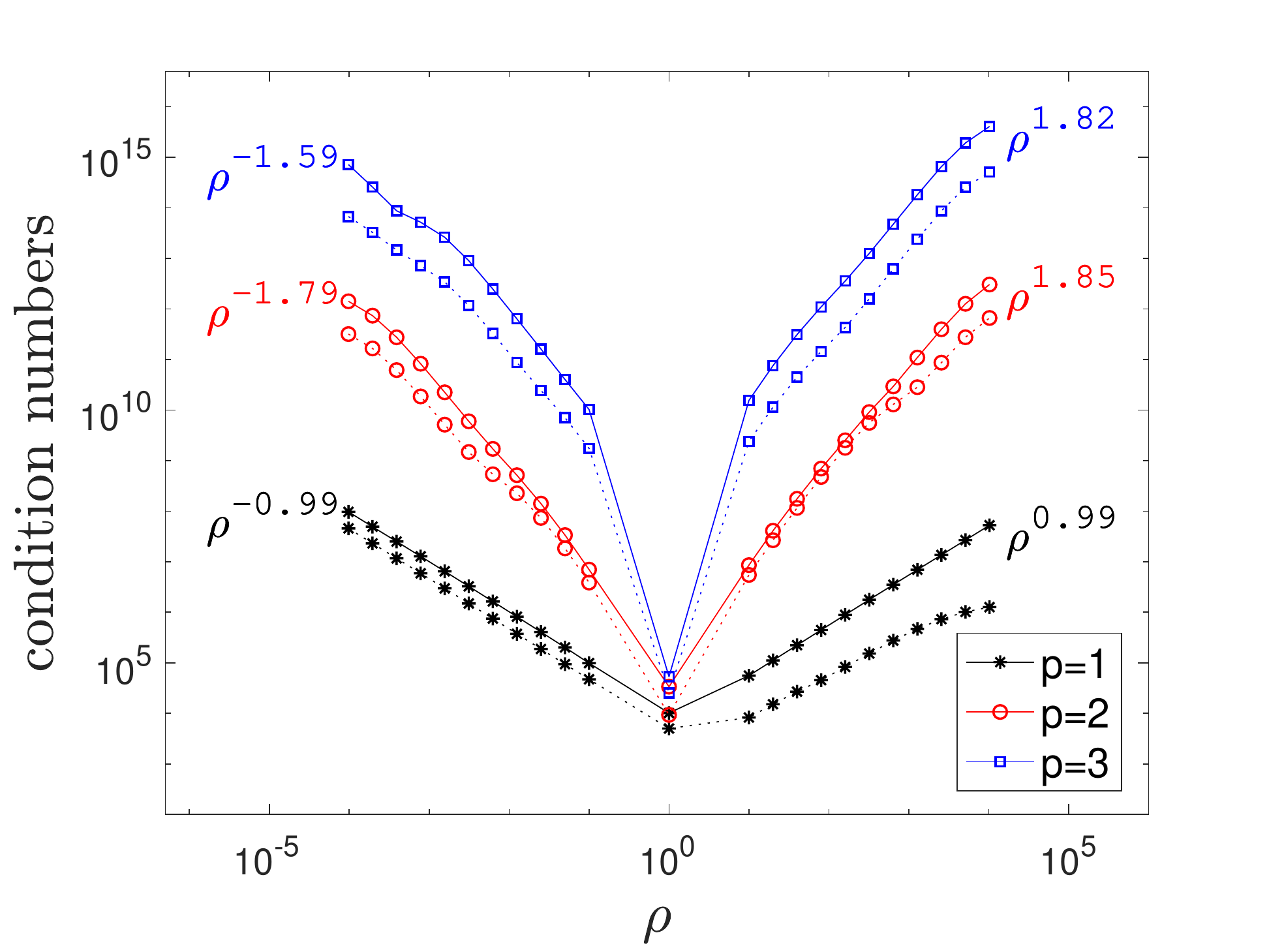}
  \caption{{Condition numbers $\kappa(\mathbf{K}_h)$ (solid line) and $\kappa_S(\mathbf{K}_h)$ (dotted line) versus the contrast $\rho$ for $p=1,\ 2,\ 3$.}}
  \label{fig:cond_IFE_vs_beta_rpi4}
\end{figure}

In this set of numerical experiments we investigate the dependence of the condition
numbers on both small-cut elements and the contrast $\rho$ for the model problem
 \eqref{model} with the linear interface $y=\delta$ spliting  $\Omega=(-1,1)^2$ into $\Omega^+=\{(x,y)\in \Omega:y<\delta \}$ and $\Omega^-=\{(x,y)\in \Omega:y>\delta \}$.
{We compute $\kappa(\mathbf{K}_h)$ and $\kappa_S(\mathbf{K}_h)$
for $\delta=1/(40*2^l)$, $l=0,1,\ldots,8$, $\beta^-=1/10240$, $1/640$, $1$, $640$, $10240$,
$\beta^+=1$ on a uniform mesh with $h=2/40$ and $p=1$, and we plot the related condition numbers versus $\delta$ in Figures \ref{fig:cond_IFE_vs_dist_h40} and \ref{fig:Dpre_cond_IFE_vs_dist_h40}.
We have also carried out similar experiments for $p=2,3$ with all other parameters kept unchanged.
We observe that, on this mesh, as $\delta\rightarrow0$, the interface gradually approaches
the edges  and vertices of interface elements leading to small-cut elements, but 
both condition numbers stay bounded which is not necessarily true for other methods on unfitted meshes in the literature. 
These numerical results are in full agreement with the stability analysis in Section \ref{sec:conditioning}. 
Hence, the proposed IFE method does not suffer from the presence of small-cut elements, and 
this is an important feature for applications with moving interfaces where small-cut elements might be inevitable. 

\begin{figure}[H]
\centering
\subfigure{
    \includegraphics[width=2.1in]{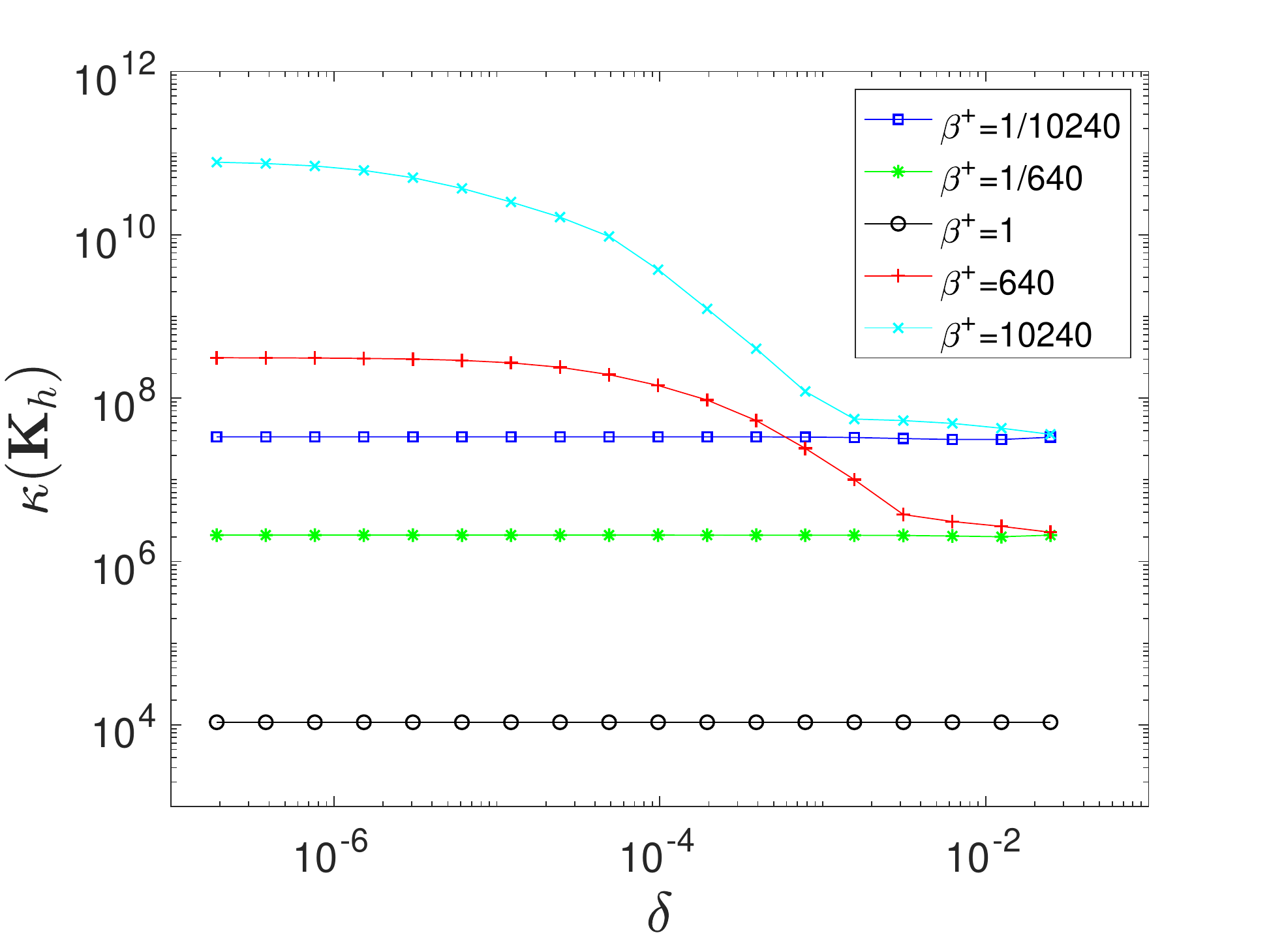}}
\subfigure{
    \includegraphics[width=2in]{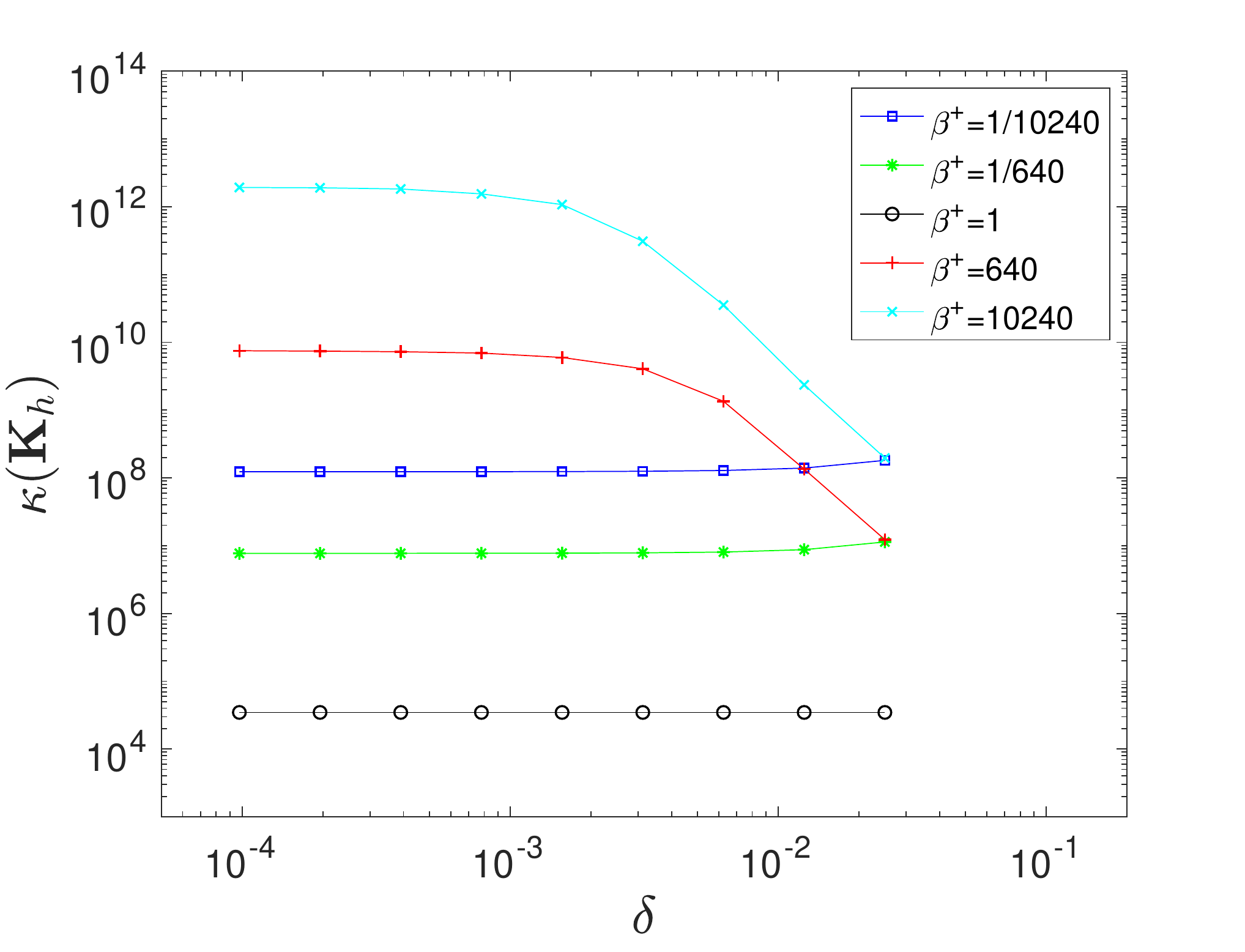}}
 \subfigure{
    \includegraphics[width=2.1in]{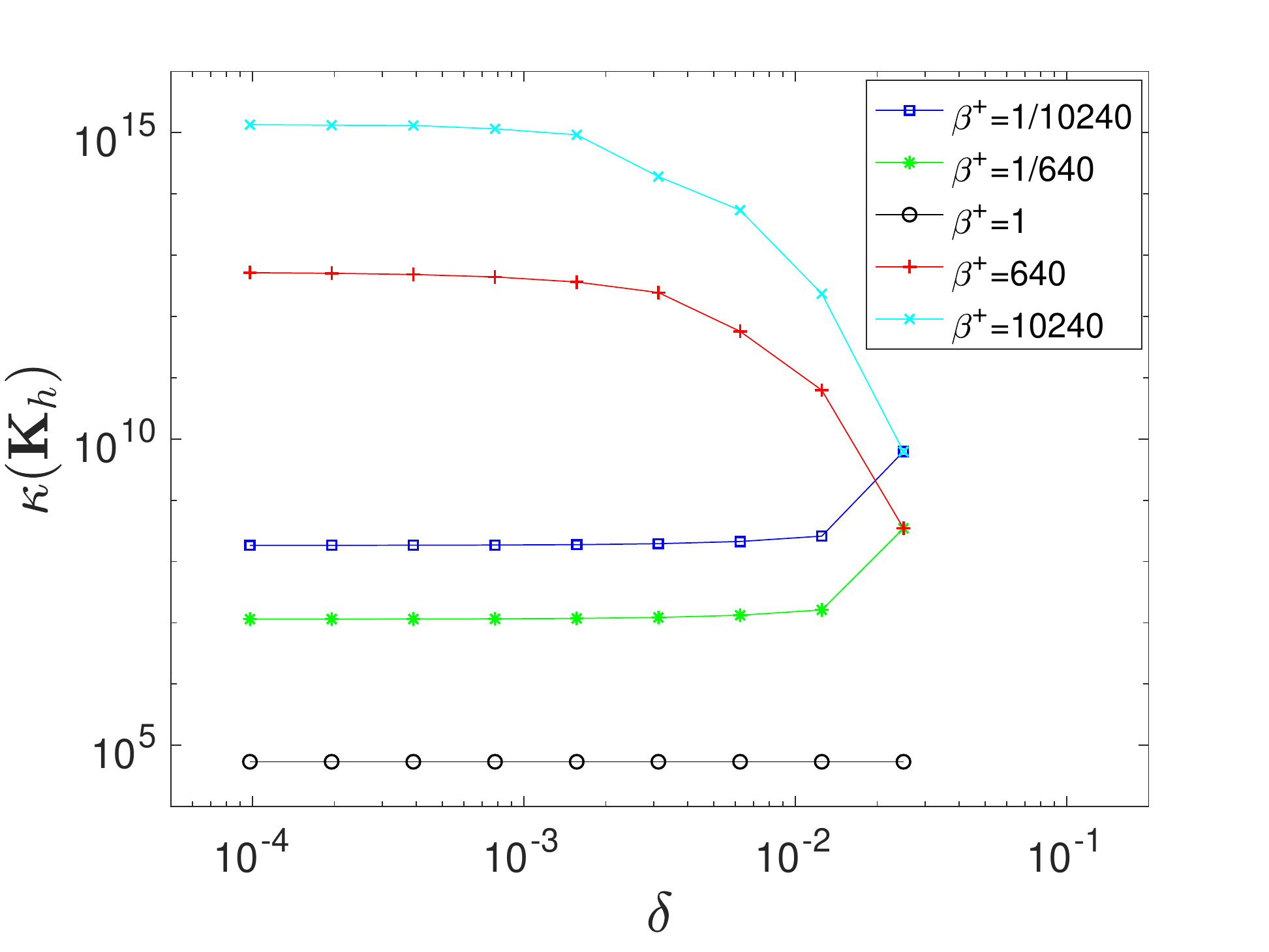}}
  \caption{Condition number $\kappa(\mathbf{K}_h)$ versus $\delta$ ($y=\delta$) for contrast $\rho = 1/10240, 1/640, 1, 640, 10240$ and $p=1,2,3$ (left to right).}
  \label{fig:cond_IFE_vs_dist_h40} 
\end{figure}

\begin{figure}[H]
\centering
\subfigure{
    \includegraphics[width=2.1in]{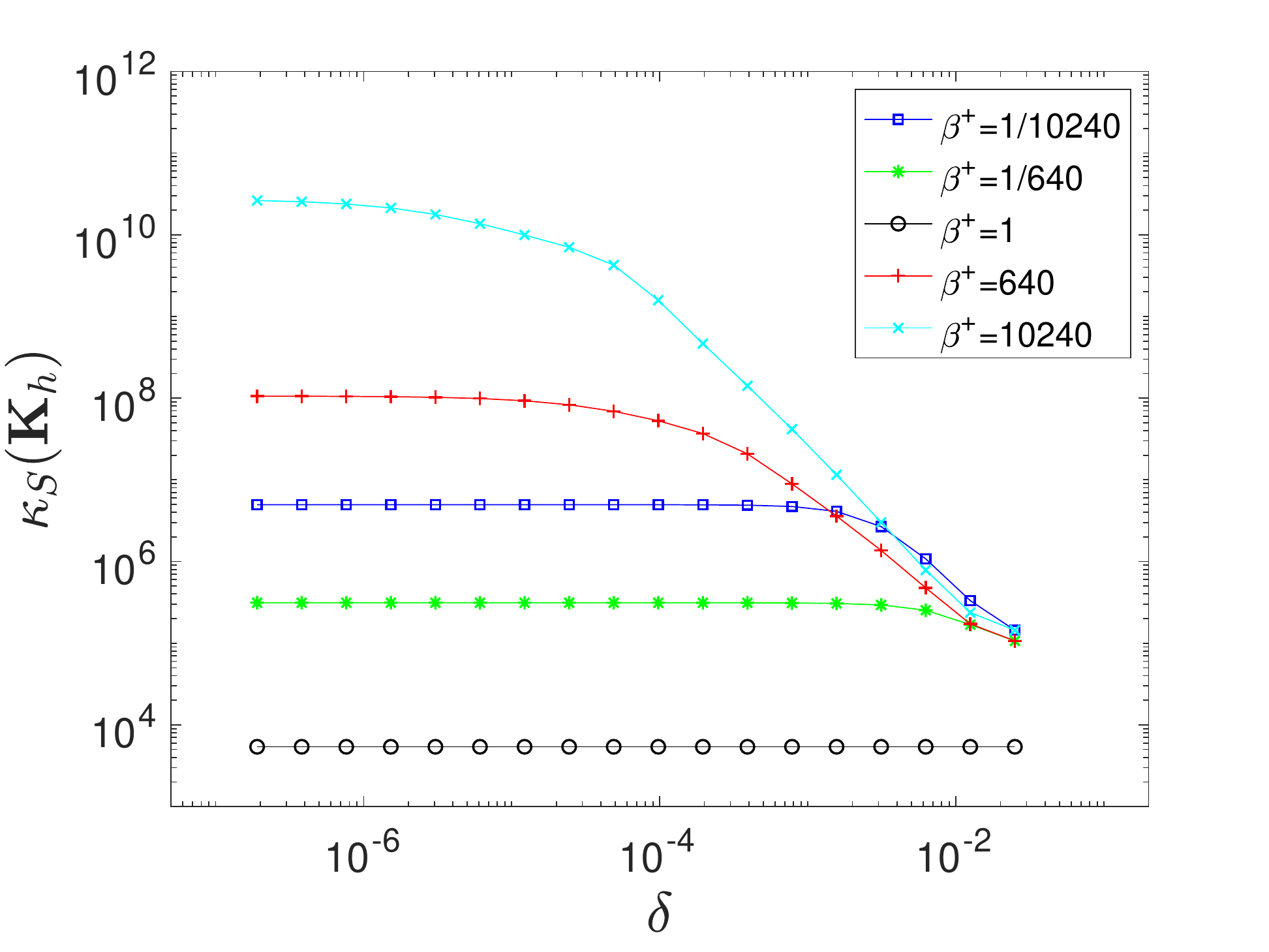}}
\subfigure{
    \includegraphics[width=2in]{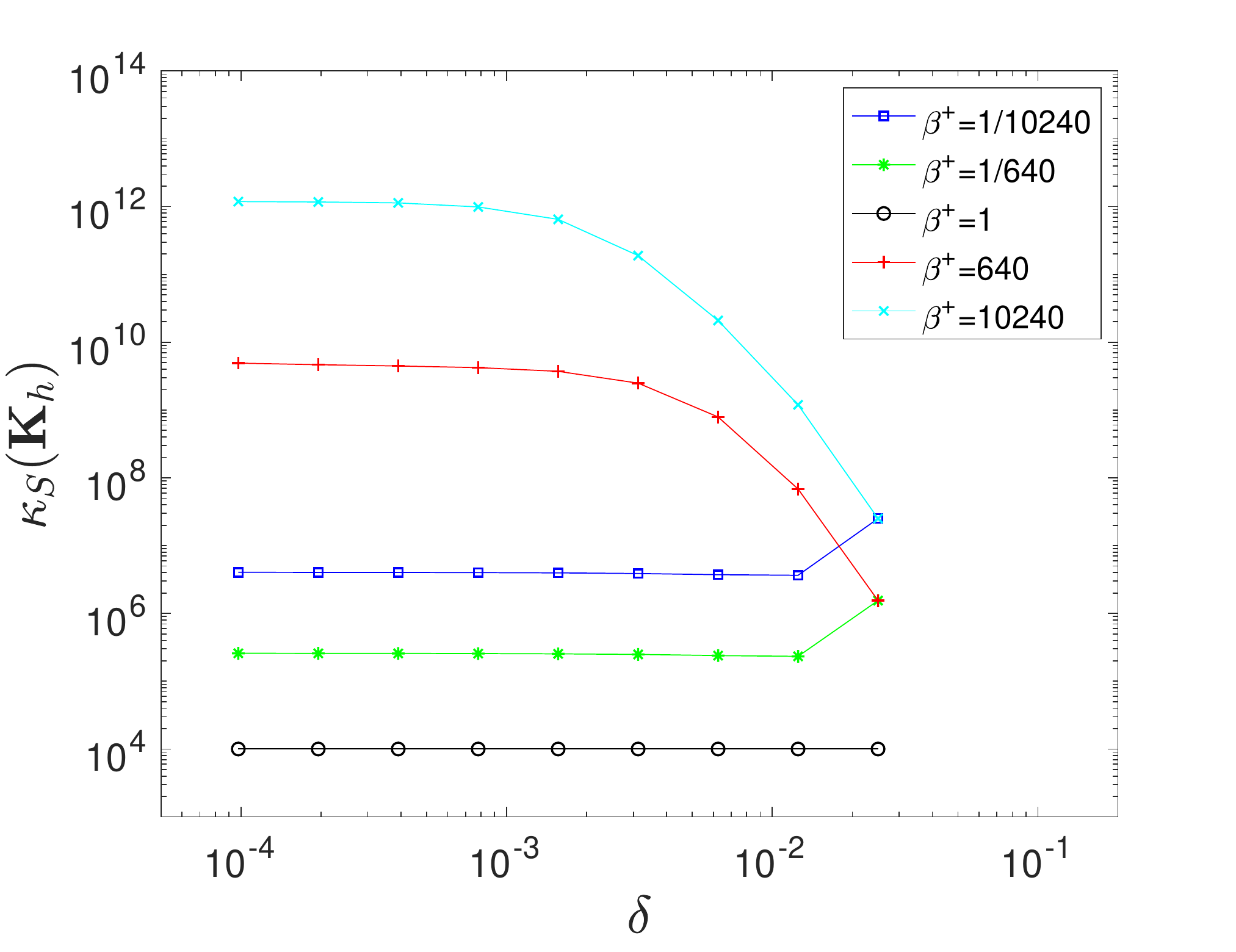}}
 \subfigure{
    \includegraphics[width=2.1in]{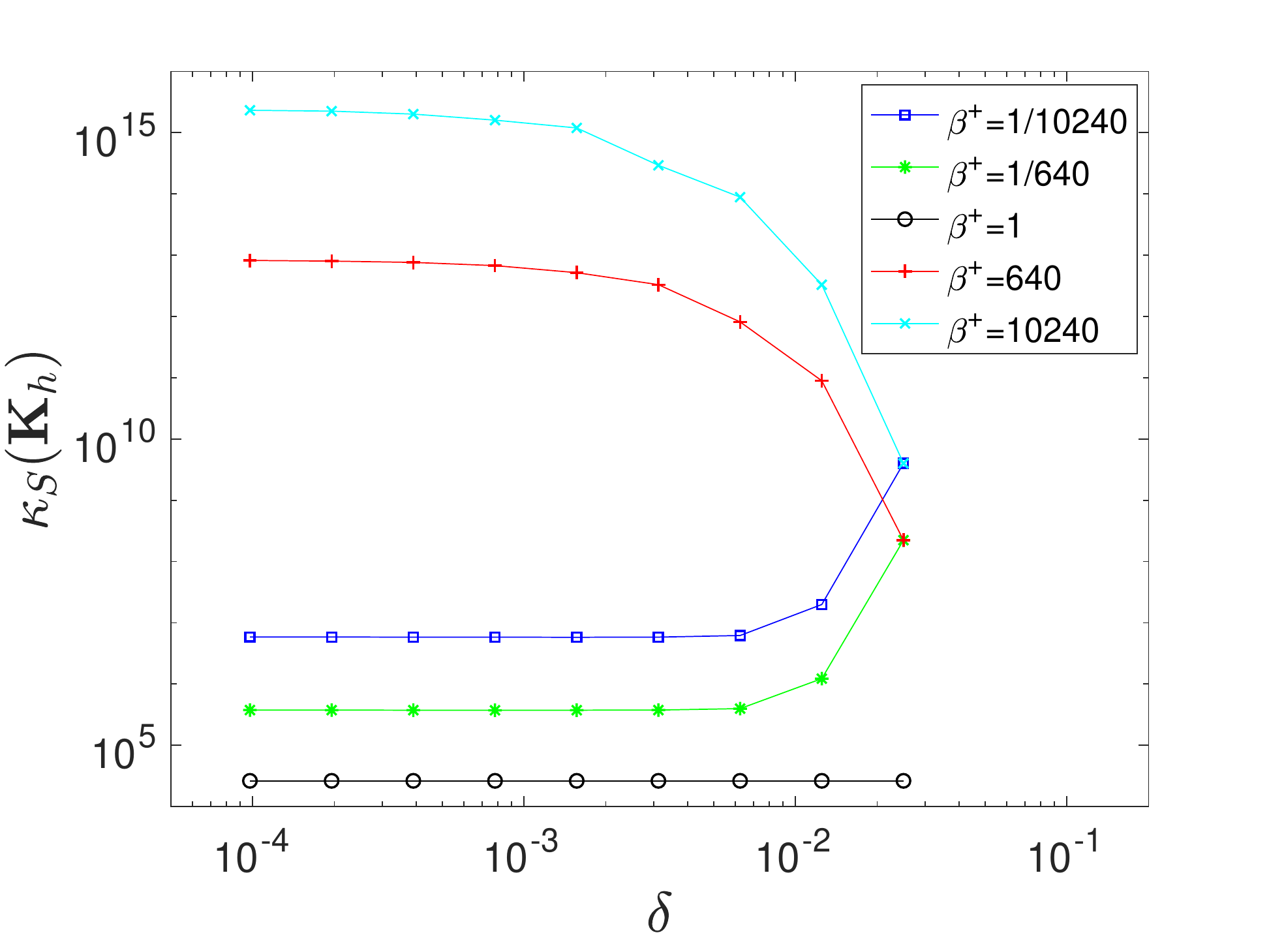}}
  \caption{Scaled condition number $\kappa_S(\mathbf{K}_h)$ versus $\delta$ ($y=\delta$) for contrast $\rho = 1/10240, 1/640, 1, 640, 10240$ and $p=1,2,3$ (left to right).}
  \label{fig:Dpre_cond_IFE_vs_dist_h40} 
\end{figure}

The last set of numerical experiments are conducted to study the spectral condition numbers and scaled spectral condition numbers versus the contrast $\rho$ for both the proposed IFE method on an unfitted mesh and the
standard finite element method on a fitted mesh.
We first construct a fitted mesh formed by a mesh generated on $\Omega^-$
and another mesh generated on $\Omega^+$ by partitioning each subdomain into
$40\times 20$ rectangles where each rectangle is split into two triangles
along its diagonal. The unfitted mesh is obtained by partitioning $\Omega$ using size $h=2/40$.
We set $\delta=1/40$ and $\beta^+=1$ and compute the condition numbers versus $\beta^-= 10^{\textrm{sgn}(i)}\cdot2^i$, $i=-10,\ldots,-1,0,1,\ldots,10$ for three values of $p=1,2,3$.
We plot the evolution of condition numbers versus the contrast $\rho$ for both fitted and unfitted meshes in Figure
\ref{fig:cond_IFE_vs_FE_dist40}. We repeat this experiment with the same parameters for
$\delta=1/640,1/10240$ and show the resuts in Figures \ref{fig:cond_IFE_vs_FE_dist640}
and \ref{fig:cond_IFE_vs_FE_dist10240}. {In these plots, the pink solid line with circles refers to the condition number $\kappa(\mathbf{ K}_h)$ of the IFE method, the pink dotted line with circles refers to the scaled condition number $\kappa_S(\mathbf{ K}_h )$ of the IFE method, while the black solid line with squares refers to the condition numbers of the FE method. Furthermore, on the unfitted mesh with $\delta=1/40$ the interface cuts the interface elements roughly around the middle, while for $\delta\rightarrow 0$ all interface elements
are small-cut elements.

We note that, as expected, the condition number for the standard FE method grows linearly
with respect to the contrast $\rho$.
For $\delta=1/40$, the interface cuts each interface element around the middle with
no small-cut elements, and we observe a linear growth of the condition number $\kappa(\mathbf{ K}_h)$ with respect
to $\rho$ for the IFE method which is comparable with the standard FE method.
However, for $\delta=1/640$ or $\delta=1/10240$ where all interface elements are small-cut elements, 
we observe that the condition number grows quadratically versus $\rho\geqslant1$.
{We also observe that, once the diagonal scaling preconditioner is applied, when $\delta=1/40$, i.e., there is no small-cut subelements, we can clearly observe a slower growth of $\kappa_S(\mathbf{ K}_h) $ versus $\rho$ for the IFE method. But as $\delta\rightarrow 0$, i.e., the subelements become smaller, the growth behavior of $\kappa(\mathbf{ K}_h)$ and $\kappa_S(\mathbf{ K}_h)$ become almost the same. In particular, when $\rho>1$, the curves almost overlap with each other.

More specifically, as $\rho$ approaches $0$, the coefficient $\beta^+$ (on the smaller part of
small-cut elements) is greater than $\beta^-$ (on the larger part of small-cut elements)
leading to a linear growth of $\kappa(\mathbf{K}_h)$ versus $\rho$.
On the other hand when $\rho\geqslant1$ approaches $\infty$, the coefficient $\beta^+$ (on
the smaller part of small-cut elements) becomes smaller than $\beta^-$ (on the larger
part of small-cut elements), and $\kappa(\mathbf{K}_h)$ grows quadratically.
This behavior becomes more pronounced with increasing degree $p$ as shown in Figures
\ref{fig:cond_IFE_vs_FE_dist640} and \ref{fig:cond_IFE_vs_FE_dist10240}.
A similar behavior is observed with the interface $x + y = \delta$ for $\delta \to 0$
leading to small-cut elements near the diagonal.
Furthermore, the superlinear growth rates of $\kappa(\mathbf{K}_h)$ observed in  Figure
\ref{fig:cond_IFE_vs_beta_rpi4} for a circular interface may be caused by the presence
of small-cut elements. {We believe this is also the reason why $\kappa(\mathbf{K}_h)$/$\kappa_S(\mathbf{ K}_h)$ behave differently for $\rho<1$ and $\rho>1$.}
Finally, even though the quadratic growth of $\kappa(\mathbf{K}_h)$ with respect to $\rho$ for the proposed
IFE method is faster than the linear
growth of the standard FE method, it much slower than the theoretical bound derived in Section 5.

\begin{figure}[H]
\centering
\subfigure{
    \includegraphics[width=2.1in]{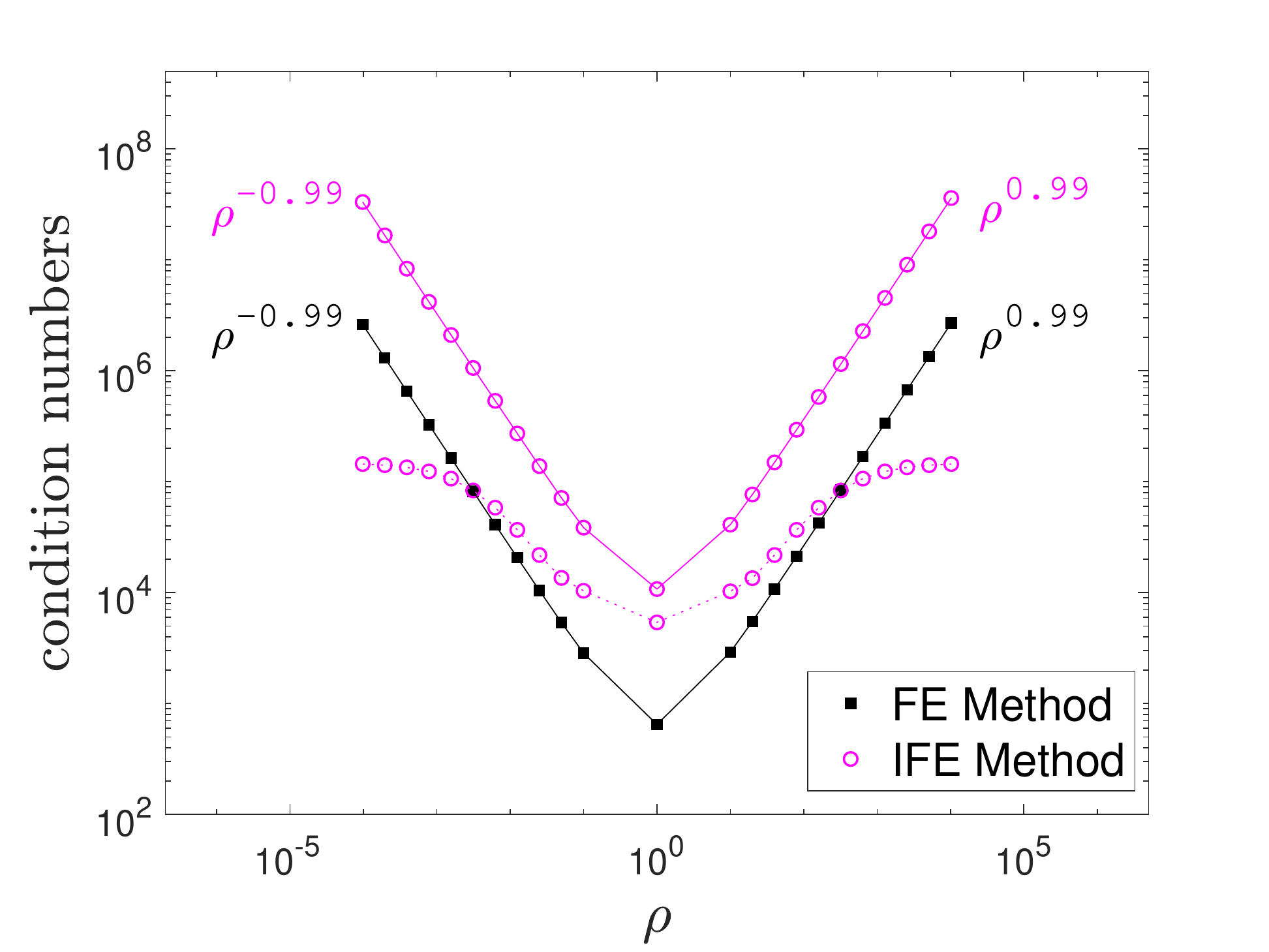}}
\subfigure{
    \includegraphics[width=2.1in]{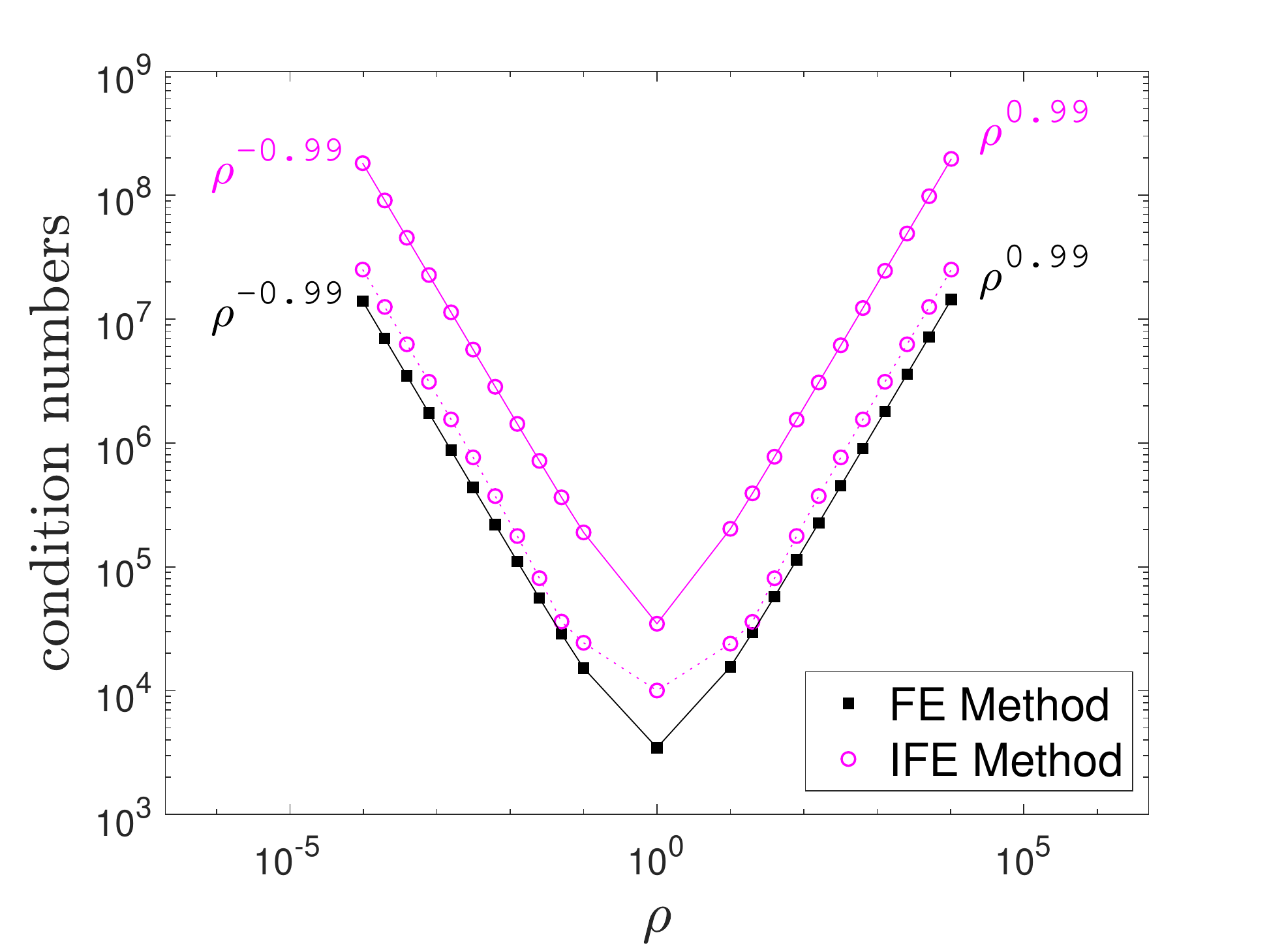}}
\subfigure{
    \includegraphics[width=2.1in]{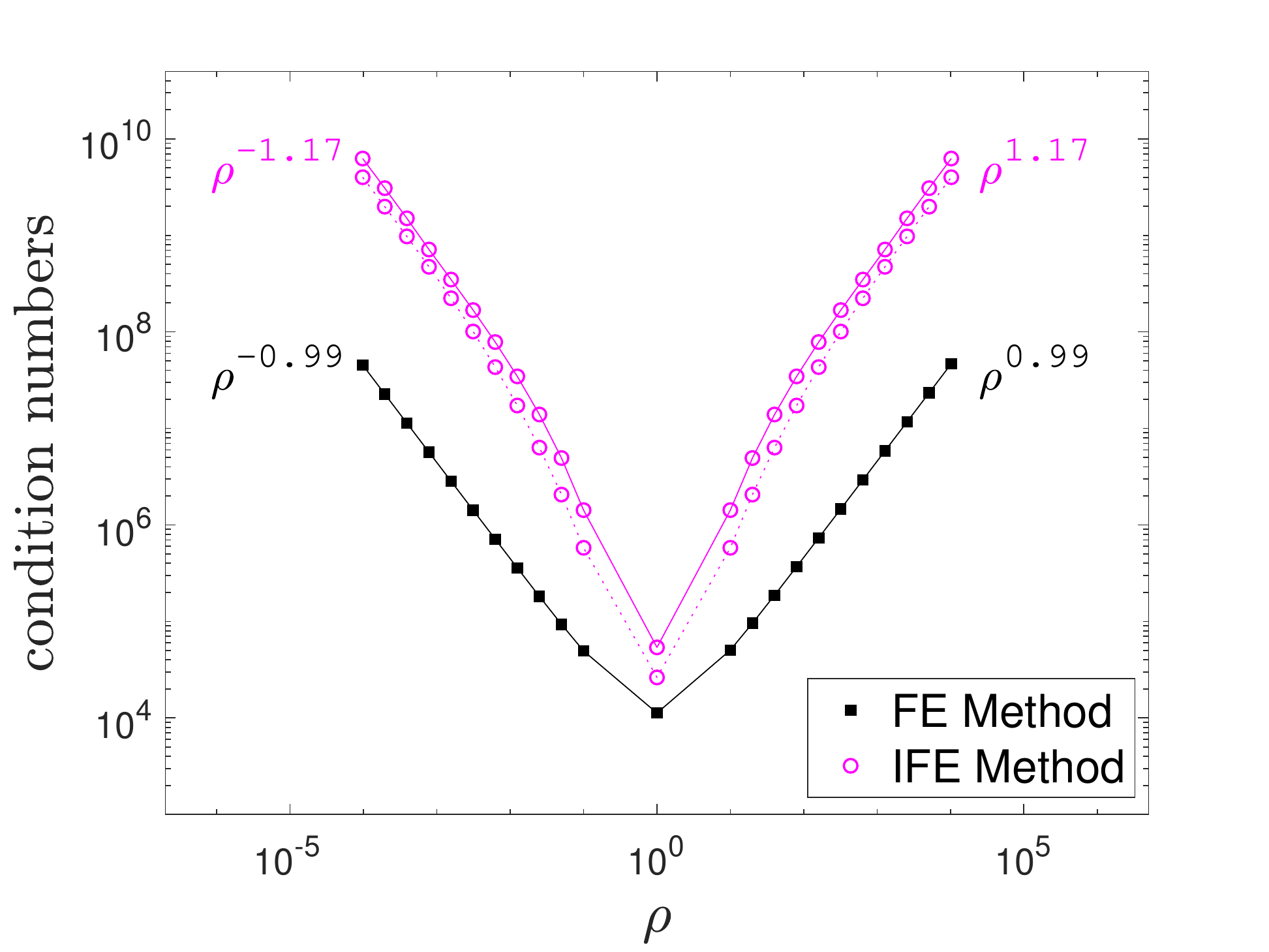}}
  \caption{Condition number versus $\rho$ with $\delta=1/40$ for $p=1,2,3$ (left to right).}
  \label{fig:cond_IFE_vs_FE_dist40}
\end{figure}

\begin{figure}[H]
\centering
\subfigure{
    \includegraphics[width=2.1in]{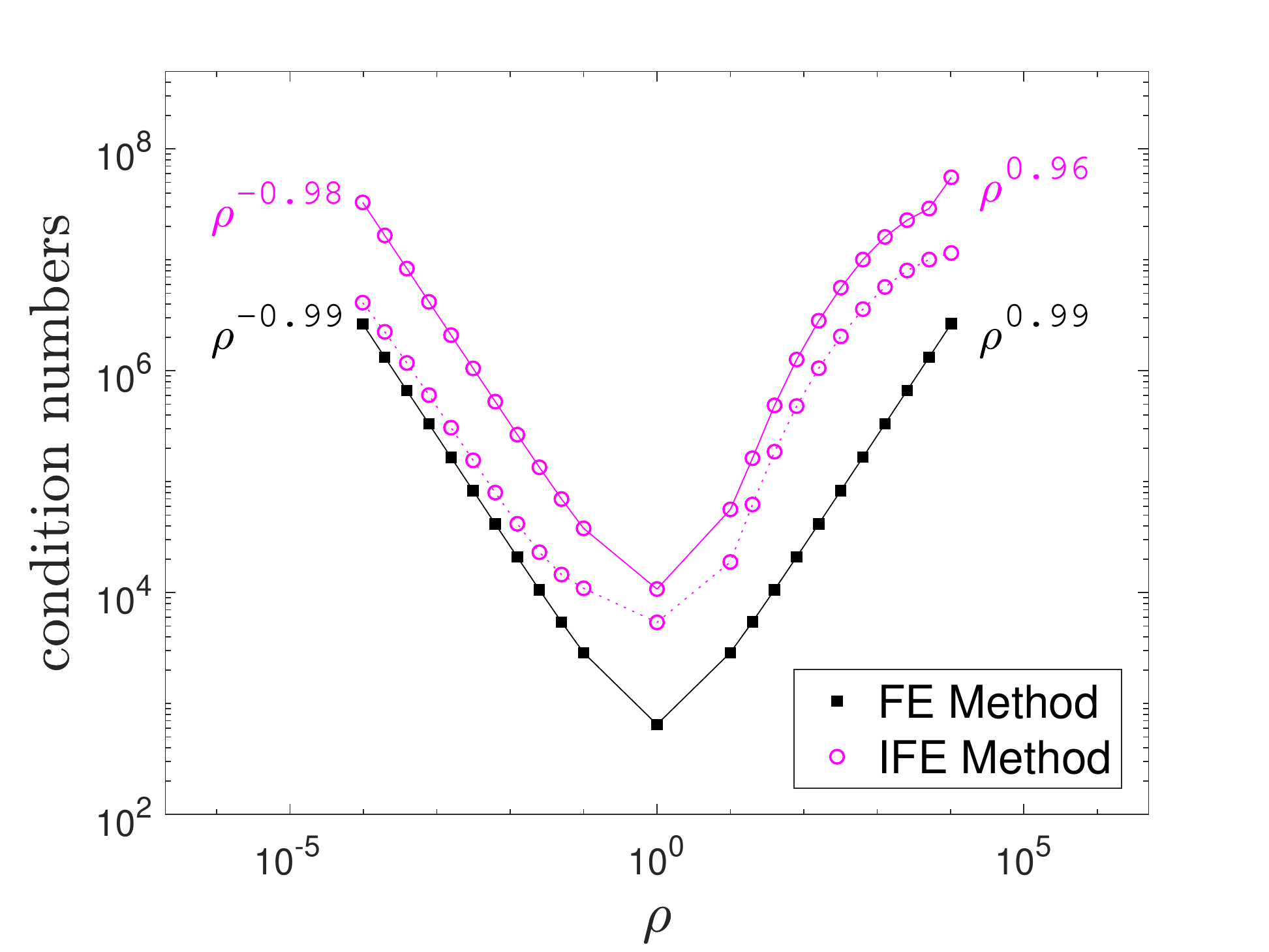}}
\subfigure{
    \includegraphics[width=2.1in]{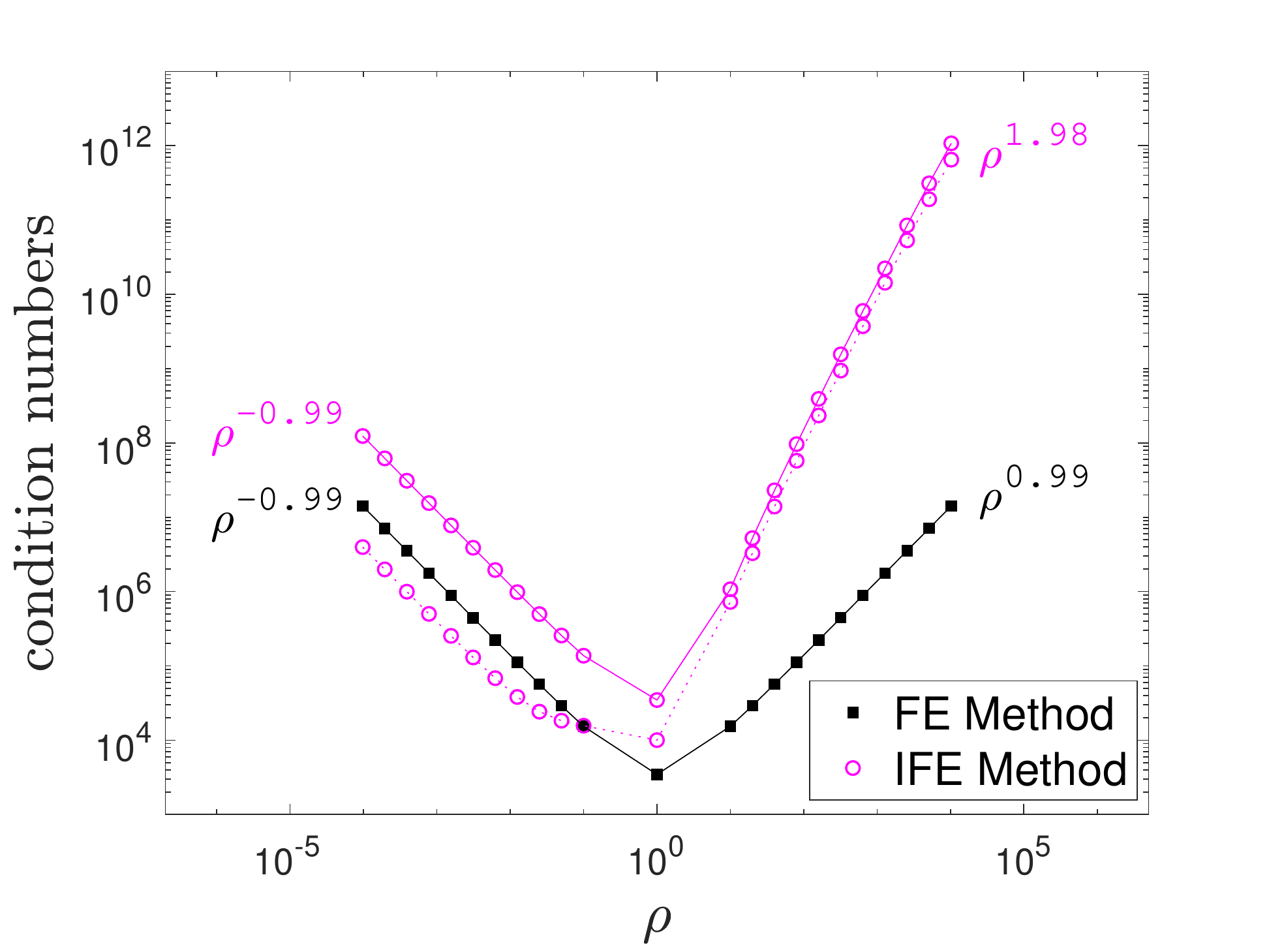}}
\subfigure{
    \includegraphics[width=2.1in]{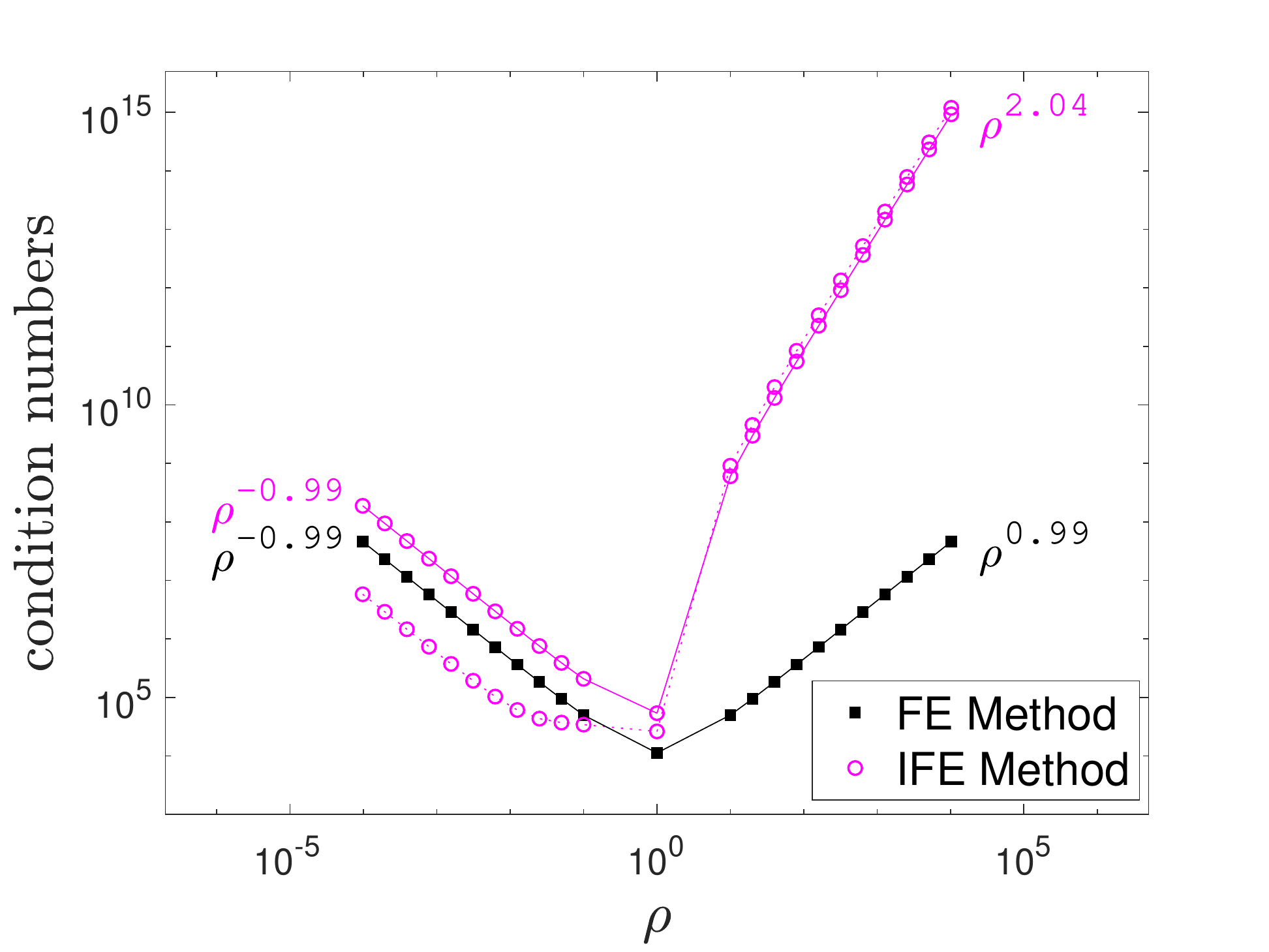}}
  \caption{Condition number versus $\rho$ with $\delta=1/640$
  for $p=1,2,3$ (left to right).}
  \label{fig:cond_IFE_vs_FE_dist640}
\end{figure}

\begin{figure}[H]
\centering
\subfigure{
    \includegraphics[width=2.1in]{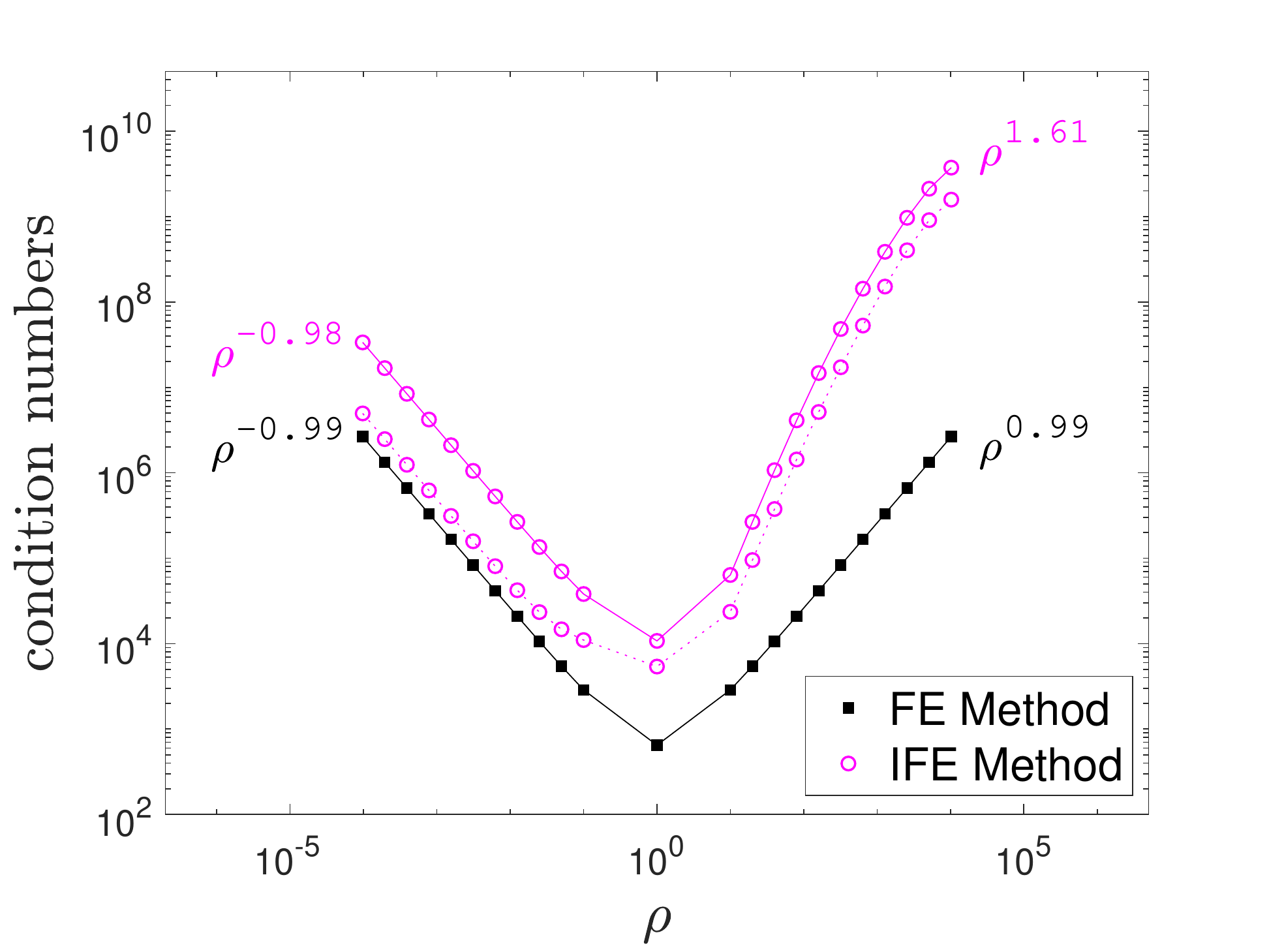}}
\subfigure{
    \includegraphics[width=2.1in]{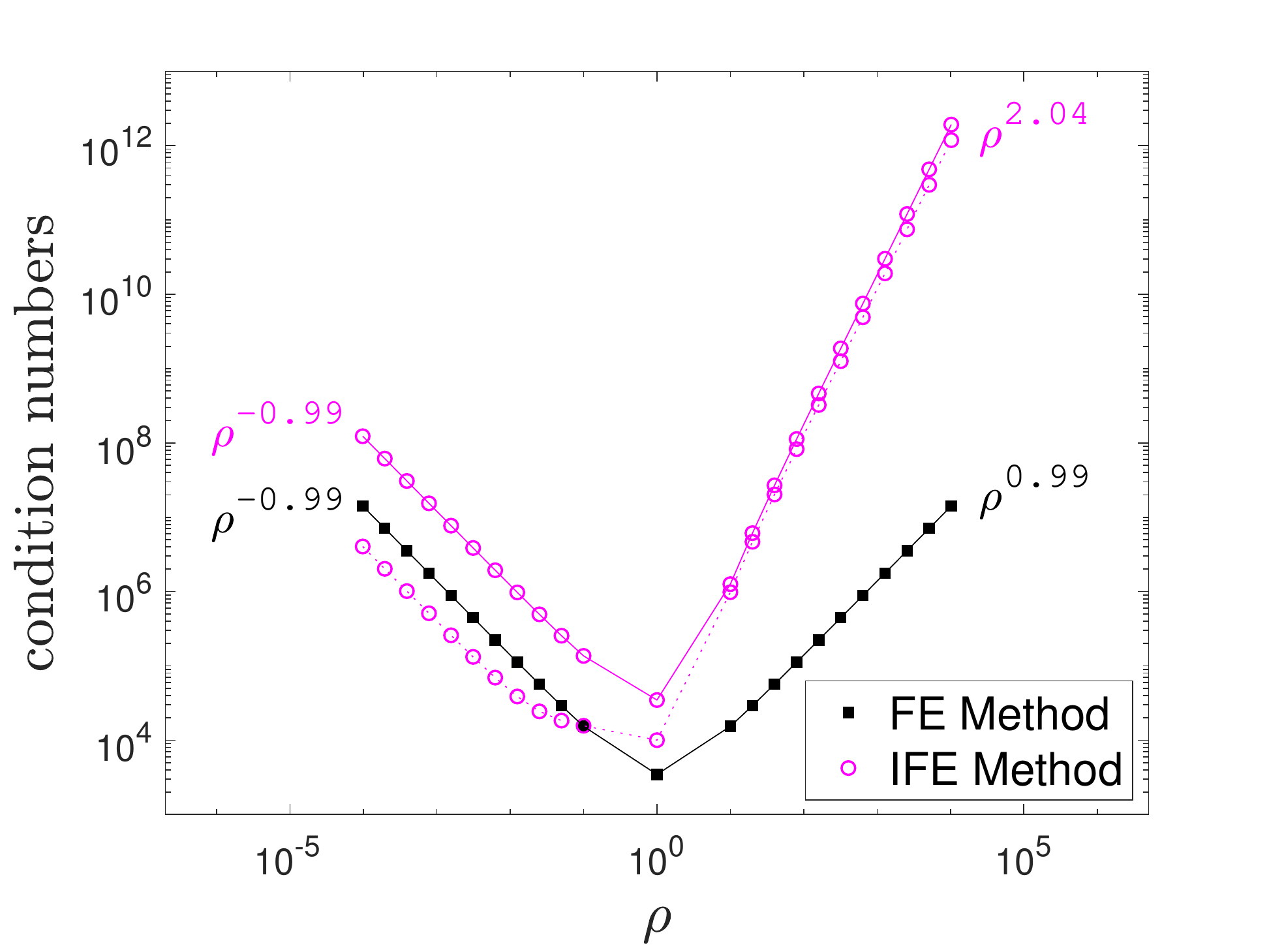}}
\subfigure{
    \includegraphics[width=2.1in]{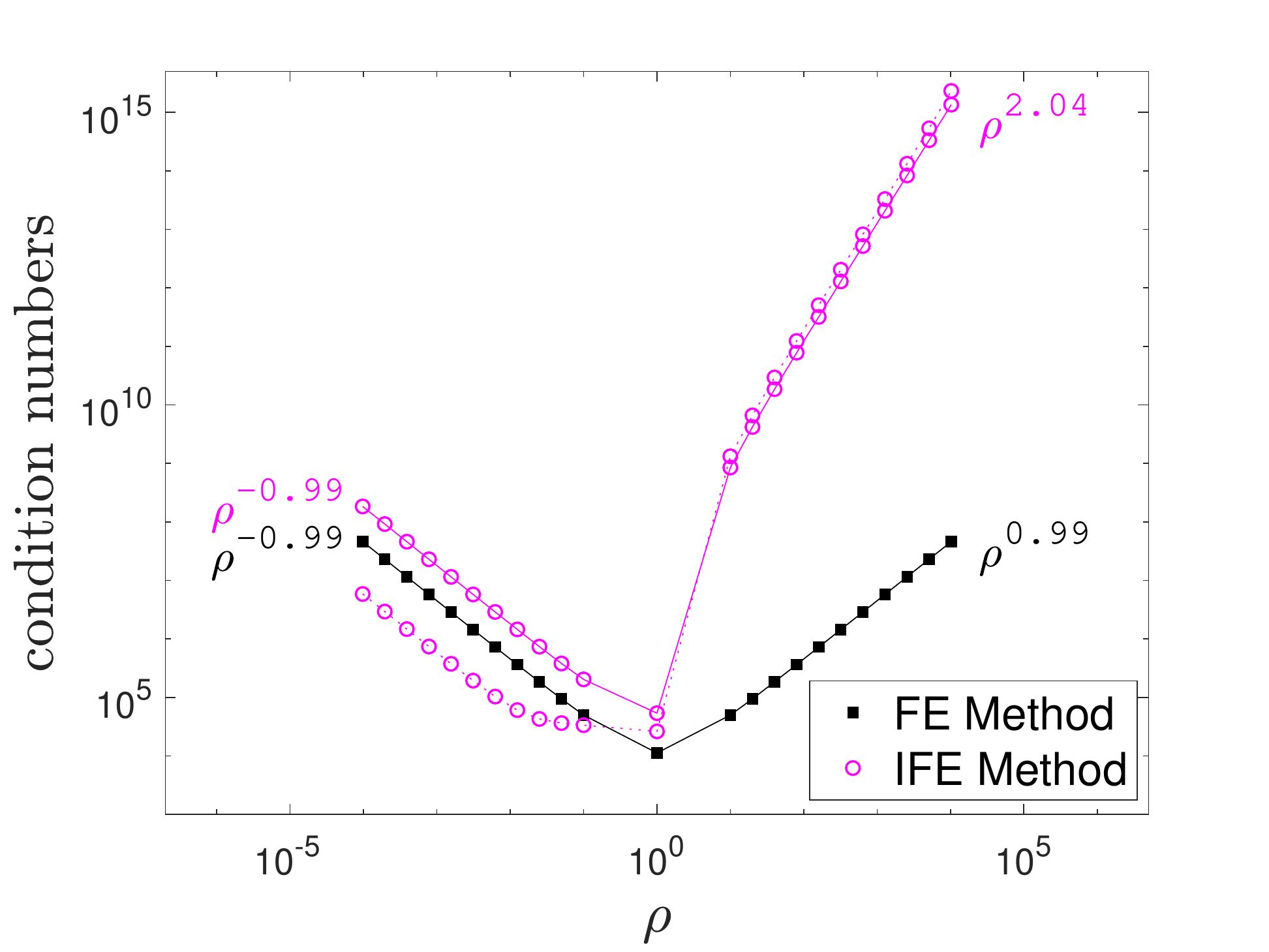}}
  \caption{Condition number versus $\rho$ with $\delta=1/10240$
  for $p=1,2,3$ (left to right).}
  \label{fig:cond_IFE_vs_FE_dist10240} 
\end{figure}

\noindent\textbf{Example 4}.

In \cite{2018BabuskaSoderlind}, it is observed that the accumulation of round-off errors with the size of discrete problems
may be affected by many factors such as the order of the PDE,
the boundary conditions, the discretization methods, as well as algorithms used to solve the resulting algebraic systems.
Moreover, there are several sources of round-off errors such as errors in computing
the stiffness matrix, right-hand side and rounding during the solution of the algebraic systems. It is also shown, through computations, that for many problems
these accumulation rates are not directly linked to the condition numbers.

In this example we investigate the accumulation of round-off error for the IFE method
applied to two-dimensional second-order elliptic interface problems.
For this purpose, we define the relative error
\begin{equation}
\label{loa_1}
\eta = \frac{\| \bfu_h - \hat{\bfu}_h \|_2}{\|\bfu_h\|_2}
\end{equation}
where $\mathbf{ u}_h$ and $\hat{\bfu}_h$, respectively, are the exact and computed solutions
of the linear system \eqref{uh_solu}, and $\hat{\bfu}_h$ is obtained by a typical direct method. 
According to \cite{2012BabuskaBanerjee}, $\eta$ defined by \eqref{loa_1} is a reliable indicator to measure the loss of accuracy in the computed solutions. To be specific, we consider the interface problem posed in the domain $\Omega = (-1, 1)^2$ with a linear interface $y=\delta=1/640$ whose exact solution is $u(X) = (y-\delta)/\beta^s$ on $\Omega^s$, $s=\pm$. Because the exact solution $u(X)$ to the interface problem is linear, we can use it to directly generate the exact solution $\mathbf{ u}_h$ to IFE system \eqref{uh_solu}.
 
We solve the interface problem with $\beta=(2,1)$ on uniform meshes having size
$h=2/N$, $N=10,20,\ldots,500$ such that all meshes contain interface elements. We
compute the numerical solution $\hat{\bf u}_h$ and $\eta$ for IFE spaces of degree $p=1$ by using Cholesky factorization, Gaussian
elimination without pivoting, and Gaussian elimination with pivoting to solve the unscaled linear system of the IFE method and 
the one scaled by diagonal entries. We present the condition numbers and the round-off errors
in Figures \ref{fig:cond_round_off_p1_N10_500} and \ref{fig:round_off_p1_N10_500}. These data demonstrate that 
the round-off errors in computing the IFE solution accumulate as $O(h^{-2})$ which is also the growth rate of condition number of the IFE method as indicated by the stability analysis in the previous section and previous numerical examples. As expected, the data in Figure \ref{fig:round_off_p1_N10_500} also indicate that a more robust method such as either the Cholesky factorization or Gaussian
elimination with pivoting has a better control on the round-off than the method based on the Gaussian elimination without pivoting. 

\begin{figure}[H]
\centering
    \includegraphics[width=6in]{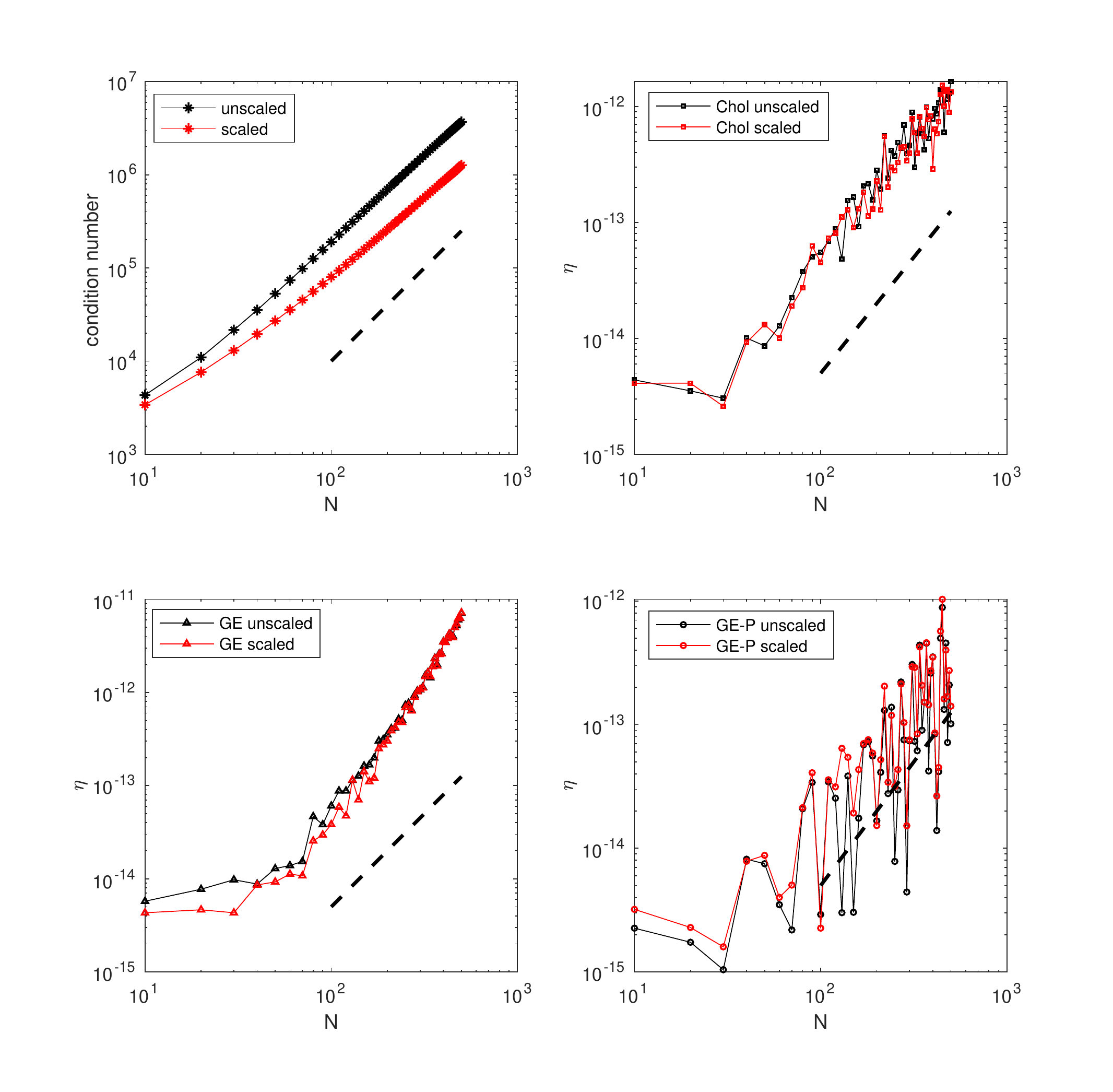}
  \caption{ Condition numbers versus $N$ (top left) and round-off errors versus $N$
   for the IFE method. The dashed line is a reference line having slope 2.}
  \label{fig:cond_round_off_p1_N10_500}
\end{figure}


\begin{figure}[H]
\centering
\subfigure{
    \includegraphics[width=2.5in]{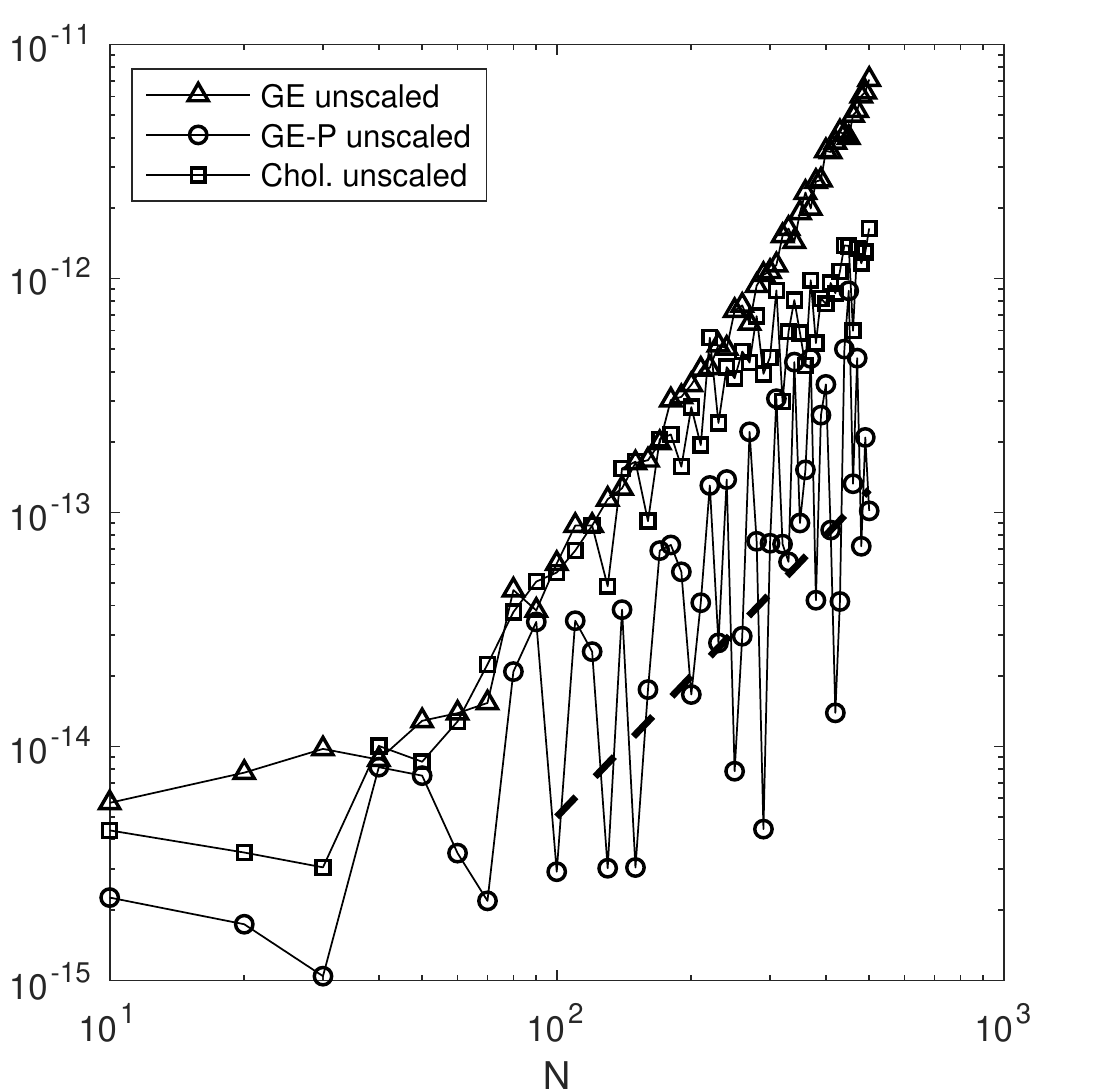}}
\subfigure{
    \includegraphics[width=2.5in]{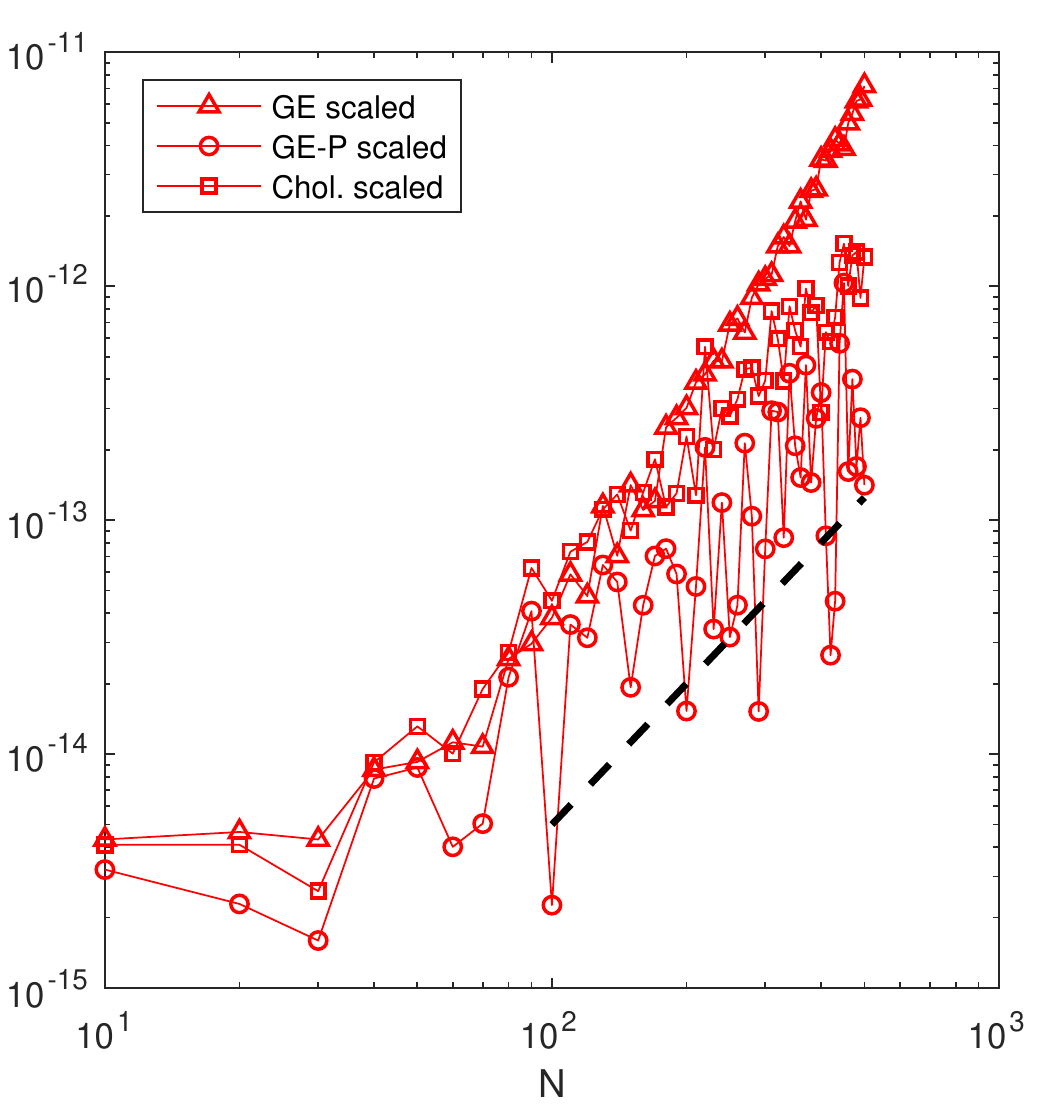}}
  \caption{Round-off errors versus $N$ for the IFE method for unscaled (left) and
  scaled (right) problem. The dashed line is a reference line having slope 2.}
  \label{fig:round_off_p1_N10_500} 
\end{figure}

\section{Conclusions} \label{sec:conclusions}
We have presented a DG immersed finite element method to solve second-order elliptic
interface problems with nonhomogeneous interface jumps and a discontinuous source term.
We have established error estimates for the proposed method in both $L^2$ and energy norms. We have also derived stability
estimates which show that the condition numbers of the stiffness matrix in the proposed method grow like $h^{-2}$
with mesh refinement. However, the constants in these estimates depend on the contrast in a
suboptimal manner as demonstrated by presented numerical examples. 
Future work includes derivation of error bounds and stability results that are optimal with respect to the contrast
as well as extending our error and stability analyses to other existing immersed finite
element methods.


\begin{appendices}
  \section{Proof of Theorem \ref{thm_regularity}}
  \label{append_regularity}

Following \cite{1998ChenZou}, the key idea is to construct two functions $u_1=(u^+_1,u^-_1)$ and $u_2=(u^+_2,u^-_2)$ to homogenize the jump conditions such that the regularity result \cite{1970Babuska} for interface problems with homogeneous jump conditions can be applied. Without loss of generality, we assume $\Omega^+$ is the subdomain inside $\Gamma$ and $\Omega^-$ is the outside subdomain. First of all, we simply let $u^+_1=0$ in $\Omega^+$, and let  $u^-_1$ satisfy the following biharmonic problem:
\begin{equation}
\begin{split}
\label{append_regularity_3}
&-\triangle^2 u^-_1 =0 ~~~~ \text{in} ~~ \Omega^-, \\
& u^-_1 =  0, ~~~~~~~  \beta^- \frac{\partial u^-_1}{\partial \mathbf{ n}} = J_N ~~~~ \text{on} ~~ \Gamma,  \\
& u^-_1 = 0,~~~~~~  \frac{\partial u^-_1}{\partial \mathbf{ n}}=0 ~~~~~~~~~~~ \text{on} ~~ \partial \Omega.
\end{split}
\end{equation}
By the classical results in \cite{2011GiraultRaviart}, $u^-_1$ uniquely exists and satisfies the stability bound in terms of the boundary data:
\begin{equation}
\label{append_regularity_4}
\| u^-_1 \|_{H^{m+2}(\Omega^-)}  \leqslant \frac{C}{\beta^-} \| J_{N} \|_{H^{m+1/2}(\Gamma)}.
\end{equation}
We note that $u_1=(u_1^+,u_1^-)$ satisfies $[u]_\Gamma=0$ and $[\frac{\beta \partial u_1}{\partial {\bf n}} ]_\Gamma=J_N $ and is used to homogenize the flux jump condition. As for the discontinuous jump condition, we let $u^+_2=0$ and $u^-_2$ satisfy the following biharmonic equation:
\begin{equation}
\begin{split}
\label{append_regularity_7}
&-\triangle^2 u^-_2 =0 ~~~~ \text{in} ~~ \Omega^-, \\
& u^-_2 = J_D, ~~~~  \frac{\partial u^-_2}{\partial \mathbf{ n}}= 0 ~~~~~~~~~ \text{on} ~~ \Gamma,  \\
& u^-_2 = 0, ~~~~~~  \frac{\partial u^-_2}{\partial \mathbf{ n}}=0 ~~~~~~~~~ \text{on} ~~ \partial \Omega.
\end{split}
\end{equation}
Using the regularity for biharmonic equations \cite{2011GiraultRaviart}, we have
\begin{equation}
\label{append_regularity_8}
\| u^-_2 \|_{H^{m+2}(\Omega^-)}   \leqslant C \| J_{D} \|_{H^{m+3/2}(\Gamma)}.
\end{equation}
Here $u_2 = (u_2^+,u_2^-)$ satisfies $[u_2]_{\Gamma}=J_D$ and $[\beta \nabla u_2\cdot\mathbf{ n}]_{\Gamma}=0$. Now we define $\tilde{u} := u- u_1-u_2$, i.e., $u = \tilde{u} + u_1 +u_2$, which satisfies
\begin{equation}
\begin{split}
\label{append_regularity_9}
& -\beta \triangle \tilde{u} = f + \beta \triangle u_1 + \beta \triangle u_2 ~~~~~~ \text{in} ~~ \Omega, \\
& [\tilde{u}]_{\Gamma} = 0, ~~~~~~ [\beta \nabla \tilde{u} \cdot \mathbf{ n}]_{\Gamma}=0, ~~~~~~ \text{on} ~~ \Gamma, \\
& \tilde{u} = 0, ~~~~~~~~~~~~~~~~~~~~~~~~~~~~~~~~~~~~ \text{on} ~~ \partial\Omega.
\end{split}
\end{equation}
By the classical results of interface problems with homogeneous jump conditions \cite{1970Babuska,2010ChuGrahamHou,2015GuzmanSanchezSarkisP1,2002HuangZou}, we have
\begin{equation}
\label{append_regularity_10}
\sum_{k=1}^{m+2} | \beta \tilde{u} |_{PH^{k}(\Omega)} \leqslant C \left( \| f \|_{PH^m(\Omega)} + \| \beta \triangle u_1 \|_{PH^m(\Omega)} + \| \beta \triangle u_2 \|_{PH^m(\Omega)} \right).
\end{equation}
Since $u^+_1=0$ and using \eqref{append_regularity_4}, we have
\begin{equation}
\label{append_regularity_11}
\| \beta \triangle u_1 \|_{PH^m(\Omega)} \leqslant \| \beta^- u^-_1 \|_{H^{m+2}(\Omega^-)} \leqslant C \| J_{N} \|_{H^{m+1/2}(\Gamma)}.
\end{equation}
Similarly, since $u^+_2=0$ and using \eqref{append_regularity_8}, we obtain
\begin{equation}
\label{append_regularity_12}
\| \beta \triangle u_2 \|_{PH^m(\Omega)} \leqslant \| \beta^- u^-_2 \|_{H^{m+2}(\Omega^-)} \leqslant C \beta^- \| J_{D} \|_{H^{m+1/2}(\Gamma)}.
\end{equation}
Note that the biharmonic equation \eqref{append_regularity_7} can be also posed on $\Omega^+$ and let $u^-_2=0$ in $\Omega^-$. Therefore, combining \eqref{append_regularity_11}, \eqref{append_regularity_12} and \eqref{append_regularity_10}, we actually have
\begin{equation}
\label{append_regularity_13}
\sum_{k=1}^{m+2} | \beta \tilde{u} |_{PH^{k}(\Omega)} \leqslant C \left( \| f \|_{PH^m(\Omega)} + \min\{\beta^-,\beta^+\} \| J_D \|_{H^{m+3/2}(\Omega)} + \| J_N \|_{H^{m+1/2}(\Omega)} \right)
\end{equation}
Substituting this estimate into $u=\tilde{u}+u_1+u_2$, we have the desired result.


\section{Decomposition of the Solution}
\label{append_decomposition}
Here we split the solution $u$ of the interface problem \eqref{model} into four functions ${u}_0, ~ u_D, ~ u_N, ~ u_f$ as:

\begin{equation} \label{append_p1}
-\beta \triangle {u_0} =\mathfrak{E}^+(f) ~~~~ \text{in} ~~ \Omega,~ [{u}_0]_\Gamma = 0, ~  [\beta \frac{\partial {u_0}}{\partial \mathbf{ n}}]_\Gamma=0, ~ {u}_0|_{\partial \Omega} = 0.
\end{equation}

\begin{equation} \label{append_p2}
-\beta \triangle {u_D} =0 ~~~~ \text{in} ~~ \Omega,~ [u_D]_\Gamma = J_D, ~  [\beta \frac{\partial u_D}{\partial \mathbf{ n}}]_\Gamma= 0, ~ {u_D} |_{\partial \Omega} = 0.
\end{equation}

\begin{equation} \label{append_p3}
-\beta \triangle {u_N} =0 ~~~~ \text{in} ~~ \Omega,~ [u_N]_\Gamma = 0, ~  [\beta \frac{\partial u_N}{\partial \mathbf{ n}}]_\Gamma= J_N, ~ {u_N} |_{\partial \Omega} = 0.
\end{equation}
If $f_j=(0, f^- - \mathfrak{E}^+({f})) $, $u_f$ is the solution of the interface problem
\begin{equation}
\label{append_p4}
-\beta \triangle {u_f} = f_j ~~~~ \text{in} ~~ \Omega,~ [u_f]_\Gamma = 0, ~  [\beta \frac{\partial u_f}{\partial \mathbf{ n}}]_\Gamma= 0, ~ u_f |_{\partial \Omega} = 0.
\end{equation}
By a direct computation one can check that $u = \tilde{u} + u_D + u_N + u_f$.

In order to decompose $Q_Tu$ in \eqref{J_interpolation_2} as the sum of functions satisfying homogeneous interface
jumps and nonhomogeneous jumps, we note that on an interface element $T$ and the associated fictitious element $T_\lambda = T_\lambda^+ \cup T_\lambda^-$, for each of the solutions $u_k = (u_k^+,u_k^-)$, $k=D,N,f$, $u_k^+$ can be approximated by its $L^2$-projection $z_p^k \in\mathbb{P}_p(T_\lambda), ~k=D,N,f$ on $T_\lambda^+$, namely
 $\int_{T_\lambda^+} (z_p^k - u_k^+)v_p dx =0$, $\forall  v_p \in \mathbb{P}_p(T_\lambda^+)$. We also need the following functions $\theta_k^-$, $\tilde{\theta}_k^-$  in $\mathbb{P}_p(T_\lambda)$ for $k=D,N,f$ satisfying the discrete problems
\begin{equation}
a_\lambda(\theta_D^-,v_p) = b_\lambda(z_p^D,v_p) + h^{-3} \int_{\Gamma_T^\lambda} J_D v_p ds, ~~ \forall v_p \in \mathbb{P}_p(T_\lambda^-),
\end{equation}
\begin{equation}
a_\lambda(\theta_N^-,v_p) = b_\lambda(z_p^N,v_p) + h^{-1} \int_{\Gamma_T^\lambda} \frac{J_N}{\beta^-} v_p ds, ~~ \forall v_p \in \mathbb{P}_p(T_\lambda^-),
\end{equation}
\begin{equation}
a_\lambda(\theta_f^-,v_p) = b_\lambda(z_p^f,v_p) + \int_{T_\lambda^-} \frac{f_j}{\beta^-} v_p dX, ~~ \forall v_p \in \mathbb{P}_p(T_\lambda^-),
\end{equation}
\begin{equation}
a_\lambda(\tilde{\theta}_k^-,v_p) = b_\lambda(z_p^k,v_p), ~~ \forall v_p \in \mathbb{P}_p(T_\lambda^-),~~ \text{for} ~k=D,N,f.
\end{equation}
On an interface element $T$ if $\theta_k =(z_p^k,\theta_k^-)$ and
 $\tilde{\theta}_k = (z_p^k,\tilde{\theta}_k^-)$, then $\theta_k - \tilde{\theta}_k = (0, \theta_k^- - \tilde{\theta}_k^-)=\phi_{T,k}$, for $k=D,N,f$ given in \eqref{enrich_IFE_fun_DN} and \eqref{enrich_IFE_fun_Lap}. This procedure actually motivates an approximation to the decomposition of $u$ given in \eqref{cauchy_exten_sum}. In order to motivate that the global approximation of $u_k$ can be decomposed into a function satisfying homogeneous jump conditions and a function satisfying the inhomogeneous jumps, we subtract the following approximations
 \begin{equation}
Q_hu_k =
\begin{cases}
I_T u_k, & \text{on} ~T \in \mathcal{T}^n_h \\
\theta_k, & \text{on} ~T \in \mathcal{T}^i_h
\end{cases}
,~~~
\tilde{Q}_hu_k =
\begin{cases}
I_T u_k, & \text{on} ~T \in \mathcal{T}^n_h \\
\tilde{\theta}_k, & \text{on} ~T \in \mathcal{T}^i_h
\end{cases}
 \end{equation}
where $I_T$ denotes the standard Lagrange interpolation, to write
\begin{equation}
\Phi_{h,k} = Q_hu_k- \tilde{Q}_hu_k =
\begin{cases}
0, & on ~T \in \mathcal{T}^n_h \\
\phi_{T,k} & on ~T \in \mathcal{T}^i_h
\end{cases},
\end{equation}
which, in turn, yields
\begin{equation}
Q_hu_k = \Phi_{h,k} + \tilde{Q}_hu_k, ~~~ k=D,N,f,
\end{equation}
with $\phi_{T,k}, ~D,~N,~f$ given in \eqref{enrich_IFE_fun_DN} and\eqref{enrich_IFE_fun_Lap} and
$\tilde{Q}_{h} u_k \in S_h^p(\Omega)$. Since $u= {u_0}+u_D + u_N + u_f$  we use Theorem 5.2 in \cite{2019GuoLin} to approximate
${u}_0 \in V_{h,0}$ by
${u}_{0,h} \in S^p_h(\Omega)$ and apply $Q_h$ to $u_D + u_N + u_f$ to write
\begin{eqnarray}
u &\approx& \tilde{u}_{0,h} + Q_hu_D + Q_hu_N + Q_hu_f =  u_{0,h} + \Phi_h ,
\end{eqnarray}
where $\Phi_h = \Phi_{h,D} + \Phi_{h,N} + \Phi_{h,f}$ and
$u_{0,h} = \tilde{u}_{0,h} + \tilde{Q}_hu_D + \tilde{Q}_hu_N + \tilde{Q}_hu_f \in S_h^p(\Omega)$. Thus, the solution $u$ of \eqref{model} can be approximated by $\tilde{u}_h + \Phi_h \in S_h^p(\Omega) \bigoplus \{ \Phi_h \}$.

\section{Local Cauchy Problem}
\label{append_cauchy}
\begin{thm} \label{appen:thm1}
Let $T$ and $T_\lambda$, respectively, be an interface element and its associated fictitious element cut by a linear interface $\Gamma_T^\lambda$. Then, for every given $z_p\in \mathbb{P}_p(T_ \lambda)$ there exists a unique $v_p\in \mathbb{P}_p(T_ \lambda)$ such that
\begin{equation}\label{appen:eq:interface}
(z_p - v_p)|_{\Gamma_{T}^\lambda}=0, ~(\beta^+ \partial_{\mathbf n} z_p - \beta^- \partial_\mathbf{n} v_p)|_{\Gamma_{T}^\lambda}=0, ~
(\beta^+ \partial_{\mathbf n}^l \triangle z_p - \beta^- \partial_\mathbf{n}^l \triangle v_p)|_{\Gamma_{T}^\lambda},~~ l=0,1,\ldots,p-2.
\end{equation}
 Furthermore, $v_p$ is a solution of the Cauchy problem on $T_\lambda^-$
\begin{equation}\label{appen:eq:cauchy}
  \triangle v_p = \frac{\beta^+}{\beta^-} \triangle z_p, ~~in ~T_\lambda^-, ~~ v_p = {z_p}, ~~
 \partial_{\mathbf{n}} v_p = \frac{\beta^+}{\beta^-}\partial_{\mathbf{n}} {z_p} ~on~ \Gamma_T^\lambda.
\end{equation}
\end{thm}
\begin{proof}
First, we establish (\ref{appen:eq:cauchy}) by considering
$g^+= -\beta^+ \triangle z_p$ and  $g^- = -\beta^- \triangle v_p$ in
$\mathbb{P}_{p-2}(T_\lambda)$. We note that $d = g^+-g^- \in \mathbb{P}_{p-2}(T_\lambda$)
and its tangential derivatives are zero on $\Gamma_T^\lambda$ (see the arguments of Lemma 4.1 in \cite{2019GuoLin}). Moreover, the remaining extended interface conditions $\partial_\mathbf{n}^l d |_{\Gamma_T^\lambda}= 0, ~~l=0,1,\ldots,p-2$ yield
$\partial_\mathbf{n}^j \partial_{\boldsymbol{\tau}}^{j} d |_{\Gamma_T^\lambda} =0$, $i,j \ge 0$ leading to $d =0$ on $T_\lambda$, \emph{i.,e.}, $g^- = g^+$ on $T_\lambda^-$.
This combined with the first two interface conditions in (\ref{appen:eq:interface}) establishes (\ref{appen:eq:cauchy}).

In order to show the existence and uniqueness  of $v_p$ (for  given $z_p$) we consider the two operators
$$
\mathfrak{L}^\pm: \mathbb{P}_p(T_\lambda) \rightarrow \mathbb{W}_p = \mathcal{P}_p \bigotimes  \mathcal{P}_{p-1} \bigotimes \cdots \bigotimes \mathcal{P}_0,
$$
defined by
$$
\mathfrak{L}^\pm(v_p) =(v_p,\beta^\pm \partial_\mathbf{n} v_p,\beta^\pm \triangle v_p,
\beta^\pm \partial_\mathbf{n} \triangle v_p, \ldots, \beta^\pm \partial_\mathbf{n}^{p-2} \triangle v_p) |_{\Gamma_T^\lambda},
$$
where $\mathcal{P}_k(\Gamma_T^\lambda)$ is the subspace of polynomials of degree not excedding $k$ on $\Gamma_T^\lambda$.
Hence, the interface conditions (\ref{appen:eq:interface}) can be written as $\mathfrak{L}^-(v_p) = \mathfrak{L}^+(z_p)$. Since $dim(\mathbb{P}_p) = dim(\mathbb{W}_p)$ it suffices to show uniqueness by
observing that $\mathfrak{L}^- (v_p)=0$ yields $\partial_\mathbf{n}^i \partial_{\boldsymbol{\tau}}^j v_p|_{\Gamma_T^\lambda}=0,~~i,j \ge 0$ which, in turn,  shows that $v_p=0$ on $T_\lambda$.
This concludes the proof.
\end{proof}

%
\begin{rem}
$\mathfrak{L}^\pm$ are isomorphims from  $\mathbb{P}_p(T_\lambda)$ to
$\mathbb{R}^{dim(\mathbb{P}_p)}$.
\end{rem}

\end{appendices}

\%bibliographystyle{plain}

%

\end{document}